\documentclass[11pt,a4paper,reqno]{amsart}

\usepackage[utf8]{inputenc}
\usepackage[T1]{fontenc}
\usepackage[margin=3cm]{geometry}
\usepackage{amsmath}
\usepackage{amsthm}
\usepackage{amssymb}
\usepackage{amsopn}
\usepackage{enumitem}
\usepackage{xcolor}
\usepackage{mathtools}
\usepackage{hyperref}
\usepackage{doi}
\usepackage{xfrac}
\usepackage{graphicx}
\usepackage{bbm}
\usepackage{relsize}
\usepackage{subcaption}

\makeatletter
\def\paragraph{\@startsection{paragraph}{4}%
  \z@\z@{-\fontdimen2\font}%
  {\normalfont\bfseries}}
\makeatother

\theoremstyle{plain}
\newtheorem{theorem}{Theorem}[section]
\newtheorem{prop}[theorem]{Proposition}

\newtheorem{lemma}[theorem]{Lemma} 

\theoremstyle{definition}
\newtheorem{definition}[theorem]{Definition}

\theoremstyle{remark}
\newtheorem{remark}[theorem]{Remark}

\numberwithin{equation}{section}

\DeclareMathOperator{\divo}{Div}

\DeclareMathOperator{\dist}{dist}

\DeclareMathOperator{\supp}{supp}

\DeclareFontFamily{OT1}{pzc}{}
\DeclareFontShape{OT1}{pzc}{m}{it}{<-> s * [1.40] pzcmi7t}{}
\DeclareMathAlphabet{\mathpzc}{OT1}{pzc}{m}{it}

\newcommand{\Dm}[1]{D_{1#1}}
\newcommand{\Ds}[1]{D_{2#1}}

\newcommand{\R}{\mathbb{R}}

\newcommand{\Rc}{\ensuremath{\mathcal{R}}}

\renewcommand{\Mc}{\ensuremath{\mathcal{M}}}
\newcommand{\uc}{\ensuremath{\boldsymbol{\mathpzc{u}}}}
\newcommand{\Dc}{\ensuremath{\mathpzc{D}}}

\newcommand{\Hcc}{\ensuremath{\dot{\mathcal{H}}^1}}
\newcommand{\La}{\Lambda}

\newcommand{\Z}{\mathbb{Z}}

\newcommand{\N}{\mathbb{N}}

\usepackage{color}
\newcommand{\cb}{}

\definecolor{fgreen}{HTML}{2ECC40}

\newcommand{\RFgL} % rate function gradient case in (1+1) with interval [0,1] - LHS
{
       \Sigma^{\nabla}(h) 
}
\newcommand{\RFgR}[1] % rate function gradient case in (1+1) with interval [0,1] - RHS
{
       \frac{1}{2} \int_0^1 \dot{h}(t)^2 \, dt - #1|\{t \in [0,1] : h(t) = 0 \}|
}

\newcommand{\RFlL} % rate function laplacian case in (1+1) with interval [0,1] - LHS
{
       \Sigma^{\Delta}(h)
}

\newcommand{\RFlR}[1] % rate function laplacian case in (1+1) with interval [0,1] - RHS
{
       \frac{1}{2} \int_0^1 \ddot{h}(t)^2 \, dt - #1|\{t \in [0,1] : h(t) = 0 \}|
}

\newcommand{\<}[2]{\ensuremath{\langle #1,\,#2\rangle}}

\newcommand{\E}{\mathcal{E}}

\newcommand{\G}{\mathcal{G}}

\makeatletter
\def\subsubsection{\@startsection{subsubsection}{3}%
  \z@{.5\linespacing\@plus.7\linespacing}{-.5em}%
  {\normalfont\bfseries}}

\begin{document}

\title{Analysis of an atomistic model for anti-plane fracture}

%    Remove any unused author tags.

%    author one information
\author{Maciej Buze}
\author{Thomas Hudson}
\author{Christoph Ortner}
\address{Mathematics Institute\\
  Zeeman Building\\
  University of Warwick\\
  Coventry\\
  CV4 7AL\\
  United Kingdom}
\curraddr{}
\email[M.~Buze]{m.buze@warwick.ac.uk}
\email[T.~Hudson]{t.hudson.1@warwick.ac.uk}
\email[C.~Ortner]{c.ortner@warwick.ac.uk}

\thanks{MB is supported by EPSRC as part of the MASDOC DTC, Grant No. EP/HO23364/1.}
\thanks{TH is supported by the Leverhulme Trust through Early Career Fellowship ECF-2016-526}
\thanks{CO is  supported by ERC Starting Grant 335120 and by EPSRC Grant EP/R043612/1}



\date{31 July 2019}

\dedicatory{}

\begin{abstract}
We develop a model for an anti-plane crack defect posed on a square lattice under an interatomic pair-potential with nearest-neighbour interactions. In particular, we establish existence, local uniqueness and stability of solutions for small loading parameters and further prove qualitatively sharp far-field decay estimates. The latter requires establishing decay estimates for the corresponding lattice Green's function, which are of independent interest.
\end{abstract}
\maketitle
\section{Introduction}
In crystalline solids, various aspects of material behaviour related to mechanical, electrical, and chemical properties are governed by the appearance of irregularities (defects) in their underlying lattice structure \cite{phillips_2001}. Typical crystalline defects include point defects, dislocations and cracks. Crystalline defects are inherently discrete 
objects and to accurately capture their mechanical behaviour it is 
crucial to construct models starting from atomistic principles. 
Establishing their mathematical foundations also enables a rigorous numerical analysis of various 
multi-scale simulation techniques such as \cite{KANZAKI,Sinclair,Mullins,TOP,P-Lin, Blanc2007,2013-atc.acta,2015-qmtb1}.

A general approach to describe a single localised defect embedded in a homogeneous host crystal was rigorously formalised in \cite{2013-disl,EOS2016,2017-bcscrew} for point defects and straight dislocations 
in Bravais lattices, then extended in \cite{2016-multipt} to point defects in multilattices. The overarching idea is to use a continuum model to specify the far-field behaviour where continuum theories such as continuum linear elasticity (CLE) are accurate, while employing the underlying atomistic model to capture the details of the defect core.

So far, however, this framework explicitly excluded cracks due to two challenges that do not arise for point defects and dislocations. Firstly, as is already evident when comparing CLE approaches to modelling screw dislocations and cracks (cf. \cite{hirth-lothe}), the latter involves a slower rate of decay of strain away from the defect core, which makes it more difficult to prove that the corresponding atomistic model is even well-defined. Furthermore, in order to employ an atomistic model in the presence of a crack, one has to consider a domain that is both discrete and inhomogeneous, since the crack breaks translational symmetry. A particularly limiting consequence of this is that before one can establish results about {\it regularity} of the resulting discrete elastic fields, one first has to prove the existence and decay properties of a {\it lattice Green's function, $\G$, in the crack geometry}. While the cases of point defects and dislocations permit a spatially homogeneous setup of the reference configuration, allowing us to obtain the lattice Green's function via the semi-discrete Fourier transform, this approach breaks down in the presence of a crack. 

The main purpose of this paper is the (non-trivial) extension of the theory to the case of an anti-plane crack defect. We overcome the problem of inhomogeneity of the domain and are able to prove existence and decay estimates for $\G$. The approach employed is centred around the observation that the problem of finding $\G$ can itself be cast as an instance of coupling between continuum and atomistic descriptions, as we prescribe the explicit continuum Green's function $G$ as a boundary condition. This construction ensures the existence of $\G$ and is then followed by a technically involved argument establishing the decay properties of $\G$.

To simplify the presentation, we restrict the analysis to a two--dimensional square lattice with nearest neighbour interactions together with an interatomic potential satisfying {\it anti-plane mirror symmetry}, as introduced in \cite{2017-bcscrew} (see Section \ref{conclusion} for more details). In particular, this assumption will ensure that the atomistic model is indeed well-defined. Most of the results readily translate to anti-plane models with finite interactions on a general two--dimensional Bravais lattice (in particular the triangular lattice), except for one technical step related to constructing a suitable locally isomorphic mapping from the defective lattice to a homogeneous one. Furthermore, drawing from the ideas developed in \cite{2017-bcscrew} it is expected one can also extend the results beyond models with mirror symmetry (again see Section \ref{conclusion} for a more detailed discussion). \\

\paragraph{Outline:}
In Section \ref{main} we introduce and describe in detail the model for an anti-plane crack defect together with underlying assumptions and state the main results, including the key construction of the lattice Green's function for the crack geometry $\G$. In Section \ref{setup} we describe the spatial and functional setup of the problem, stressing its discreteness and inhomogeneity. In Section \ref{model} the details of the atomistic model are discussed and the results establishing the existence and regularity of a solution are stated. Section \ref{Green} is dedicated to constructing $\G$ and establishing its decay properties. In Section \ref{numerics} we present a numerical scheme that tests the rate of decay of $\bar{u}$ and the resulting convergence analysis. We conclude with the proofs of the main results in Section \ref{proofs}.  

\section{Main results}\label{main}
\subsection{Discrete kinematics}\label{setup}
Let $\La$ denote the two dimensional square lattice defined as $\La := \{l - (\frac{1}{2},\frac{1}{2})\,|\,l \in \Z^2\}$. With the crack tip placed at the origin, we consider a crack opening along
\[
 \Gamma_0 := \{(x_1,0)\,|\,x_1 \leq 0\}.
\]
It is useful to distinguish the lines that include lattice points directly above and below $\Gamma_0$. These are defined as
\[
\Gamma_{\pm} := \big\{m \in \La\,\big|\, m_1 < 0\,\text{ and }\,m_2 = \pm \textstyle{\frac{1}{2}} \big\}.
\]
We refer to Figure \ref{fig:lattice_gamma} for a visualisation of this setup.
\begin{figure}[!htbp]
\centering
\includegraphics[width=0.65\linewidth]{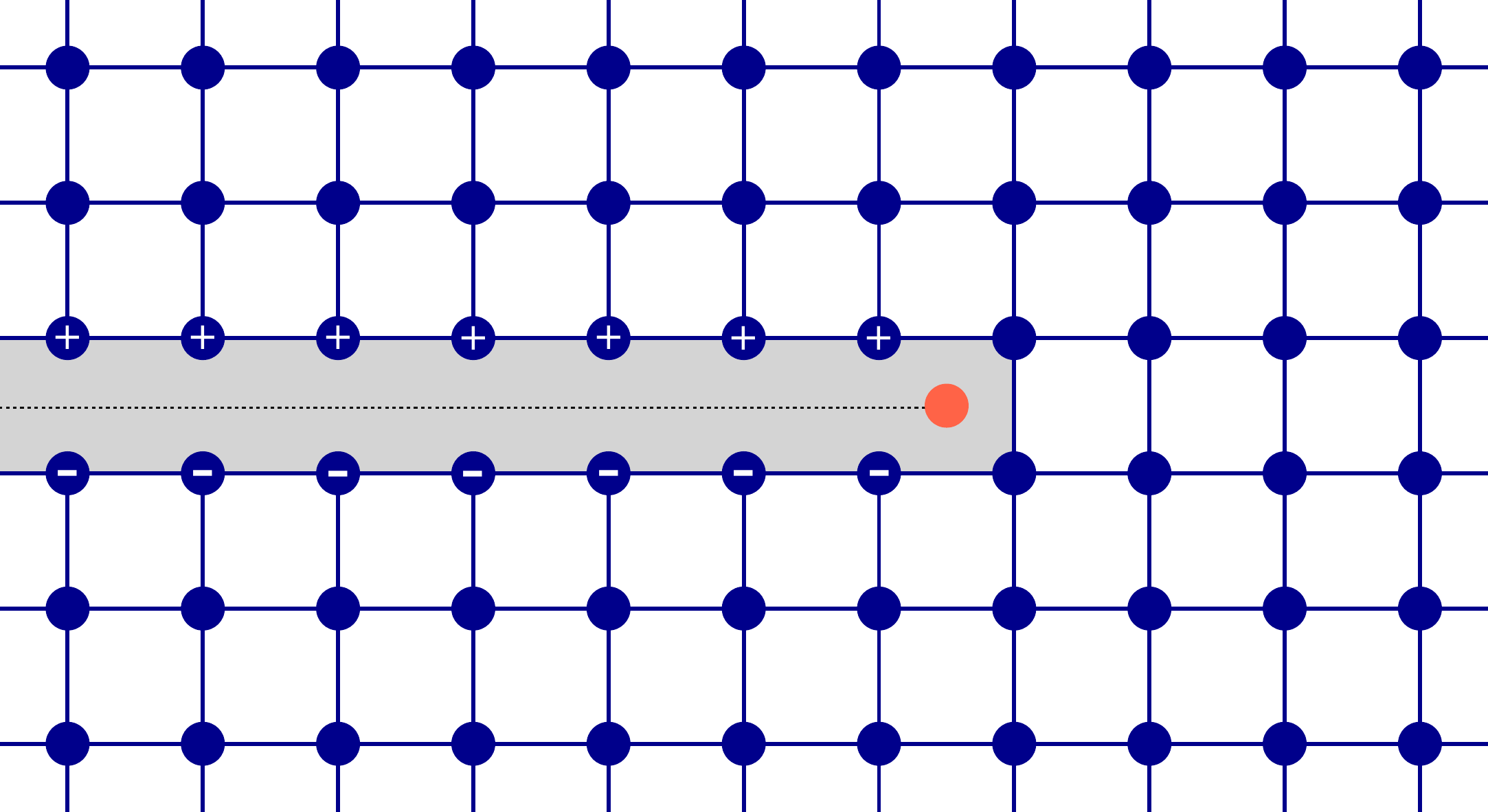}
\caption{The spatial setup of the problem with the crack tip depicted by a red dot, the crack cut $\Gamma_0$ by a dashed black line, the lattice points belonging to $\Gamma_+$ by dots with $+$ signs and the lattice points belonging to $\Gamma_-$ by dots with $-$ signs.} 
\label{fig:lattice_gamma}
\end{figure}
In the following we consider nearest-neighbour (NN) interactions between lattice sites, with the exception of sites on $\Gamma_{\pm}$, for which we adjust stencils to account for the presence of the crack. While the usual set of NN directions of the homogeneous square lattice is given by
{\cb \[
\Rc = \left\{e_1,e_2,-e_1,-e_2\right\},
\]}
we modify it as follows to accommodate the crack. For any $m \in \La$,  
\begin{equation}\label{Rcm}
\Rc(m) := \begin{cases} \Rc \quad & \text{for }\, m \not\in (\Gamma_+\cup\Gamma_-), \\
						\Rc \setminus \left\{\mp e_2\right\}\quad & \text{for }\, m \in \Gamma_{\pm}.\end{cases}.
\end{equation} 
For an anti-plane displacement $u\,\colon \,\La \to \R$, we define the finite difference operator as $D_{\rho}u(x) := u(x+\rho) - u(x)$ and the discrete gradient $Du(m) \in \R^{\Rc}$ as
\begin{equation}\label{dgrad}
\left(Du(m)\right)_{\rho} := \begin{cases}
D_{\rho}u(m)\quad &\text{if }\rho \in \Rc(m),\\
0\quad&\text{if }\rho \not\in\Rc(m).
\end{cases}
\end{equation}
As a result the gradient always lies in a four-dimensional space (as $|\Rc| = 4$) and if  $m\in \Gamma_{\pm}$ then $Du(m)$ has components corresponding to erased lattice directions set to zero.

Accordingly, we define the appropriate discrete Sobolev space as
\begin{align}\label{Hcc}
&\Hcc := \left\{u\,\colon\,\La \to \R \;|\; Du \in \ell^2\;\text{and}\; u(\hat{x}) = 0 \right\},\\
&\text{with }\|u\|_{\Hcc} := \|Du\|_{\ell^2} = \left(\sum_{m\in\La}|Du(m)|^2\right)^{\sfrac{1}{2}}. \nonumber
\end{align}
Here $\hat{x} = \left(\frac{1}{2},\frac{1}{2}\right)$ represents one of the lattice sites closest to the origin and the restriction $u(\hat{x})=0$ ensures that only one constant displacement lies in the space, thus making $\|\cdot\|_{\Hcc}$ a norm.
\subsection{Atomistic model for anti-plane fracture}\label{model}
Following the theory developed in \cite{EOS2016,2013-disl,2017-bcscrew} for point defects and straight dislocations, we formulate the static crack model as a minimisation problem
\begin{equation}\label{minprob}
\text{find}\quad\bar{u} \in \arg\min_{\Hcc} \E,
\end{equation}
with the energy difference functional given by
\begin{equation}\label{energy}
\E(u) = \sum_{m \in \La} V(D\hat{u}(m) + Du(m)) - V(D\hat{u}(m))
\end{equation}
where $V\,\colon\,\R^{\Rc} \to \R$ is an interatomic potential, $\hat{u}\,\colon\,\La\to\R$ is the {\em far-field predictor} and $u$ a core correction, thus giving us the actual displacement as $\hat{u} + u$. We choose the potential to be a NN pair-potential of the form
\begin{equation}\label{pair-potential}
V(Du(m)) = \sum_{\rho \in \Rc} \phi((Du(m))_{\rho}),
\end{equation}
with $\phi \in C^k(\R)$ for $k\geq 5$ satisfying {\cb $\phi(0)= 0$, $\phi(-r) = \phi(r)$ and $\phi''(0) = 1$. The assumption $\phi(0) = 0$ is made without loss of generality since we may always replace $\phi(r) \mapsto \phi(r) - \phi(0)$ without changing the energy difference. The assumption $\phi(-r) = \phi(r)$ is consistent with {\em anti-plane mirror symmetry} (see also \cite[Section 2.2.]{2017-bcscrew} and the relevant discussion in Section \ref{conclusion}). Finally $\phi''(0)=1$ may be assumed without loss of generality {\em as long as} $\phi''(0) > 0$ (upon replacing $\phi(r) \mapsto c \phi(r)$). This latter condition is equivalent to lattice stability \cite{E-Ming,2013-disl,EOS2016}, which is satisfied for virtually all bulk materials. Due to the symmetry $\phi(-r) = \phi(r)$ it follows that $\phi'(0)= \phi'''(0) = 0$, which we will use later on.}

The far-field predictor $\hat{u}$ is obtained from continuum linear elasticity (CLE), which is to be regarded as a boundary condition at infinity that imposes the existence of the defect. Following a standard procedure of pairing of the atomistic potential $V$  with its continuum counterpart $W\,\colon\,\R^2\to\R$ (the so-called Cauchy-Born strain energy function \cite{E-Ming,2012-ARMA-cb}){\cb , via $W(F):= V((F\cdot \rho)_{\rho\in \Rc})$}, the resulting CLE equation for $\hat{u}$ is given by 
\begin{align}
-\Delta \hat{u} &= 0\quad\text{in }\;\R^2\setminus \Gamma_0,\label{eqn-u-hat}\\
\nabla\hat{u} \cdot\nu &= 0\quad\text{on }\;\Gamma_0\setminus\{0\}\nonumber.
\end{align}
{\cb The particularly simple form of the equation is due to the fact that under the Cauchy-Born coupling we have  $\delta^2W(0) = c{\rm Id}$, as shown in \cite{2017-bcscrew}.} This equation has infinitely many solutions with the canonical choice being the sole solution that ensures local integrability near the crack tip and induces a stress which decays at infinity \cite{SJ12}. This solution can be characterised via the complex square root mapping $\omega\,\colon\,\R^2 \to \R^2$. In polar coordinates, $x = (r_x \cos{\theta_x},r_x\sin{\theta_x}) \in \R^2\setminus\Gamma_0$, it is given by
\begin{equation}\label{sqrt-map}
\omega(x) = (\omega_1(x),\omega_2(x)) = \left(\sqrt{r_x}\cos{\left(\sfrac{\theta_x}{2}\right)}, \sqrt{r_x}\sin{\left(\sfrac{\theta_x}{2}\right)}\right)
\end{equation}
and the canonical solution to \eqref{eqn-u-hat} is
\begin{equation}\label{upred}
\hat{u}(x) = \epsilon\,\omega_2(x).
\end{equation}
Here $\epsilon$ is a loading parameter with its magnitude corresponding to the size of the displacement on the opposite sides of the crack and its sign determining which side is being pulled up. Without loss of generality we assume $\epsilon \geq 0$.  As per Lemma \ref{gradomega} below, we further note that $|\nabla^j\hat{u}(x)| \lesssim |x|^{1/2-j}$ for any $j \in \N$.

\begin{remark}\label{CLE}
The premise of this formulation is two-fold. Firstly, it seeks to validate CLE as an accurate approximation of the atomistic effects away from the defect core in a crack defect setup. On the other hand, it also shows that the CLE solution can serve as an appropriate boundary condition for finite-domain numerical computation in a discrete setup, as tested in numerical tests described in Section \ref{numerics}.  
\end{remark}

\begin{remark}\label{rem-decay}
While in the case of an anti-plane screw dislocation the predictor $\hat{u}_s$ derived from CLE only just fails to be in the discrete energy space $\Hcc$ (namely $D\hat{u}_{\rm s} \in \ell^{2+\delta}$ for any $\delta > 0$, as described e.g. in \cite{hirth-lothe}), in the present case it is only true that $D\hat{u} \in \ell^{4+\delta}$. This phenomenon is a key reason why the analysis of a general crack defect is more involved. In the anti-plane case one way of circumventing it is to impose $\phi'''(0)=0$, but in a more general setup it is an open problem. We refer to Section \ref{conclusion} for an extended discussion.
\end{remark}
With the predictor and the interatomic potential specified, we can now state our main results.
\begin{theorem}\label{Eu-well-def}
The energy difference functional $\E$ in \eqref{energy} is well-defined on $\Hcc$ and $k$-times continuously differentiable. Furthermore, for $\epsilon$ sufficiently small, the minimisation problem \eqref{minprob} has a locally unique solution $\bar{u} \in \Hcc$ that depends continuously  on $\epsilon$ and satisfies strong stability, that is there exists $\lambda>0$ such that for all $v \in \Hcc$
\begin{equation}\label{stability}
\delta^2\E(\bar{u})[v,v] \geq \lambda \|v\|{\cb ^2}_{\Hcc}.
\end{equation}
\end{theorem}
For the proof, see Section \ref{P-u}.
\begin{theorem}\label{thm::u}
For any $\epsilon \geq 0$, every critical point of the energy difference functional $\E$ in \eqref{energy} satisfies 
\begin{equation}\label{Du-est}
|D\bar{u}(l)| \lesssim \epsilon |l|^{-3/2+\delta},
\end{equation}
for any $\delta>0$ and $|l|$ large enough. 
\end{theorem}
For the proof, see Section \ref{P-u2}. The sharpness of this result is tested numerically in Section \ref{numerics}. The appearance of arbitrarily small $\delta > 0$ in \eqref{Du-est} is due to the way we construct the lattice Green's function, as discussed after Theorem \ref{thm::G} below. \\
\subsection{Discrete Green's function for anti-plane crack geometry}\label{Green}
In order to establish Theorem \ref{thm::u}, we need to discuss the notion of a Green's function in the setup of anti-plane crack. We begin by defining the discrete divergence operator 
\begin{equation}\label{ddiv}
\divo g(m) := -\sum_{\rho \in \Rc}g_{\rho}(m-\rho) - g_{\rho}(m),\quad\text{for }g\,\colon\,\La\to\R^{\Rc}
\end{equation}
and discussing the lattice Hessian operator for the crack geometry. Adapting the general formulation from \cite{EOS2016} to the case of a pair-potential defined in \eqref{pair-potential} we have
\begin{equation}\label{h-simple}
\<{Hu}{v} = \sum_{m \in \La} Du(m)\cdot Dv(m)\implies Hu(m) = -\divo Du(m)\quad\forall m \in \La,
\end{equation}
with the pointwise formulation following from summation by parts.

\begin{definition}\label{G-def}
A function $\G\,\colon\,\La \times \La \to \R$ is said to be a lattice Green's function $\G$ for the anti-plane crack geometry if for all $m,s \in \La$,
\begin{subequations}
\label{G-var-sym+delta}
\begin{align}
H\G(m,s) &= \delta_{ms} \label{G-delta} \\
\G(m,s) &= \G(s,m), \label{G-var-sym}
\end{align}
\end{subequations}
where $\delta_{ms}$ denotes the Kronecker delta and $H$ is applied with respect to first variable.  
\end{definition}
We note that in \eqref{G-delta} one can view $s$ as a parameter and $H$ as a difference operator applied with respect to the first variable. However, due to \eqref{G-var-sym}, it is also true that \eqref{G-delta} holds with $H$ applied with respect to the second variable. Furthermore, $\G$ is not uniquely determined, since any discretely harmonic function can be added, i.e. if $v\,\colon\,\La\to\R$ is such that $Hv = 0$, then $\G(m,s) + v(m) + v(s)$ also satisfies \eqref{G-var-sym+delta}. 

For functions in two variables such as $\G$, we introduce a notation for finite differences as follows. If $v\,\colon\,\La\times\La \to \R$, then 
\[
\Dm{\rho}v(m,l) := v(m+\rho,l) - v(m,l)\quad\text{ and }\quad\Ds{\rho}v(m,l) := v(m,l+\rho) - v(m,l)
\]
and for $j\in\{1,2\}$, we have $D_j v(m^1,m^2) \in \R^{\Rc}$ with
\begin{equation}\label{dgrad2}
\left(D_j v(m)\right)_{\rho} := \begin{cases}
D_{j\rho}v(m^1,m^2)\quad &\text{if }\rho \in \Rc(m^j),\\
0\quad&\text{if }\rho \not\in\Rc(m^j).
\end{cases}
\end{equation}
For any $\G$ satisfying Definition \ref{G-def}, any solution $\bar{u}$ to \eqref{minprob} can be rewritten as
\[
D_{\tau}\bar{u}(l) = \sum_{m\in\La}\left(H(\Ds{\tau}\G(m,l)\right)\bar{u}(m) = \sum_{m\in\La}D\bar{u}(m)\cdot \Dm{}\Ds{\tau}\G(m,l),
\]
hence the result that enables us to prove Theorem \ref{thm::u} consists in finding a lattice Green's function that has desired decay properties of its mixed derivative. 
\begin{theorem}\label{thm::G}
There exists a lattice Green's function $\G\,\colon\,\La\times\La\to\R$ satisfying Definition \ref{G-def} such that, for any $\delta >0$, $\rho \in \Rc(l)$, and $\sigma \in \Rc(s)$,
\[
|\Dm{\rho}\Ds{\sigma}\G(l,s)| \lesssim (1+|\omega(l)||\omega(s)||\omega(l)-\omega(s)|^{2-\delta})^{-1},
\]
where $\omega$ is the complex square root map defined in \eqref{sqrt-map}.
\end{theorem} 
The approach we employ is based on the observation that finding $\G$ can also be cast as a predictor-corrector problem, with the decomposition $\G = \hat{\G} + \bar{\G}$, where $\hat{\G}$ has an explicit formula and $\bar{\G}$ belongs to the energy space $\Hcc$ in both variables. This idea has already been explored in \cite{EOS2016}, but was notably aided by the applicability of Fourier methods due to the spatial homogeneity of the reference configuration. The novelty of our work stems from the fact that the discreteness and inhomogeneity of the domain means that Fourier analysis is no longer applicable. In particular, it renders the task of establishing the decay estimates on $\G$ much more challenging. In our approach we first establish suboptimal estimates on $\bar{\G}$ with the help of the homogeneous lattice Green's function $\G^{\rm hom}$ employed together with suitably chosen cut-offs and a local mapping onto a discrete Riemann surface corresponding to the complex square root. We then use this initial estimate in a boot-strapping argument. The appearance of arbitrarily small $\delta > 0$ follows from the fact that this argument saturates at the known decay of $\hat{\G}$.

\begin{remark}
While $|\omega(m)| = |m|^{-1/2}$, in general it is not true that $|\omega(m)-\omega(s)| \sim |m-s|^{-1/2}$, as in fact
\begin{equation}\label{csqrt-id}
|m-s| = |\omega(m)-\omega(s)||\omega(m)+\omega(s)|.
\end{equation}
The estimate is thus expressed in terms of $\omega$-map, as one can then conveniently resort to a change of variables $\xi = \omega(m)$ when working with $\G$. See Figure \ref{fig:distorted_lattice}. 
\end{remark}
\begin{figure}[!htbp]
\centering
\begin{minipage}{0.48\textwidth}
\raggedleft
\includegraphics[width=0.8\linewidth]{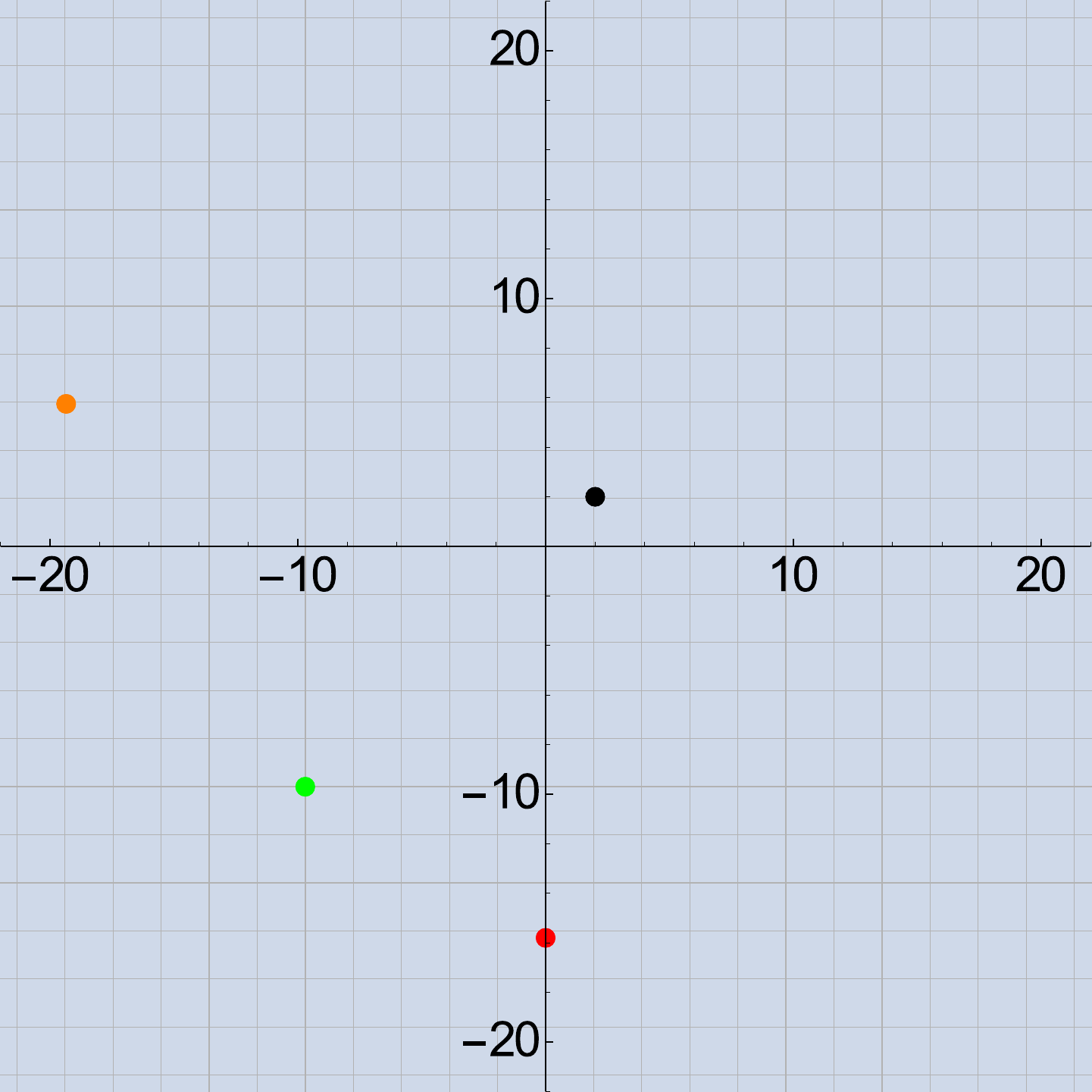}
\end{minipage}
\begin{minipage}{0.48\textwidth}
\raggedright
\includegraphics[width=0.8\linewidth]{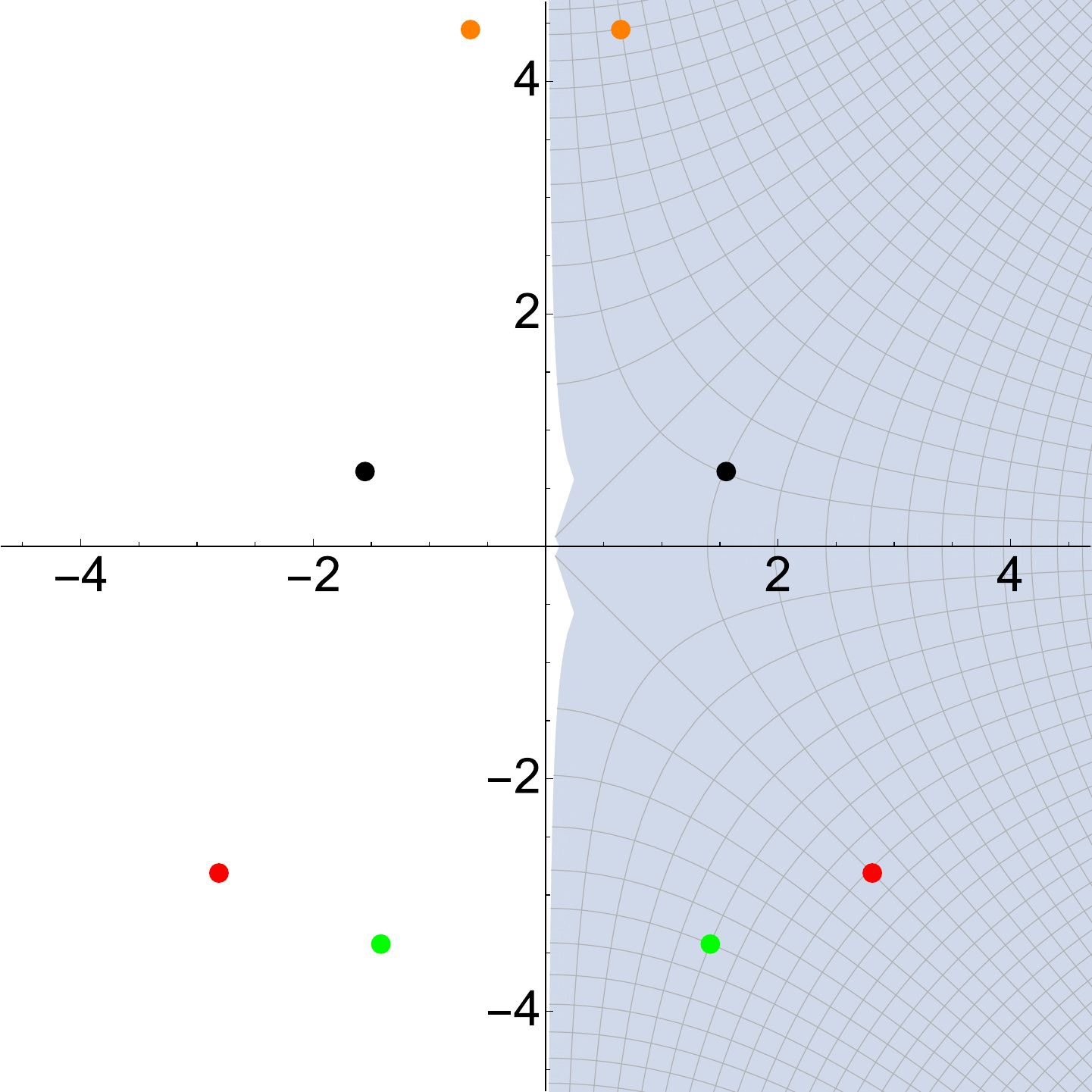}
\end{minipage}
\caption{The complex square root $\omega$ maps the square lattice (left) onto a distorted half-space lattice (right). In particular, the distorted lattice lives in $\R^2_+$, the half-space with positive first coordinate. The dots represent lattice points and their images under $\omega$ and also their reflections across $y$-axis.}
\label{fig:distorted_lattice}
\end{figure}
\subsection{Rate of convergence to the thermodynamic limit}
In this section we consider a supercell approximation to \eqref{minprob} on a finite domain confined to a ball of radius $R$ and establish the rate of convergence as $R\to\infty$.

The setup is similar to the one descibed in \cite{EOS2016,2017-bcscrew}, that is we consider a domain $B_R \cap \La \subset \Omega_R \subset \La$ with the boundary condition $\hat{u}$ on $\La\setminus \Omega_R$ and state it as a Galerkin approximation
\begin{equation}\label{minprob_approx}
\text{find }\,\bar{u}_R \in \arg\min_{\mathcal{H}^0_R} \E,
\end{equation}
\[
\mathcal{H}^0_R := \{v\,\colon\,\La\to\R \,|\,v = 0 \text{ in } \La\setminus\Omega_R\}.
\]
We prove the following.
\begin{theorem}\label{thm::conv}
If $\bar{u}$ is a solution to \eqref{minprob} that is strongly stable in the sense of \eqref{stability}, then for all $\beta > 0$, there exist $C,R_0 > 0$ such that for all $R > R_0$, there exists a stable solution $\bar{u}_R$ to \eqref{minprob_approx} satisfying
\[
\|\bar{u}_R - \bar{u}\|_{\Hcc} \leq CR^{-1/2+\beta}.
\]
\end{theorem}
\begin{proof}
The proof of the statement follows almost immediately from the corresponding result in \cite[Theorem 3.8]{EOS2016pp}, as long as we extend the Discrete Poincar\'e inequality described therein to the domain with a crack. This requires a construction of a suitable interpolation operator that correctly takes into account the region between $\Gamma_0$ and $\Gamma_+ \cup \Gamma_-$. This construction shall be carried out as part of the proof of Theorem \ref{Eu-well-def}.
\end{proof}
\subsection{Numerical results}\label{numerics}
In this section we present results of numerical tests that confirm the rate of decay of $|D\bar{u}|$ established in Theorem \ref{thm::u} and the convergence rate from Theorem \ref{thm::conv}.  The setup precisely follows the one described in \cite[Section 3]{2017-bcscrew}, with $\La$ and $\Rc$ already specified and the pair-potential employed given by
\[
\phi(r) = \frac{1}{6}(1 - \exp(-3r^2)).
\]
Theorem \ref{thm::u} suggests that  $\lvert D \bar{u}(x) \rvert \lesssim \lvert x \rvert^{-3/2}$, while Theorem \ref{thm::conv} suggests that in the supercell approximation \eqref{minprob_approx} we expect $\|\bar{u}_R - \bar{u}\|_{\Hcc}\sim \mathcal{O}(R^{-1/2})$, where $R$ is the size of the domain. To compute equilibria we employ a standard Newton scheme, terminating at an $\ell^{\infty}$-residual of $10^{-8}$.

In Figure \ref{fig::numerics} we plot the decay of $|D\bar{u}|$ rescaled by the value of $\epsilon$ used, as well as the convergence rate to the thermodynamic limit, confirming the predictions of Theorems \ref{thm::u} and~\ref{thm::conv}. 
\begin{remark}
We also carried out a similar set of tests for the anti-plane crack problem on a triangular lattice, obtaining qualitatively equivalent results. This indicates that the current restriction to the square lattice has purely technical origins and that it is possible to extend our results to other Bravais lattices. 
\end{remark}
\begin{figure}[!htb]
  \begin{subfigure}[t]{.48\textwidth}
    \centering
    \includegraphics[width=\linewidth]{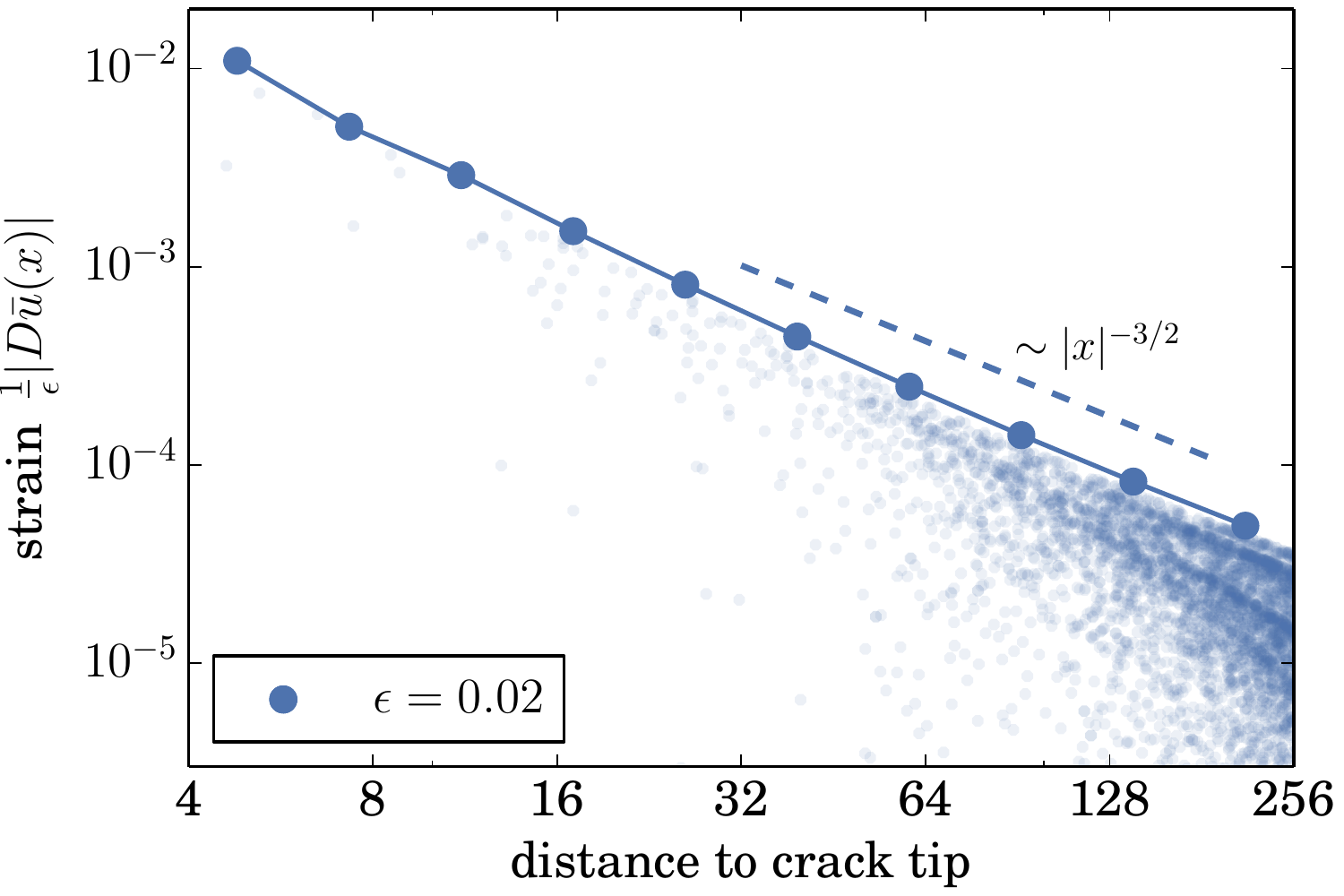}    
  \end{subfigure}
  \hfill
  \begin{subfigure}[t]{.48\textwidth}
    \centering
    \includegraphics[width=\linewidth]{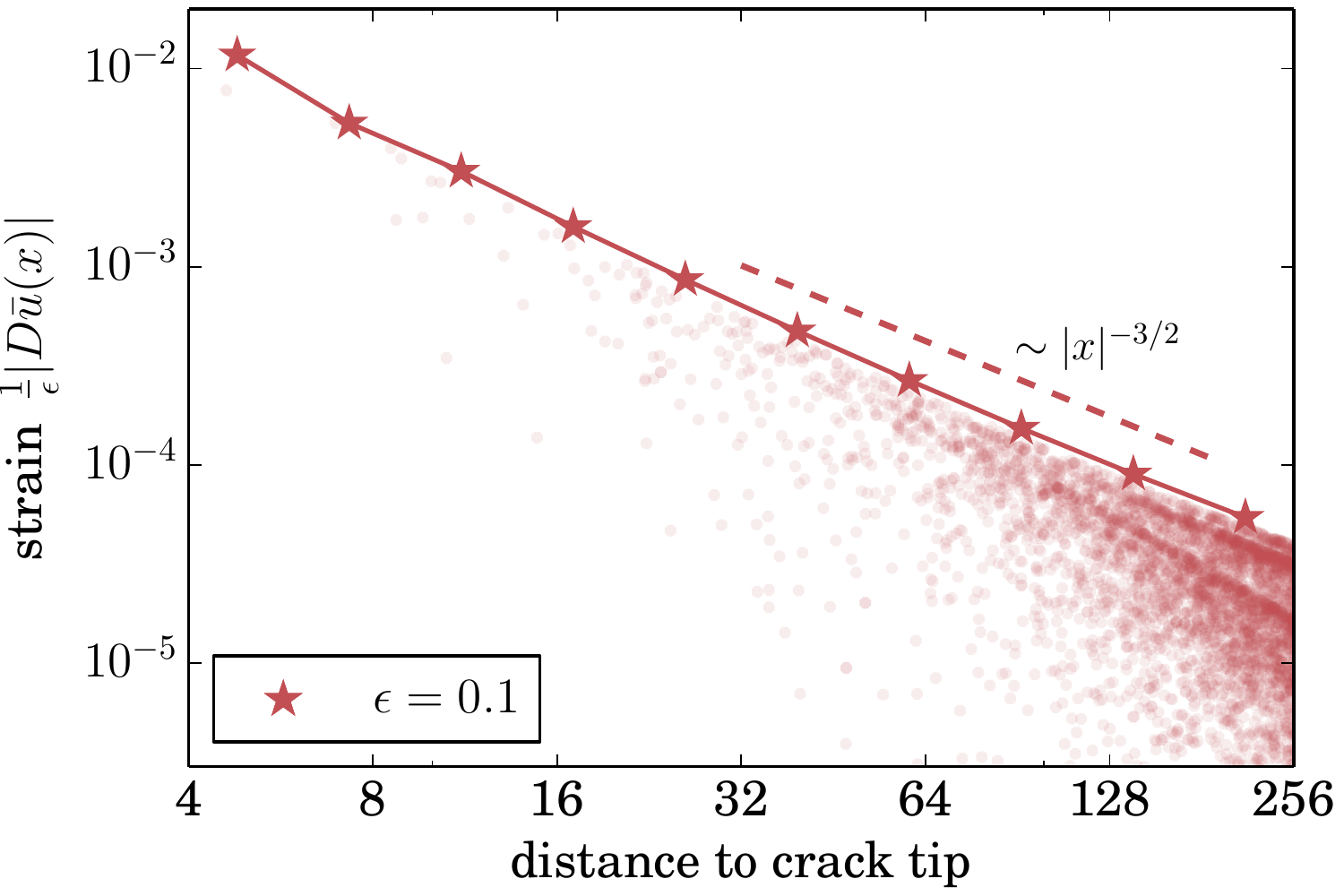}    
  \end{subfigure}
  \medskip
  \begin{subfigure}[t]{.48\textwidth}
    \centering
    \includegraphics[width=\linewidth]{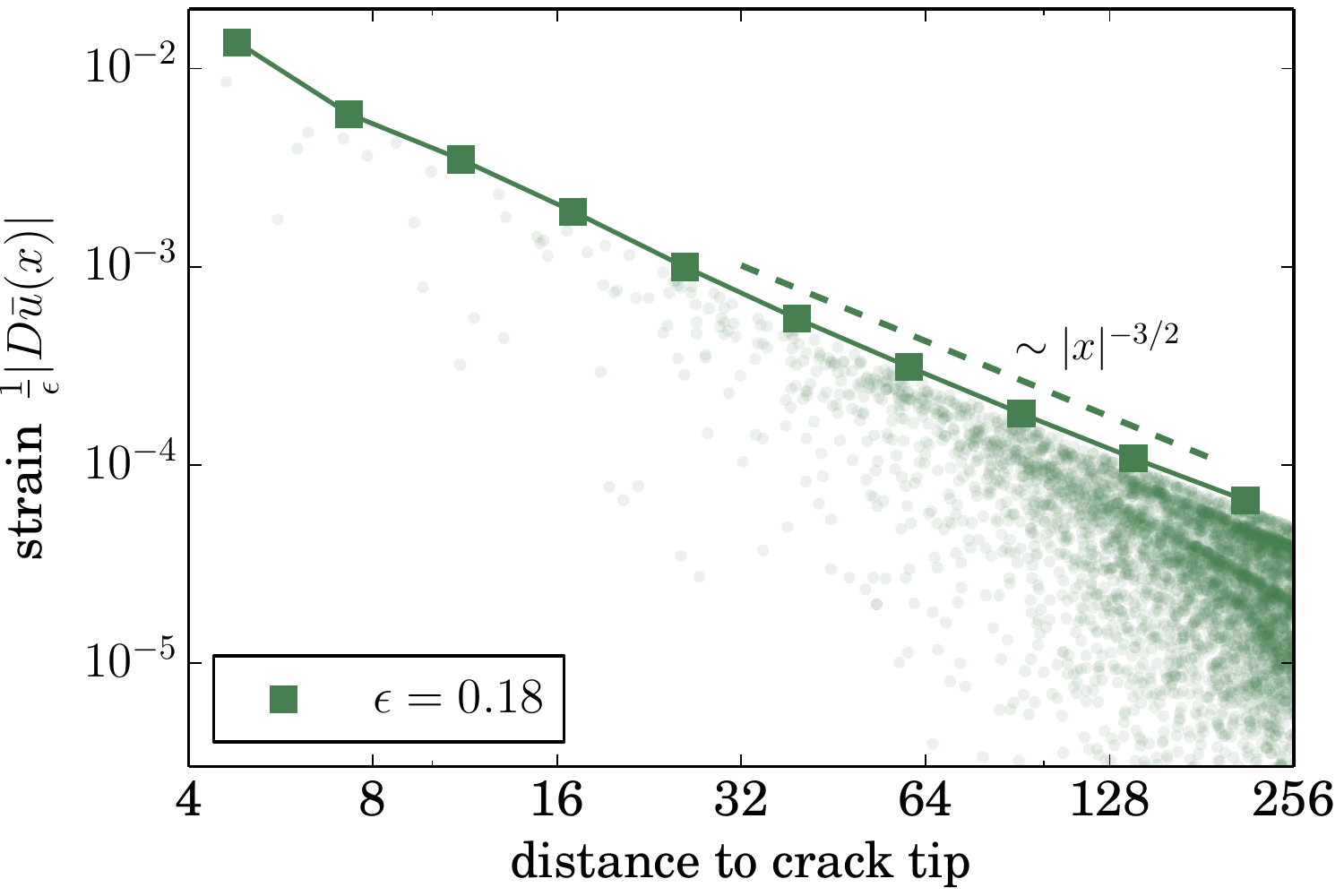}    
  \end{subfigure}
  \hfill
  \begin{subfigure}[t]{.48\textwidth}
    \centering
    \includegraphics[width=\linewidth]{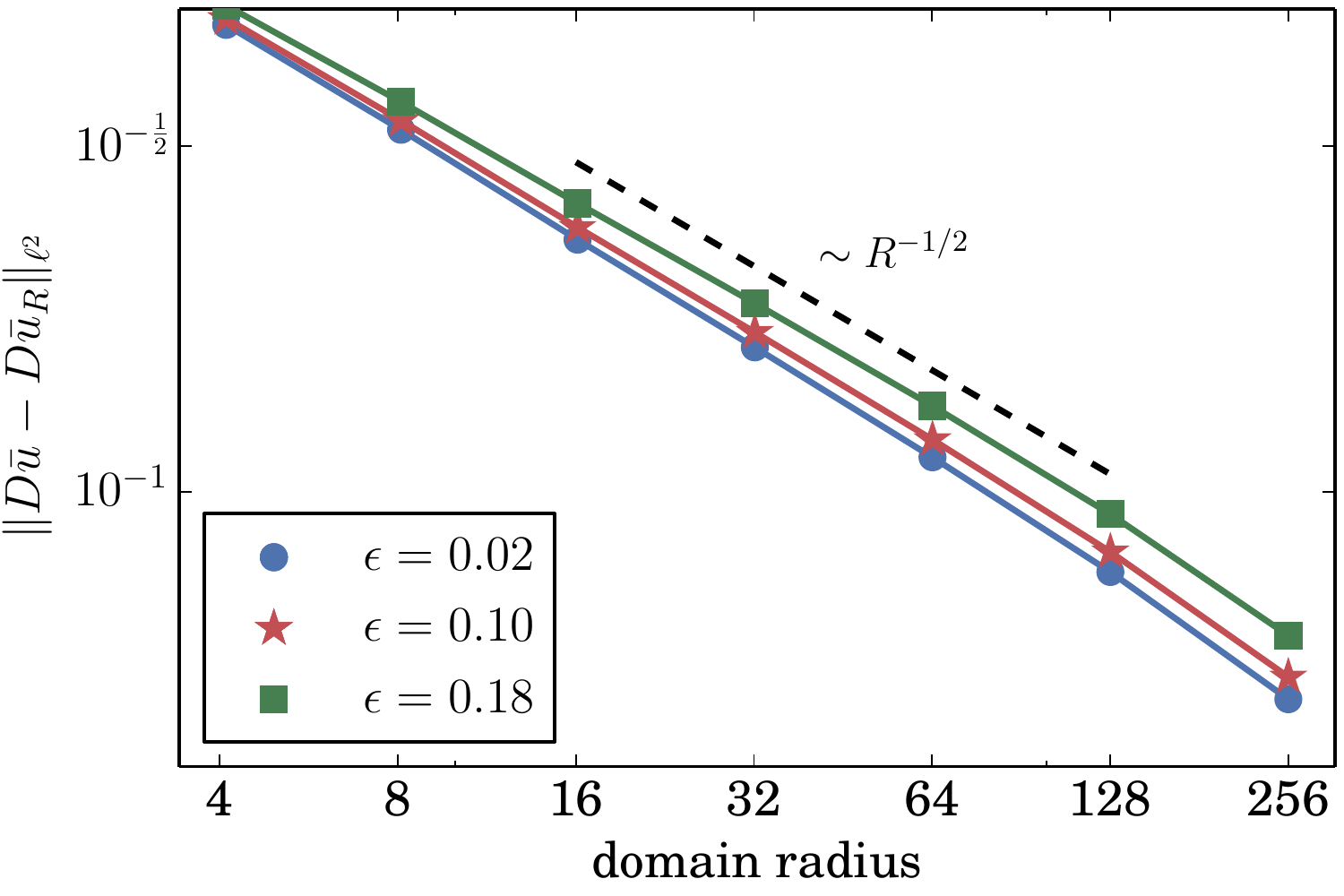}    
  \end{subfigure}
  \captionsetup{width=\linewidth}
  \caption{The decay of the corrector rescaled by the loading parameter, i.e. $\frac{1}{\epsilon}|D\bar{u}|$, for different values of $\epsilon$. Transparent dots denote data points $(|x|, |Du(x)|)$, solid curves their envelopes. We observe the expected rate of $|x|^{-3/2}$ and the linear scaling of $D\bar{u}$ is evident. Bottom right: The rate of convergence of the corresponding supercell approximation. The expected rate $R^{-1/2}$ is observed. }\label{fig::numerics}
\end{figure}
\section{Conclusion and discussion}\label{conclusion}
We have extended the mathematical theory of atomistic modelling of crystalline defects studied in \cite{EOS2016,2013-disl,2017-bcscrew}  to the case of an anti-plane crack defect under nearest-neighbour interactions on a square lattice. This work can be regarded as a first step towards an extension to general atomistic models of fracture, including vectorial models on an arbitrary lattice under an arbitrary interatomic potential.\\
In this paper we have laid out many of the steps needed to achieve this, and in what follows we discuss some of the key technical difficulties which must be overcome to extend the present work.\\

\paragraph{Anti-plane models on an arbitrary Bravais lattice under many-body finite interactions potential:}
The missing ingredient needed to extend the results to anti-plane models beyond NN interactions on a square lattice is the ability to estimate $\bar{\G}$. In our arguments, we rely on a construction of a locally isomorphic mapping from the defective lattice to a homogeneous lattice, which preserves the fact that $\bar{\G}$ is a critical point of the associated energy-difference functional (see Section \ref{P-G-setup}). A similar construction based on a different reflection can also be carried out for the triangular lattice under NN interactions, but this approach is ill-suited to arbitrary finite interactions. This is because as we enlarge the radius of interaction, we increase the number of constraints required for the extended version of $\bar{\G}$ to remain a critical point of the corresponding extended functional, whereas any argument based on reflection argument (possibly coupled with translation and scaling) has a fixed number of degrees of freedom associated with it. For the same reason the current framework only permits many-body terms in the interatomic potential that do not contribute to the Hessian (which is why we restrict ourselves to pair-potentials).\\

\paragraph{More general static crack models:}
Already in the simplified anti-plane setup, the key limiting consequence of the the slow decay of the predictor $\hat{u}$ can be seen by looking at {\cb 
\[
\<{\E(\bar{u})}{v} = \sum_{m \in \La} \sum_{\rho \in \Rc(m)} \phi'(D_{\rho}\hat{u}(m)+D_{\rho}\bar{u}(m))D_{\rho}v(m)
\]
}and Taylor-expanding $\phi'$ around $0$. Crucially, without {\cb the further assumption of mirror symmetry, which in the case of a square lattice, as discussed in \cite[Section 2.2.]{2017-bcscrew}, fails if {\em both} $\phi(-r)\neq\phi(r)$ and the interaction range is such that each nearest-neighbour finite difference is accounted for only in one direction, that is $\tilde{\Rc} := \{e_1,e_2\}$ (with suitable adjustments in the resulting $\tilde{\Rc}(m)$ for atoms at the crack surface $\Gamma$)}, the slow decay rate of $\hat{u}$ implies that 
\begin{equation}\label{term-cancelllation}
\left( v\mapsto \frac{\phi'''(0)}{2}\sum_{m\in\La}\sum_{\rho \in {\cb \tilde{\Rc}(m)}}(D_{\rho}\hat{u}(m))^2 D_{\rho}v(m)\right)\not\in (\Hcc)^*,
\end{equation}
{\cb unless $\phi'''(0) = 0$, which is usually not true when considering the 2D lattice as a projection of an associated 3D lattice. Note that $\phi(-r)\neq\phi(r)$ is not in itself enough for the mirror symmetry to fail, because for $\Rc(m)$ defined in \eqref{Rcm}, it holds that
\[
m \in \La\,\text{ and }\,\rho \in \Rc(m) \implies m+\rho \in \La\,\text{ and }\,-\rho \in \Rc(m+\rho),
\]
which renders \eqref{term-cancelllation} (with inner sum over $\Rc(m)$ and not $\tilde{\Rc}(m)$)  null even if $\phi'''(0) \neq 0$, since the contribution of $(m,m+\rho)$ is cancelled by the contribution of $(m+\rho,m)$.} 

To extend the theory beyond models with mirror symmetry one has to follow the idea of {\it development of solutions} introduced in \cite{2017-bcscrew}, which consists in prescribing a predictor of the form $\hat{u} + \hat{u}_2$, with the additional term arising from higher-order PDE theory (related to nonlinear elasticity). This ensures that
\[
v \mapsto \sum_{m\in\La}\phi''(0)D\hat{u}_2(m)\cdot Dv(m)
\]
up to leading order cancels with \eqref{term-cancelllation}. The role of $\hat{u}_2$ is especially important for vectorial models, since the concept of mirror symmetry does not translate to models that allow for in-plane displacements, meaning that the vectorial equivalent of \eqref{term-cancelllation} never automatically vanishes.\\

\paragraph{In-plane static crack models:}
A further complication related to vectorial models is that as soon as we look beyond nearest-neighbours interactions, we begin to observe surface effects, as for instance investigated in \cite{theilsurface11}. These effects, induced by the crack surface, do not enter the analysis of vectorial models for dislocations and point defects in \cite{EOS2016} and thus pose a major new challenge, as they can potentially lead to surface atoms assuming a notably different structure compared to the bulk which renders the approximation of CLE invalid. Likewise, it may have an impact on the  corresponding lattice Green's function and can potentially make obtaining its decay estimates much more involved. \\

\paragraph{The role of loading parameter $\epsilon$:}
It is the appearance of $\epsilon$ that ensures we can prove existence of strongly-stable solutions to the problem in \eqref{minprob}, as it allows us to employ the Implicit Function Theorem. This is in contrast with dislocation problems, where, except for specific cases with stringent assumptions as e.g. in \cite{2013-disl}, we simply assume that a solution exists. The use of IFT also points to a potential bifurcation occurring for some critical $\epsilon_{\rm crit}$, which is a further deviation from the known theory, as in CLE the choice of $\epsilon$ is irrelevant.

\section{Proofs}\label{proofs}
\subsection{Preliminaries}
In this section we introduce the remaining notation and concepts to be used throughout that were left out of the introductory section. 

Firstly, we define sets
\[
\Omega_{\Gamma} := \big\{x \in \R^2 \,\big|\,x_1 \leq \textstyle{\frac{1}{2}} \text{ and }  x_2 \in \left(-\frac{1}{2},\frac{1}{2}\right)\big\} \setminus \Gamma_0, \quad\quad\Gamma := \partial\Omega_{\Gamma} \setminus \Gamma_0,
\]
with $\Gamma$ being the line that includes lattice points encompassing $\Gamma_0$ and, similarly, $\Omega_{\Gamma}$ being the space that $\Gamma$ encompasses, except for $\Gamma_0$ itself.

We further would like to comment on the definition of the gradient operator and why we set the contribution of broken bonds to zero. This formulation allows us to sum by parts in a convenient way. For instance, for any $u,v\,\colon\,\Lambda \to \R$ with compact support, we have (cf. \eqref{h-simple}) that
\[
\sum_{m\in\La}\sum_{\rho \in \Rc(m)}D_{\rho}u(m)D_{\rho}v(m) = \sum_{m \in\La}Du(m)\cdot Dv(m) = \sum_{m\in\La}\left(-\divo Du(m)\right)v(m).
\]
In the following it is often of interest to only sum over bonds at the crack surface. To this end, for any $m \in \La$ and $\rho \in \Rc(m)$, we introduce the notation $b(m,\rho):= \{m+t\rho\,|\,t\in[0,1]\}$ and the following short-hand summation notation
\[
\sum_{b(m,\rho)\subset\Gamma} \mathlarger{\mathlarger{\equiv}} \sum_{\substack{m\in\La, \rho\in\Rc(m),\\ b(m,\rho) \subset \Gamma}}
\]
together with an analogous definition for bonds not on the crack surface. Likewise, it is important to distinguish the following sets corresponding to the unit square centered at the origin
\begin{equation}\label{omega0u0}
\Omega_0 := \Omega_{\Gamma}\cap [-\textstyle{\frac{1}{2}},-\textstyle{\frac{1}{2}}]^2, \quad Q_0 := \Omega_0 \cap \Gamma.
\end{equation}
We also introduce a shorthand notation related to the complex square root mapping,
\begin{equation}\label{omegashort}
\omega_x := \omega(x),\quad\omega_{xs}^- := \omega(x) - \omega(s),\quad\omega_{xs}^+ := \omega(x) + \omega(s)
\end{equation}
and quote the following standard result without proof. 
\begin{lemma}\label{gradomega}
For $j \in \N$, the complex square root map $\omega$ defined in \eqref{sqrt-map} satisfies
\[
|\nabla^j\omega(x)| \lesssim |x|^{1/2-j}.
\]
\end{lemma}

\subsection{Proofs for static anti-plane crack model}\label{P-c}
\subsubsection{Proof of Theorem \ref{Eu-well-def}}\label{P-u}
We separate the proof into two parts, with one devoted to $\E$ defined in \eqref{energy} and the other to the solution to \eqref{minprob}.\\
\paragraph{The energy difference functional $\E$ is well-defined and differentiable:}
For any $v\,\colon\,\La\to\R$ with compact support we can rewrite the energy difference functional $\E$ as
\begin{equation}\label{Eu-rewritten}
\E(v) = \E_0(v) + \<{\delta \E(0)}{v},	
\end{equation}
where
\begin{align*}
\E_0(v) := \sum_{m \in \La}\sum_{\rho \in \Rc(m)}\Big(\phi(D_{\rho}\hat{u}(m) + D_{\rho}v(m)) - \phi(D_{\rho}\hat{u}(m)) - \phi'(D_{\rho}\hat{u}(m))D_{\rho}v(m)\Big)
\end{align*}
and
\[
\<{\delta \E(0)}{v} = \sum_{m \in \La}\sum_{\rho \in \Rc(m))}\phi'(D_{\rho}\hat{u}(m))D_{\rho}v(m).
\]
Since $\phi \in C^k(\R)$ for $k\geq 5$, a simple Taylor expansion argument ensures that $\E_0$ is well-defined on $\Hcc$ (cf. \cite{EOS2016}). Thus the proof relies on showing that $\delta\E(0)$ is a bounded linear functional on $\Hcc$, as then \eqref{Eu-rewritten} holds for any $v\in\Hcc$. Noting that $\phi'(0)= \phi'''(0) = 0$ and $\phi''(0) = 1$, we Taylor-expand $\phi'$ around zero to get
\begin{equation}\label{E-Taylor-exp}
|\<{\delta \E(0)}{v}| \lesssim \left|\sum_{m\in\La}D\hat{u}(m)\cdot Dv(m)\right| + \left|\sum_{m \in \La}R_\phi(m)\cdot Dv(m)\right|,
\end{equation}
where $R_{\phi}$ represents the remaining terms in the Taylor expansion and due to the fact $|D\hat{u}(m)| \lesssim |m|^{-1/2}$, it is immediate that 
\begin{equation}\label{rem-phi}
|R_{\phi}(m)| \lesssim |m|^{-3/2}
\end{equation}
and hence
\[
\left|\sum_{m \in \La}R_\phi(m)\cdot Dv(m)\right|\lesssim \|Dv\|_{\ell^2}.
\]
It remains to estimate the first term of the right-hand side of \eqref{E-Taylor-exp}. To this end we shall exploit the fact that $\hat{u}$ solves the equation given by \eqref{eqn-u-hat}, in particular after constructing a suitable interpolation operator that takes any lattice function to the continuum space. Firstly we tessellate the domain $\R^2\setminus\Gamma_0$ as follows. We carve the squares in the lattice into two right-angle triangles and introduce a (P1) piecewise linear interpolation operator $I$ over the resulting triangulation (see the left of Figure \ref{fig:lattice_ext}).

In order to exploit the boundary condition that $\hat{u}$ satisfies, we also want $Iv$ to be well-defined on $\Omega_{\Gamma}$ and continuous across $\Gamma$. Away from the defect core this is possible by extending it so that it aligns with the the values of $Iv(x)$ for $x \in \Gamma$ and is constant in the normal direction, as shown in the centre of Figure \ref{fig:lattice_ext}. Additionally, near the origin we create two new interpolation points as shown on the right of Figure \ref{fig:lattice_ext}, one at the origin and one in-between points $a$ and $d$ and we denote it by $\widehat{ad}$. We define the interpolation there as $Iu(0) := \frac{1}{4}\left(u(a) + u(b) + u(c) + u(d)\right)$ and $\lim_{x\downarrow \widehat{ad}}u(x) = u(d)$ whereas $\lim_{x\uparrow\widehat{ad}}u(x)=u(a)$, emphasising the fact that the resulting P1 interpolant does not need to be continuous across $\Gamma_0$, but is continuous across the triangle $T_3$. 
\begin{figure}[!htbp]
\centering
\begin{minipage}{0.40\textwidth}
\centering
\includegraphics[width=0.98\linewidth]{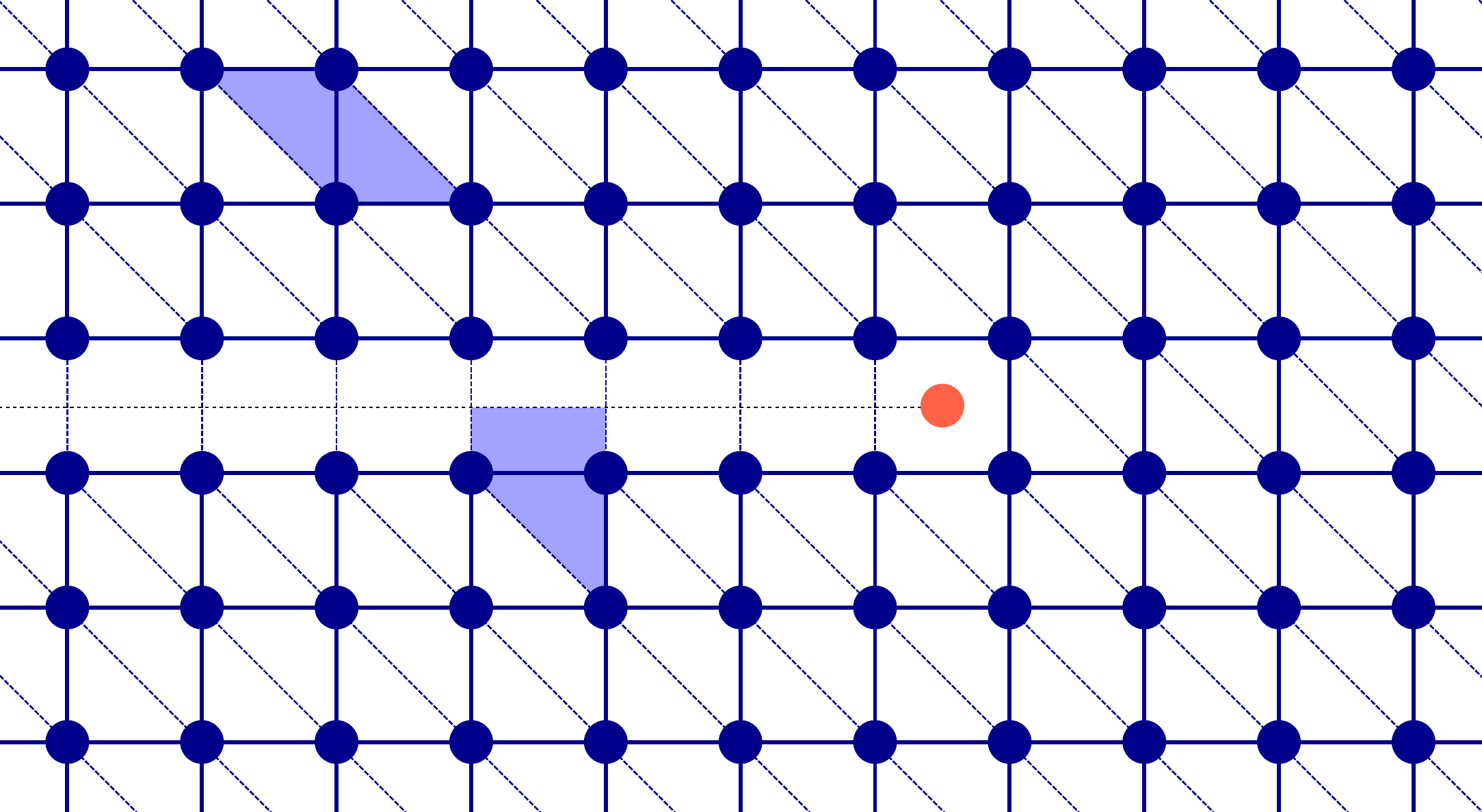}
\end{minipage}
\begin{minipage}{0.33\textwidth}
\centering
\includegraphics[width=0.98\linewidth]{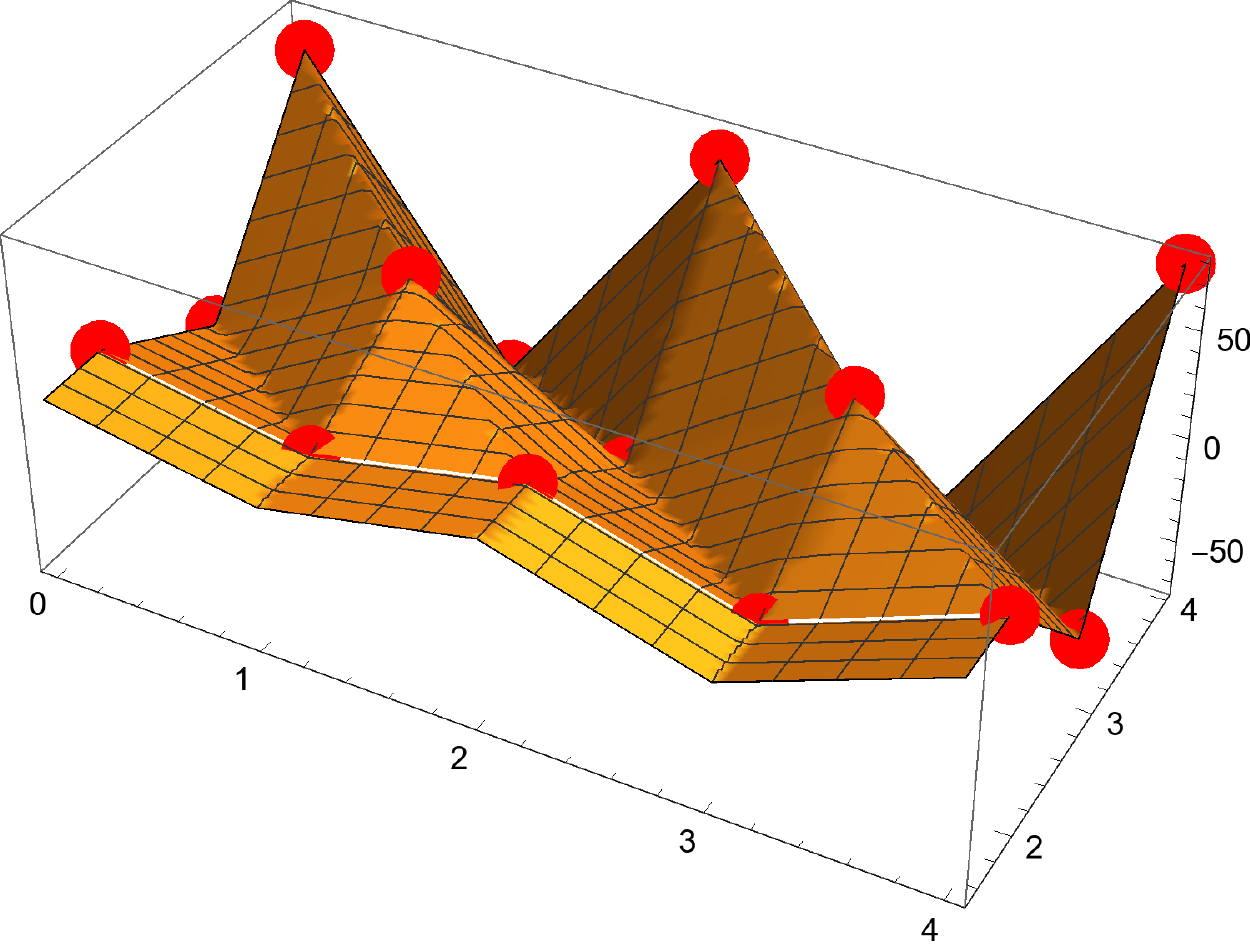}
\end{minipage}
\begin{minipage}{0.25\textwidth}
 \centering
 \raggedright
\includegraphics[width=0.95\linewidth]{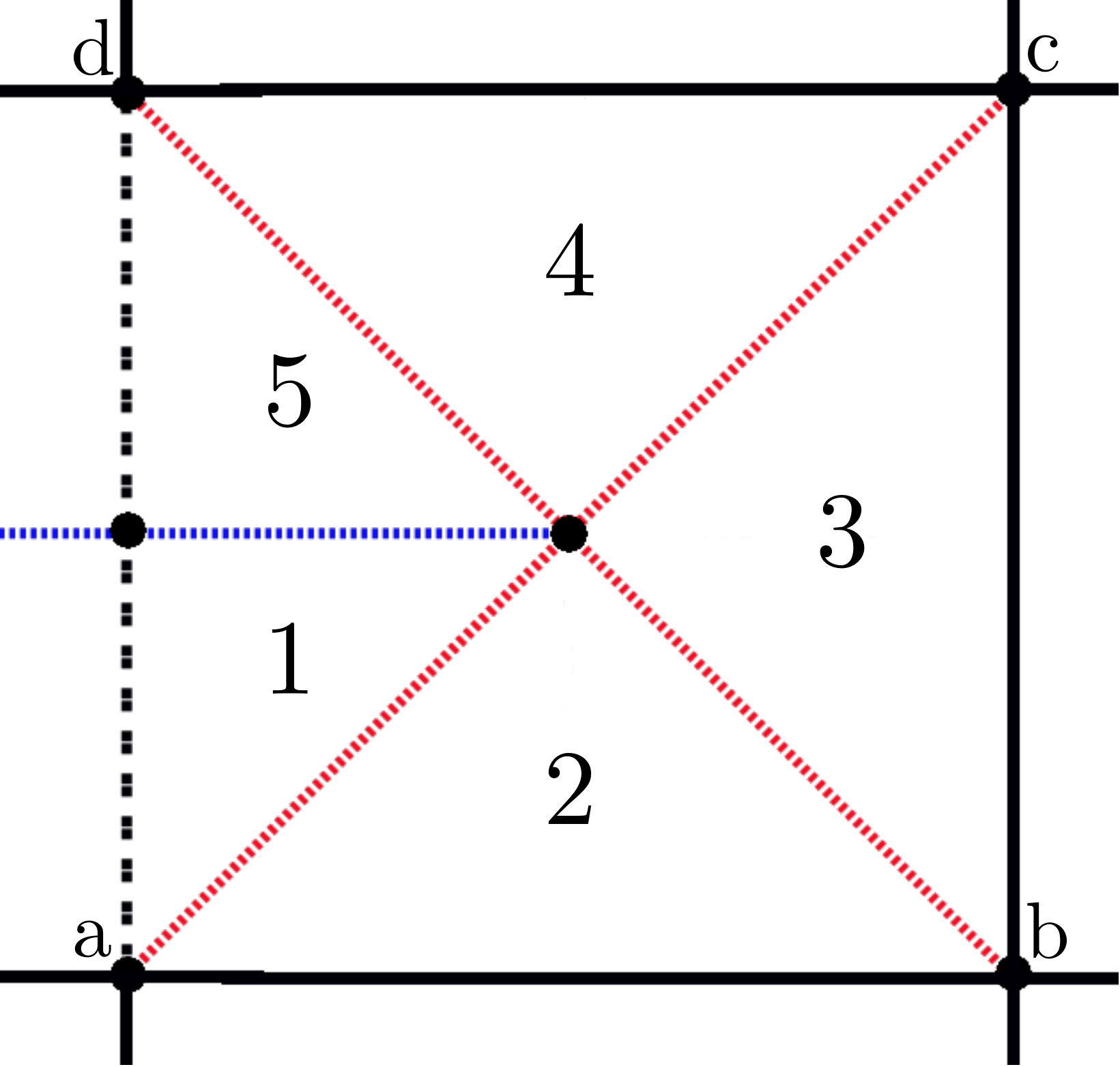}
\end{minipage}
\captionsetup{width=\linewidth}
\caption[caption]{Left: The tessellation of the domain $\R^2\setminus\Gamma_0$, with triangles away from the crack and rectangles at the crack surface. In blue a typical region of integration associated with a bond.\\
Middle: For some lattice function $u\,\colon\,\La \to \R$ each red dot represents the point in the three dimensional space corresponding to $(l_1,l_2,u(l))$ for some lattice point $l \in \La$. The orange region represents the graph of the corresponding interpolant $Iu$, in particular clearly illustrating its extension to $\Omega_{\Gamma}\setminus\Gamma_0$ (here looking from above).\\
Right: Near the origin we create two additional interpolation points, one at the origin and one half-way between lattice points on the crack surface closest to the origin and impose a triangulation as shown. The resulting P1 interpolation introduces a collection triangles $\{T_1,\dots,T_5\}$  and we stress that $Iu$ is not continuous across the common edge of $T_1$ and $T_5$.}\label{fig:lattice_ext}
\end{figure}

We can thus write
\begin{align*}
0 &=  C_{\La}\int_{\R^2\setminus\Gamma_0} (-\Delta \hat{u}(x))Iv(x)\,dx = C_{\La}\int_{\R^2\setminus\Gamma_0} \nabla \hat{u}(x) \cdot \nabla Iv(x)\,dx,
\end{align*}
where in particular the second equality follows from integration by parts and the boundary term is not there due to the boundary condition. Hence we in fact aim to estimate
\[
\sum_{m\in\La}\<{D\hat{u}(m)}{Dv(m)}  - C_{\La}\int_{\R^2\setminus(\Omega_{\Gamma}\cup\Gamma_0)} \nabla \hat{u}(x)\cdot\nabla Iv(x) - C_{\La}\int_{\Omega_{\Gamma}} \nabla \hat{u}(x)\cdot\nabla Iv(x).
\]
\begin{remark}\label{Clambda}
The constant $C_{\La}$ depends on the lattice under consideration. In the case of the square lattice, $C_{\La} = 2$, but for in instance if we were to consider the triangular lattice with NN interactions, the constant would be $2\sqrt{3}$. The freedom of choice is a consequence of the fact that $\hat{u}$ satisfies Laplace equation with zero Neumann boundary condition. It also justifies why $\hat{u}$ is a valid predictor for any choice of constant the $\epsilon$ in \eqref{upred}. This is in contrast with the subsequent Green's function argument in Section \ref{P-G-setup} where we have to prescribe the correct constant in the equation for the corresponding predictor. 
\end{remark}

The triangulation of $\R^{2}\setminus\Omega_{\Gamma}$ induced by the P1 interpolation introduces a collection of triangles $\mathcal{T}$. Inside any given $T \in \mathcal{T}$ both components of $\nabla Iv$ are constant and each corresponds to $D_{\rho}v(l)$ for some bond $b(l,\rho)$ being an edge of $T$. As a result we can write 
\[
C_{\La}\int_{\R^2\setminus(\Omega_{\Gamma}\cup\Gamma_0)}\nabla \hat{u}(x)\cdot \nabla Iv(x) = \sum_{m \in \La}\sum_{\rho\in\Rc(m)}\left(\int_{U_{m\rho}} \nabla_{\rho}\hat{u}(x)\,dx\right)D_{\rho}v(m),
\]
where $U_{m\rho}$ is the union of triangles for which a given bond $b(m,\rho)$ is an edge (cf. Figure \ref{fig:lattice_ext}). The constant $C_{\La} = 2$ disappears due to the fact that the set of lattice directions under consideration counts each bond twice. 

A similar analysis is applicable to the integral over $\Omega_{\Gamma}$. Away from $\Omega_0$ (the unit square centred at the origin  defined in \eqref{omega0u0}), it can be tessellated into a collection of rectangles $\big(Q_{m\rho}\big)$, each associated with one lattice bond $b(m,\rho)\subset\Gamma\setminus Q_0$ (cf. Figure \ref{fig:lattice_ext}), where we recall $Q_0 = \Omega_0\cap\Gamma$. Due to how we construct the interpolant of $v$, we can thus conclude that
\[
C_{\La}\int_{\Omega_{\Gamma}\setminus(\Omega_0)} \nabla \hat{u}(x) \cdot \nabla Iv(x)\,dx = \sum_{b(m,\rho)\subset(\Gamma\setminus Q_0)}\left(\int_{Q_{m\rho}}\nabla_{\rho}\hat{u}(x)\,dx\right)D_{\rho}v(m).
\] 
It can also be readily checked that (using the notation from Figure \ref{fig:lattice_ext}) {\cb 
\[
\int_{\Omega_0}\nabla\hat{u}(x,s)\cdot\nabla Iv(x)dx = C_{b,a}(v(b)-v(a)) + C_{c,b}(v(c) - v(b)) + C_{d,c}(v(d)-v(c)),
\]
where the coefficients are given by
\begin{align}\label{Q0interpol1}
C_{b,a} &:= \frac{1}{2}\Big(3\int_{T_1}\nabla_{e_1}\hat{u} + 2\int_{T_2}\nabla_{e_1}\hat{u} + \int_{T_2}\nabla_{e_2}\hat{u} 
+\int_{T_3}\nabla_{e_1}\hat{u}+\int_{T_4}\nabla_{e_2}\hat{u}-\int_{T_5}\nabla_{e_1}\hat{u} \Big),\\
C_{c,b} &:=\Big(\int_{T_1}\nabla_{e_1}\hat{u} + \int_{T_2}\nabla_{e_2}\hat{u}+\int_{T_3}\nabla_{e_2}\hat{u}+\int_{T_4}\nabla_{e_2}\hat{u}-\int_{T_5}\nabla_{e_1}\hat{u}\Big),\\
C_{d,c} &:= \frac{1}{2}\Big(\int_{T_1}\nabla_{e_1}\hat{u}+\int_{T_2}\nabla_{e_2}\hat{u}-\int_{T_3}\nabla_{e_1}\hat{u}-2\int_{T_4}\nabla_{e_1}\hat{u} + \int_{T_4}\nabla_{e_2}\hat{u}-3\int_{T_5}\nabla_{e_1}\hat{u} \Big).\label{Q0interpol2}
\end{align}
We note that the directions with respect to which finite differences are taken can be reversed, thus we also define
\begin{equation}\label{Q0interpol3}
C_{a,b} := -C_{b,a},\,C_{b,c} = -C_{c,b}\,\text{ and }\,C_{c,d} := -C_{d,c}.
\end{equation}}
We can therefore write
\begin{align*}
{\cb \<{\delta \E(0)}{v} } &= \sum_{b(m,\rho)\not\subset\Gamma}\left(D_{\rho}\hat{u}(m) - \int_{U_{m\rho}}\nabla_{\rho}\hat{u}(x)\,dx\right)D_{\rho}v(m)\\ 
&+ \sum_{b(m,\rho)\subset({\cb \Gamma\setminus Q_0)}}\left(D_{\rho}\hat{u}(m) - \int_{U_{m\rho}}\nabla_{\rho}\hat{u}(x)\,dx - \int_{Q_{m\rho}}\nabla_{\rho}\hat{u}(x)\,dx\right)D_{\rho}v(m)\\
&+\sum_{b(m,\rho) \subset Q_0}{\cb \left(D_{\rho}\hat{u}(m) - C_{m+\rho,m}\right)}D_{\rho}v(m),
\end{align*}
{\cb where the coefficients $C_{m+\rho,m}$ are given by \eqref{Q0interpol1}-\eqref{Q0interpol3}. Since
\[
\int_{B_1(0)\setminus\Gamma_0}|\nabla\hat{u}(x)|dx \lesssim \int_{0}^1 r^{1/2}dx\,{\cb < \infty},
\] 
it is clear that for any $b(m,\rho) \subset Q_0$, $\left(D_{\rho}\hat{u}(m) - C_{m+\rho,m}\right)$ can be bounded uniformly. }

Bearing in mind that $D_{\rho}\hat{u}(m) = \int_0^1 \nabla_{\rho}\hat{u}(m+t\rho)\,dt$ and observing that for $b(m,\rho)\not\subset\Gamma$, we have  $|U_{m\rho}|= 1$, we exploit the fact that both regions of integration share the same mid-point. A Taylor expansion followed by a standard quadrature error estimate thus leads to 
\[
b(m,\rho)\not\subset\Gamma \implies \left|D_{\rho}\hat{u}(m) - \int_{U_{m\rho}}\nabla_{\rho}\hat{u}(x)\,dx\right| \lesssim |\nabla^3\hat{u}(m)|.
\]

On the other hand, for $b(m,\rho)\subset\Gamma\setminus Q_0$ there is only one triangle and thus $|U_{m\rho}| = \frac{1}{2}$, but we also have $|Q_{m\rho}| = \frac{1}{2}$. While regions of integration no longer share a mid-point, we still Taylor-expand and apply a weaker (first-order) quadrature error estimate to conclude that 
\[
b(m,\rho)\subset\Gamma\setminus Q_0 \implies \left|D_{\rho}\hat{u}(m) - \int_{U_{m\rho}}\nabla_{\rho}\hat{u}(x)\,dx\right| \lesssim |\nabla^2\hat{u}(m)|.
\]
Finally, since Lemma \ref{gradomega} implies that for both $j=2,3$ and  $m \in \La$ with $m \approx 0$ we have $|\nabla^j\hat{u}(m)| \sim \mathcal{O}(1)$ (in particular finite since $|m| > \frac{1}{\sqrt{2}}$) , we can incorporate any bond $b(m,\rho)\subset Q_0$ into the general conclusion that
\begin{equation}\label{dhatudv}
\left|\sum_{m\in\La}D\hat{u}(m)\cdot Dv(m)\right| \lesssim \sum_{b(m,\rho)\not\subset\Gamma}|\nabla^3\hat{u}(m)||D_{\rho}v(m)| + \sum_{b(m,\rho)\subset\Gamma}|\nabla^2\hat{u}(m)||D_{\rho}v(m)|
\end{equation}
and since $|\nabla^3\hat{u}(m)| \lesssim |m|^{-5/2}$ and $|\nabla^2\hat{u}(m)|\lesssim |m|^{-3/2}$, then
\[
\left|\sum_{m\in\La}D\hat{u}(m)\cdot Dv(m)\right| \lesssim \|Dv\|_{\ell^2}.
\]
Thus we can conclude that for any $v \in \Hcc$,
\[
|\<{\delta \E(0)}{v}| \lesssim \|Dv\|_{\ell^2}.
\]
The fact that $\E$ is at least $k$-times continuously differentiable then naturally follows from $\phi \in C^k(\R)$, see \cite{2012-ARMA-cb} for an analogous argument. \\

\paragraph{Existence, local uniqueness, and strong-stability of solutions:} We begin by quoting the Implicit Function Theorem, adapted from \cite{serge}:
\begin{theorem}[Implicit Function Theorem]\label{IFT}
 Let $X,\,Y,\,Z$ be Banach spaces. Let the mapping $F\,:\,X\times Y\to Z$ be continuously Fr\'echet differentiable with respect to both $x$ and $y$. If $(x_0,y_0) \in X \times Y$, $F(x_0,y_0) = 0$ and the mapping $x \mapsto DF(x_0,y_0)(x,0)$ is a Banach space isomorphism from $X$ onto $Z$, then there exist neighbourhoods $U$ of $x_0$ and $V$ of $y_0$ and a Fr\/echet differentiable function $g\, :\, V \to U$ such that $F(g(y), y) = 0$ and $F(x,y) = 0$ if and only if $x = g(y)$, for all $(x,y) \in U \times V$.
\end{theorem}
In our setting, we have $X = \Hcc, Y = \R$ and $Z = (\Hcc)^*$. We can interpret the energy difference functional $\E$ as defined on $\Hcc \times \R$ and thus $F = \delta_u \E$. We notice that for $\epsilon = 0$ we have a trivial solution $\bar{u}_0 = 0$, thus giving us the pair $(\bar{u}_0,0) \in \Hcc \times \R$. We further observe that 
\[
\<{DF(u_0,0)(v,0)}{w} = \delta_u^2 \E(u_0,0)[v,w] = \sum_{m \in \La}\phi''(0)Dv(m)\cdot Dw(m)
\]
and since $\phi''(0) = 1$, the mapping $DF(u_0,0)(\cdot,0)$ is indeed an isomorphism, as it is in fact the Riesz map from Riesz Representation Theorem for Hilbert spaces (cf. \cite{rudin}).

Hence all the assumptions of the theorem are fulfilled and we can conclude that in a neighbourhood of $(\bar{u},0)$ we have a unique solution path of the form $\{(u(\epsilon),\epsilon)\,|\, \epsilon \in [0,\epsilon_{\rm crit})\}$ with continuous dependence of $u$ on $\epsilon$. The strong-stability \eqref{stability} of solutions for $\epsilon_{crit}$ small enough follows from the fact that it is trivially satisfied for $u_0$ with $\lambda = \phi''(0) = 1$ and the continuous dependence of solutions on $\epsilon$, as we can always write
\[
\delta^2\E(u(\epsilon),\epsilon)[v,v] = \left(\delta^2\E(u(\epsilon),\epsilon) - \delta^2 \E(u_0,0)\right)[v,v] + \delta^2\E(u_0,0)[v,v].
\] 
\subsubsection{Proof of Theorem \ref{thm::u}}\label{P-u2}
Let $\tau \in \Rc(l)$. Using the lattice Green's function for crack geometry from Theorem \ref{thm::G} (to be proven in Section \ref{P-G}), we define a test function $v(m):=\Ds{\tau}\G(m,l)$, which decays like $|D_{\rho}v(m)| \lesssim (1+|\omega_m||\omega_l||\omega^-_{ml}|^{2-\delta})^{-1}$ for any $\delta > 0$. We can thus write that
\begin{align*}
D_{\tau}\bar{u}(l) &= \sum_{m\in\La}\phi''(0)D\bar{u}(m)\cdot Dv(m) \\ 
&= \sum_{m \in \La}\sum_{\rho\in \Rc(m)}\left(\phi'(D_{\rho}\hat{u}(m)) + \phi''(0)D_{\rho}\bar{u}(m) - \phi'(D_{\rho}\hat{u}(m) + D_{\rho}\bar{u}(m))\right)D_{\rho}v(m)\\
 &-\sum_{m \in \La}\sum_{\rho\in \Rc(m)}\phi'(D_{\rho}\hat{u}(m))D_{\rho}v(m)\\
&=: \sum_{m\in\La} A(m)\cdot Dv(m) - B(m)\cdot Dv(m)
\end{align*}
where we exploited the fact that $\bar{u}$ is a critical point, that is it satisfies
\begin{equation}\label{ueqn}
\<{\delta\E(\bar{u})}{v} = \sum_{m\in\La}\sum_{\rho \in \Rc(m)}\phi'(D_{\rho}\hat{u}(m) + D_{\rho}\bar{u}(m))D_{\rho}v(m) = 0\quad \forall v \in \Hcc.
\end{equation}
A Taylor expansion of $\phi'$ around zero yields that
\[
|A(m)| \lesssim |D\hat{u}(m)|^{4} + |D\bar{u}(m)|^{2} \lesssim |\omega_m|^{-4} + |D\bar{u}(m)|^{2},
\]
where we used $|\nabla\hat{u}(m)| \lesssim |m|^{-1/2} = |\omega_m|^{-1}$. Similarly
\[
\left|\sum_{m\in\La} B(m)\cdot Dv(m)\right| \lesssim  \left|\sum_{m\in\La}D\hat{u}(m)\cdot Dv(m)\right| + \left|\sum_{m \in \La}R_\phi(m)\cdot Dv(m)\right|,
\]
where $R_\phi$ as in \eqref{rem-phi}. In light of \eqref{dhatudv} we thus obtain
\[
\left|\sum_{m\in\La} B(m) \cdot Dv(m)\right| \lesssim \sum_{m\in\La}|\omega_m|^{-3}|Dv(m)|,
\]
which, when put together with the decay of $v$ implies that
\begin{equation}\label{dbaru1}
|D_{\tau}\bar{u}(l)| \lesssim \sum_{m\in\La}|\omega_m|^{-3}(1+|\omega_m||\omega_l||\omega^-_{ml}|^{2-\delta})^{-1} + \sum_{m\in\La}|D\bar{u}(m)|^2(1+|\omega_m||\omega_l||\omega^-_{ml}|^{2-\delta})^{-1}.
\end{equation}
The first term on the right-hand side of \eqref{dbaru1} can be estimated as follows. We define
\[
\sum_{m \in \La}f(m) := \sum_{m \in \La}(1+|\omega_m|^{3})^{-1}(1+|\omega_m||\omega_l||\omega^-_{ml}|^{2-\delta})^{-1} 
\]
and observe that away from the sharp spikes at $m=l$ and $m= 0$ we can bound this series by the corresponding integral, that is we can say
\[
\sum_{m \in \La}f(m)  \lesssim f(l) + f(0) + \int_{D}f dm,
\]
where $D:=(\R^2\setminus\Gamma_0)\setminus(B_1(l)\cup B_1(0))$. Firstly we note that
\[
f(l) = (1+|\omega_l|^{3})^{-1}\quad\text{and}\quad f(0) = (1+|\omega_l|^{3-\delta})^{-1}.
\]
For the integral term we introduce a change of variables $\xi = \omega_m$, which leads to $\zeta:= \omega_l$, and $dm = |\xi|^2d\xi$. As a result, we have
\[
\int_{D}f(x)dx = \int_{\omega(D)}\frac{|\xi|^{2}}{(1+|\xi|^{3})(1+|\xi||\zeta|{\cb |\xi-\zeta|})^{2-\delta}} d\xi \lesssim \int_{\omega(D)}|\zeta|^{-1}|\xi|^{-2}|\xi-\zeta|^{-2+\delta}d\xi =: \int_{\omega(D)}\tilde{f}d\xi.
\]
Bearing in mind that 
\[
\omega(D) = \R^2_+\setminus (B_1(0) \cup \omega(B_1(l))),
\]
we carve the region of integration into 
\[
\Omega_0 := B_{\frac{|\zeta|}{2}}(0)\cap \omega(D),\quad\Omega_{\zeta} := B_{\frac{|\zeta|}{2}}(\zeta)\cap \omega(D)\quad\text{and}\quad \Omega' := \omega(D) \setminus (\Omega_{0} \cup \Omega_{\zeta})
\]
and estimate the integral over each region separately as follows:
\[
\int_{\Omega_0}\tilde{f}d\xi \lesssim |\zeta|^{-3+\delta} \int_1^{\frac{|\zeta|}{2}}r^{-1}dr \lesssim |\zeta|^{-3+\delta}\log|\zeta|,
\]
\begin{equation}\label{fomegazeta1}
\int_{\Omega_{\zeta}}\tilde{f}d\xi \lesssim |\zeta|^{-3} \int_{{\cb \frac{1}{|\zeta|}}}^{\frac{|\zeta|}{2}}r^{-1+\delta}dr \lesssim |\zeta|^{-3+\delta} {\cb + |\zeta|^{-3-\delta} \lesssim |\zeta|^{-3+\delta}} ,
\end{equation}
where{\cb , in the second integral, $\frac{1}{|\zeta|}$ appears due to the exclusion of $B_1(l)$ from $D$ which, due to \eqref{csqrt-id}, translates in the domain of integration to
\[
\xi \in\Omega_{\zeta} \implies |\xi-\zeta| \geq \frac{1}{|\xi+\zeta|} \geq \frac{1}{|\xi|+|\zeta|} \gtrsim \frac{1}{|\zeta|},
\]
where the first inequality follows from $1 \leq |m-l| = |\xi-\zeta||\xi+\zeta|$, as in \eqref{csqrt-id}, and the last one from $|\xi| \lesssim |\zeta|$.} Finally,
\[
\int_{\Omega'}\tilde{f}d\xi \lesssim |\zeta|^{-1} \int_{{\cb|\zeta|}}^{\infty}r^{-3+\delta}dr \lesssim |\zeta|^{-3+\delta}.
\]
Since $\zeta = \omega_l$, we can thus conclude that 
\begin{equation}\label{dbaru-first-term-est}
\sum_{m\in\La}|\omega_m|^{-3}(1+|\omega_m||\omega_l||\omega^-_{ml}|^{2-\delta})^{-1} \lesssim |\omega_l|^{-3+\delta}\log|\omega_l| \lesssim |\omega_l|^{-3+\tilde{\delta}},
\end{equation}
for any $\tilde{\delta} > \delta$.

For the second term on the right-hand side of \eqref{dbaru1}, we look at three regions separately: $\Omega_1 := B_{\frac{|l|}{2}}(0)$, $\Omega_2 := B_{\frac{|l|}{2}}(l)$ and $\Omega_3 :=\La\setminus(\Omega_1 \cup \Omega_2)$. We observe that
\[
\sum_{m\in\Omega_1}|D\bar{u}(m)|^2(1+|\omega_m||\omega_l||\omega^-_{ml}|^{2-\delta})^{-1} \lesssim |\omega_l|^{-3+\delta}\|D\bar{u}\|_{\ell^2} \lesssim |\omega_l|^{-3+\delta}
\]
%\begin{align*}
%\sum_{m\in\Omega_1}|D\bar{u}(m)|^2(1+|\omega_m||\omega_l||\omega^-_{ml}|^{2-\delta})^{-1} &\lesssim |\omega_l|^{-1}\sum_{m\in\Omega_1} |D\bar{u}(m)|^{2}|\omega_m|^{-1}|m-l|^{-2}|%\omega_l|^{2}\\
%&\lesssim |\omega_l|^{-3}\|D\bar{u}\|_{\ell^2} \lesssim |\omega_l|^{-3}.
%\end{align*}
Similarly, $m \in \Omega_3 \implies |\omega_{ml}^-|\gtrsim |\omega_l|$ and $|\omega_m| \gtrsim  |\omega_l|$, hence
\[
\sum_{m\in\Omega_3}|D\bar{u}(m)|^2(1+|\omega_m||\omega_l||\omega^-_{ml}|^{2-\delta})^{-1} \lesssim |\omega_l|^{-4}\sum_{m\in\Omega_3} |D\bar{u}(m)|^{2}\lesssim |\omega_l|^{-4}\|D\bar{u}\|^2_{\ell^2} \lesssim |\omega_l|^{-4}.
\]
Finally, we can always replace one power of $|D\hat{u}(m)|$ with the $\ell^{\infty}$--norm , thus allowing us to apply the Cauchy-Schwarz inequality to obtain
\[
\sum_{m\in\Omega_2}|D\bar{u}(m)|^2(1+|\omega_m||\omega_l||\omega^-_{ml}|^{2-\delta})^{-1} \lesssim \|D\bar{u}\|_{\ell^{\infty}(\Omega_2)}\|D\bar{u}\|_{\ell^{2}(\Omega_2)}\left(\sum_{m\in\Omega_2}(1+|\omega_m|^{2}|\omega_l|^{2}|\omega^-_{ml}|^{4-2\delta})^{-1}\right)^{1/2}
\]
Noting that the sum is finite and that $\Omega_2 \subset \La\setminus B_{\frac{|l|}{2}}(0)$ we combine this with \eqref{dbaru-first-term-est} to obtain 
\begin{align*}
|D_{\tau}\bar{u}(l)| \lesssim |\omega_l|^{-3+\tilde{\delta}} + \|D\bar{u}\|_{\ell^2(\La\setminus B_{\frac{|l|}{2}}(0))}\|D\bar{u}\|_{\ell^{\infty}(\La\setminus B_{\frac{|l|}{2}}(0))}.
\end{align*}
{\cb Subsequently we define $w(r):= \|D\bar{u}\|_{\ell^{\infty}\La\setminus B_{r}(0))}$ and employ a technical result detailed in \cite[Lemma 6.3, Step 2]{EOS2016pp} originating from the regularity theory for systems of elliptic PDEs \cite{giusti}, to conclude that the function $v(r):= r^{-3/2+\delta}w(r)$ is bounded on $\R_+$, which implies that}
\[
|D_{\tau}\bar{u}(l)| \lesssim |\omega_l|^{-3+\tilde{\delta}}.
\]
This estimate holds for an arbitrary $\tau \in \Rc(l)$ and arbitrarily small $\tilde{\delta} > 0$, hence we have established the result. The linear scaling with $\epsilon$ is evident from the fact that in the interpolation trick used to obtain \eqref{dhatudv}, the loading parameter can be taken outside the summation, thus persists linearly.
\newpage 
\subsection{Proofs for discrete lattice Green's function $\G$}\label{P-G}
\subsubsection{Setup}\label{P-G-setup}
As briefly described in Section \ref{Green}, the approach we employ is that we seek a lattice Green's function of the form $\G = \hat{\G} + \bar{\G}$, with $\hat{\G}$ explicitly known. In practice, we proceed by first considering two closely related predictor-corrector problems: one to find $\tilde{\G}_1(\cdot,s) \in \Hcc$ that satisfies \eqref{G-delta} for a fixed $s$ and the other to find $\tilde{\G}_2(m,\cdot) \in \Hcc$ that satisfies \eqref{G-delta} for a fixed $m$ but with $H$ applied to the second variable. To conclude the result, one then has to make a suitable adjustment that takes into account how $\Hcc$ is defined (in particular the restriction that $\tilde{\G}_1(\hat{x},s) = \tilde{\G}_2(m,\hat{x})=0$ resulting from \eqref{Hcc}).

Rewriting both problems in variational form, we consider
\begin{equation}\label{minprobG}
\text{find }\,\tilde{\G}_i \in \arg\min_{\Hcc} \tilde{\E_i},
\end{equation}
where
\begin{equation}\label{energyG1}
\tilde{\E_1}(\mathcal{F})  = \sum_{m \in \La}\Bigg[\frac{1}{2}\left(|D_1\hat{\G}(m,s) + D\mathcal{F}(m,s)|^{2} - |D_1\hat{\G}(m,s)|^{2}\right)- \delta_{ms}\left(\hat{\G}(m,s) +\mathcal{F}(m)\right)\Bigg],
\end{equation}
and
\begin{equation}\label{energyG2}
\tilde{\E_2}(\mathcal{F})  = \sum_{s \in \La}\Bigg[\frac{1}{2}\left(|D_2\hat{\G}(m,s) + D\mathcal{F}(s)|^{2} - |D_2\hat{\G}(m,s)|^{2}\right)- \delta_{ms}\left(\hat{\G}(m,s) +\mathcal{F}(s)\right)\Bigg],
\end{equation}
As in the case of the crack problem itself, the crucial step is the correct choice of the predictor $\hat{\G}$, which ensures the minimisation problems are well-defined. This can be achieved by prescribing $\hat{\G}$ which, away from the point source is equal to $\hat{G}$, which satisfies the corresponding continuum problem, i.e. it solves, for $s \in \La$ fixed,
\begin{align}
 -C_{\La}\Delta_x \hat{G}(x,s) &= \delta(x-s)\quad\text{ for } x\in\R^2\setminus\Gamma_0 \label{Gpred-eq1}\\
 \nabla_x\hat{G}(x,s) \cdot \nu &= 0\quad\quad\quad\quad\text{ for } x\in\Gamma_0 \nonumber,
\end{align}
and, for $x \in \La$ fixed,
\begin{align}
-C_{\La}\Delta_s \hat{G}(x,s) &= \delta(x-s)\quad\text{ for } s\in\R^2\setminus\Gamma_0 \label{Gpred-eq2}\\
 \nabla_s\hat{G}(x,s) \cdot \nu &= 0\quad\quad\quad\quad\text{ for } s\in\Gamma_0 \nonumber.
\end{align}
Here $\delta$ represents the Dirac delta. We refer to Remark \ref{Clambda} for a discussion about the constant $C_{\La}$. {\cb Since $\omega$ introduced in \eqref{sqrt-map} is a conformal mapping, it preserves harmonicity \cite{ablo-fokas}. Further, it maps the crack domain to a half-space domain (cf. Figure \ref{fig:distorted_lattice}), for which there exists a standard explicit formula for a Green's function, it therefore can be verified that \eqref{Gpred-eq1} has a solution}
\begin{equation}\label{Gcont}
 \hat{G}(x,s) =  \frac{-1}{2\pi C_{\La}}\big[\log(|\omega(x)-\omega(s)|) + \log(|\omega(x)-\omega^*(s)|)\big],
\end{equation}
where $\omega^*(x)$ is defined as the reflection of $\omega(x)$ through vertical axis, that is
\[
\omega^*(x) = \left({\cb -}\sqrt{r_x}\cos{\left(\sfrac{\theta_x}{2}\right)}, \sqrt{r_x}\sin{\left(\sfrac{\theta_x}{2}\right)}\right),
\]
where we refer to Figure \ref{fig:distorted_lattice} for a visualisation. 

It is easy to see that 
\begin{equation}\label{hatGvarsym}
\hat{G}(x,s) = \hat{G}(s,x),
\end{equation}
since $|\omega(x)-\omega(s)||\omega(x)-\omega^*(s)| = |\omega(x)-\omega(s)||\omega^*(x)-\omega(s)|${\cb , thus $\hat{G}$ also solves \eqref{Gpred-eq2}.} It is worth recalling that the complex square root mapping is also used to construct $\hat{u}$.

Finally, bearing in mind that $\hat{G}(s,s)$ is not well-defined in the pointwise sense, the predictor $\hat{\G}\,\colon\,\La\times\La \to \R$ we prescribe is given by
\begin{equation}\label{Gpred}
\hat{\G}(m,s) := \begin{cases} \hat{G}(m,s)&\quad\text{if }m \neq s, \\
						0&\quad\text{if } m = s, \end{cases}
\end{equation}
since near the point-source it will always be true that $\G(s,s) \sim \mathcal{O}(1)$.

\subsubsection{Proof of Theorem \ref{thm::G}: existence of a Green's function}
We begin by investigating the predictor $\hat{\G}$ and estimate the decay of its derivatives of relevant order.
\begin{lemma}\label{D1D2hatG}
For any $x,s \in \R^2\setminus\Gamma_0$ with $x,s \neq 0$ and $x\neq s$, and $\alpha \in \{1,2,3,4\}$ 
\begin{align}\
|\nabla^{\alpha}_x\hat{G}(x,s)| &\lesssim (1+|\omega_x|^{2\alpha-1}|\omega^-_{xs}|)^{-1} + (1+|\omega_x|^{\alpha}|\omega^-_{xs}|^{\alpha})^{-1}\nonumber \\
&=:g^{(a)}_{\alpha}(x,s) + g^{(b)}_{\alpha}(x,s)\label{gagb}
\end{align}
and
\begin{align}
|\nabla^{\alpha}_x\nabla_{s}\hat{G}(x,s)| &\lesssim (1+|\omega_x|^{2\alpha-1}|\omega_s||\omega^-_{xs}|^{2})^{-1} + (1+|\omega_x|^{\alpha}|\omega_s||\omega^-_{xs}|^{\alpha+1})^{-1}\nonumber \\
&=:h^{(a)}_{\alpha}(x,s) + h^{(b)}_{\alpha}(x,s)\label{hahb}.
\end{align}
Consequently, if $m,\,s \in \La$ and $\rho \in \Rc(m)$ and $\sigma \in \Rc(s)$, then 
\[
|D_{1,\rho}D_{2,\sigma} \hat{\G}(m,s)| \lesssim h^{(a)}_{1}(m,s).
\]
\end{lemma}
\begin{proof}
We first notice that it is sufficient to estimate $L(x,s) := \log(|\omega_{xs}^-|)$, since the part of \eqref{Gcont} that includes $\omega^*(s)$ does not decay any slower. We calculate that
\[
\nabla_x L(x,s) = \frac{1}{|\omega_{xs}^-|^2}\nabla \omega(x) \omega_{xs}^- \implies |\nabla_x L(x,s)| \lesssim |\omega_{xs}^-|^{-1}|\nabla\omega(x)| \lesssim  |\omega_x|^{-1}|\omega_{xs}^-|^{-1}.
\]
Similarly
\[
\nabla^2_x L(x,s) = \frac{1}{|\omega_{xs}^-|^2}\left(\nabla^2\omega(x)[\omega_{xs}^-] + \nabla\omega(x)\cdot \nabla\omega(x)\right) - \frac{2}{|\omega_{xs}^-|^4}\left(\nabla\omega(x)\omega_{xs}^-\right)^{\otimes 2},
\]
which implies that
\begin{align*}
|\nabla^2_x L(x,s)| &\lesssim |\omega_{xs}^-|^{-1}|\nabla^2\omega(x)| + |\omega_{xs}^-|^{-2}|\nabla\omega(x)|^2 \\
&\lesssim |\omega_x|^{-3}|\omega_{xs}^-|^{-1} + |\omega_x|^{-2}|\omega_{xs}^-|^{-2}.
\end{align*}
For mixed derivatives we first calculate
\begin{align*}
&\nabla_s \nabla_x L(x,s) = \frac{2}{|\omega_{xs}^-| ^4}\left(\nabla\omega(x)\omega_{xs}^-\right) \otimes\left(\nabla \omega(s)\omega_{xs}^-\right) - \frac{1}{|\omega_{xs}^-|^2}\nabla\omega(x)\cdot\nabla\omega(s)\\
&\implies |\nabla_s \nabla_x L(x,s)|\lesssim |\omega_{xs}^-|^{-2}|\nabla\omega(x)||\nabla\omega(s)|\lesssim |\omega_x|^{-1}|\omega_s|^{-1}|\omega_{xs}^-|^{-2}
\end{align*}
and further realise that
\begin{align*}
\nabla_x^2\nabla_s L(x,s) &= \frac{-8}{|\omega_{xs}^-|^6}\left(\nabla\omega(x)\omega_{xs}^-\right)^{\otimes 2}\otimes \left(\nabla\omega(s)\omega_{xs}^-\right)\\
&+ \frac{2}{|\omega_{xs}^-|^4}\left(\nabla^2\omega(x)[\omega_{xs}^-] + \nabla\omega(x)\cdot \nabla\omega(x)\right)\otimes\left(\nabla\omega(s) \omega_{xs}^-\right)\\
&+ \frac{4}{|\omega_{xs}^-|^4}\left(\nabla\omega(x)\omega_{xs}^-\right)\otimes\left (\nabla\omega(s)\cdot\nabla\omega(x)\right) - \frac{1}{|\omega_{xs}^-|^2}\nabla^2\omega(x)[\nabla\omega(s)],
\end{align*}
which leads to
\[
|\nabla_x^2\nabla_s L(x,s)| \lesssim |\omega_x|^{-3}|\omega_s|^{-1}|\omega_{xs}^-|^{-2} + |\omega_x|^{-2}|\omega_s|^{-1}|\omega_{xs}^-|^{-3}.
\]
Remaining cases can be calculated along similar lines, but for the sake of brevity we choose to omit these tedious calculations. In particular, for $\alpha \geq 3$ there begin to appear extra terms corresponding to intermediate permutations of powers, but these can always be bounded by the two extreme permutations stated.

The facts that $|\Dm{\rho}\Ds{\sigma} \hat{\G}(m,s)| \lesssim |\nabla_m\nabla_s \hat{\G}(m,s)|$ and $h^{(a)}_{1} \equiv h^{(b)}_{1}$ conclude the proof.
\end{proof}

\begin{prop}\label{EG-well-def}
For any $s \in \La$ ($m \in \La$ respectively) the energy difference functional $\hat{\E}_1$ ($\hat{\E}_2$ resp.) in \eqref{energyG1} (\eqref{energyG2} resp.) is well-defined on $\Hcc$ and infinitely many times differentiable. 
\end{prop}
\begin{proof}
Here we will explicitly consider the part of the proof related to  $\tilde{E}_1$, as then the variable symmetry of the predictor, i.e. $\hat{\G}(m,s) = \hat{\G}(s,m)$, implies the other part.

For any $v \in \Hcc$ we can rewrite the energy difference functional $\tilde{\E}_1$ given by \eqref{energyG1} as
\[
\tilde{\E}_1(v) = \tilde{\E_0}(v) + \<{\delta \tilde{\E_1}(0)}{v},	
\]
where
\begin{align*}
\tilde{\E_0}(v) := \sum_{m \in \La}\Bigg[\frac{1}{2}|D_1\hat{\G}(m,s) + Dv(m)|^{2} - \frac{1}{2}|D_1\hat{\G}(m,s)|^{2}- D_1\hat{\G}(m,s)\cdot Dv(m) - \delta(m,s)\hat{\G}(m,s)\Bigg]
\end{align*}
and
\[
\<{\delta \tilde{\E}_1(0)}{v} = \sum_{m \in \La}\Big(D_1\hat{\G}(m,s)\cdot Dv(m) - \delta(m,s)v(m)\Big).
\]
{\cb Due to the quadratic nature of the energy, $\tilde{\E_0}$ reduces to $\tilde{\E_0}(v) = \frac{1}{2}\|v\|_{\Hcc}^2 - \hat{\G}(s,s)$ and thus is well-defined on $\Hcc$.}  For the second term, we aim to establish that $\delta\tilde{\E}_1(0)$ is a bounded linear functional on $\Hcc$ and to achieve that we use the fact that $\hat{G}$ solves the equation given by \eqref{Gpred-eq1} by applying the same interpolation construction as in Section \ref{P-u}. Consequently, we can thus write
\begin{align*}
\sum_{m\in\La}\delta_{ms}v(m) = v(s) &=  C_{\La}\int_{\R^2\setminus\Gamma_0} \nabla \hat{G}(x,s) \cdot \nabla Iv(x)\,dx,
\end{align*}
and in particular the second equality follows from the weak form of \eqref{Gpred-eq1} and the boundary term is not there due to the boundary condition. Mirroring the argument in Section \ref{P-u} we can conclude that
\begin{align*}
\<{\delta \tilde{\E}_1(0)}{v} &= \sum_{b(m,\rho)\not\subset\Gamma}\left(\Dm{\rho}\hat{\G}(m,s) - \int_{U_{m\rho}}\nabla_{\rho}\hat{G}(x,s)\,dx\right)D_{\rho}v(m)\\ 
&+ \sum_{b(m,\rho)\subset\Gamma}\left(\Dm{\rho}\hat{\G}(m,s) - \int_{U_{m\rho}}\nabla_{\rho}\hat{G}(x,s)\,dx - \int_{Q_{m\rho}}\nabla_{\rho}\hat{G}(x,s)\,dx\right)D_{\rho}v(m)\\
&+\sum_{b(m,\rho) \subset Q_0}f(m,s)D_{\rho}v(m),
\end{align*}
where this time $|f(m,s)|\lesssim |\omega_s|^{-1}$, as 
\[
\int_{B_1(0)\setminus\Gamma_0}|\nabla\hat{G}(x,s)|dx \lesssim |\omega_s|^{-1}\int_{0}^1|x|^{-1/2}dx \lesssim |\omega_s|^{-1}.
\]
Once again employing a Taylor expansion followed by a standard quadrature result results in
\[
b(m,\rho)\not\subset\Gamma \implies \left|\Dm{\rho}\hat{\G}(m,s) - \int_{U_{m\rho}}\nabla_{\rho}\hat{G}(x,s)\,dx\right| \lesssim |\nabla_x^3\hat{G}(x,s)|
\]
and
\[
b(m,\rho)\subset\Gamma\setminus Q_0 \implies \left|\Dm{\rho}\hat{\G}(m,s) - \int_{U_{m\rho}}\nabla_{\rho}\hat{G}(x,s)\,dx\right| \lesssim |\nabla_x^2\hat{G}(x,s)|.
\]
Finally, since Lemma \ref{D1D2hatG} implies that for both $j=2,3$ and  $x_0\approx 0$ we have $|\nabla_x^j\hat{G}(x_0,s)| \sim |\omega_s|^{-1}$, we can incorporate any bond $b(m,\rho)\subset Q_0$ into the general conclusion that
\[
\<{\delta \tilde{\E}_1(0)}{v} \lesssim \sum_{i=1}^4I_i(v),
\]
with
\begin{equation}\label{I12}
I_1(v) := \sum_{b(m,\rho)\not\subset\Gamma}g^{(a)}_{3}(m)|D_{\rho}v(m)|, \quad I_2(v) := \sum_{b(m,\rho)\not\subset\Gamma}g^{(b)}_{3}(m)|D_{\rho}v(m)| 
\end{equation}
and
\begin{equation}\label{I34}
I_3(v) := \sum_{b(m,\rho)\subset\Gamma}g^{(a)}_{2}(m)|D_{\rho}v(m)|,\quad I_4(v) := \sum_{b(m,\rho)\subset\Gamma}g^{(b)}_{2}(m)|D_{\rho}v(m)|,
\end{equation}
where $g^{(a)}_{\alpha}$ and $g^{(b)}_{\alpha}$ were defined in \eqref{gagb}.

Since $|\omega(m)| = |x|^{-1/2}$, we have $|g^{(a)}_{3}|,|g^{(b)}_{3}| \lesssim |m|^{-3/2}$, which is enough to conclude  $I_1(\cdot)$ and $I_2(\cdot)$ are bounded on $\Hcc$.

Similarly, $|g^{(a)}_{2}|,|g^{(b)}_{2}| \lesssim |m|^{-1}$ and thus $I_3(\cdot)$ and $I_4(\cdot)$ are bounded on $\Hcc$, {\cb since their domain of summation is one--dimensional}. Hence we can conclude that for any $v \in \Hcc$,
\[
\<{\delta \tilde{\E}_1(0)}{v} \lesssim \|Dv\|_{\ell^2}.
\]
{\cb To conclude $\tilde{\E}_1$ is $C^{\infty}$ we first note that $\frac{1}{2}|D\mathcal{F}(m)|^2 = \sum_{\rho\in\Rc(m)}\tilde{\phi}(D_{\rho}\mathcal{F}(m))$, where $\tilde{\phi}(r) := \frac{1}{2}r^2$ is a quadratic pair-potential, with $\tilde{\phi}^{(j)} \equiv 0$ for $j \geq 3$, thus, similarly as in the corresponding part of the proof of Theorem \ref{Eu-well-def}, the argument of \cite{2012-ARMA-cb} applies. Likewise, the first variation of the Kronecker delta term can be explicitly calculated and all subsequent ones vanish, thus giving the result.}
\end{proof}
\begin{lemma}\label{G-side-lemma}
For any $s \in \La$, the minimisation problem \eqref{minprobG} for $i=1$ has a unique solution $\tilde{\G}_1(\cdot,s) \in \Hcc$. Similarly, for any $m \in \La$, the minimisation problem \eqref{minprobG} for $i=2$ has a unique solution $\tilde{\G}_2(m,\cdot) \in \Hcc$.
\end{lemma}
\begin{proof}
The existence $\tilde{\G}_1(\cdot,s)$ and $\tilde{\G}_2(m,\cdot)$ is guaranteed by the {\cb quadratic nature of the energy $\tilde{\E}_i$, which implies that the problem of finding its critical point is linear}, thus allowing us to invoke the standard Lax-Milgram lemma. The minimisers satisfy
\begin{equation}\label{G_min_eq}
\<{\delta \tilde{\E}_1(\tilde{\G}_1)}{v} = 0\quad\text{and}\quad\<{\delta \tilde{\E}_2(\tilde{\G}_2)}{v} = 0\quad\forall v \in \Hcc ,
\end{equation}
where
\[
\<{\delta \tilde{\E}_1(\tilde{\G}_1)}{v} = \sum_{m \in \La} (D_{1}\hat{\G}(m,s) + D_{1}\tilde{\G}_1(m,s))\cdot Dv(m) \,-\delta_{ms}v(m),
\]
\[
\<{\delta \tilde{\E}_2(\tilde{\G}_2)}{v} = \sum_{s \in \La} (D_{2}\hat{\G}(m,s) + D_{2}\tilde{\G}_2(m,s))\cdot Dv(s) \,-\delta_{sm}v(s).
\]
\end{proof}
It can be readily checked that in fact $\tilde{\G}_2(m,s) = \tilde{\G}_1(s,m)$, in particular since the restriction in the definition of $\Hcc$ is satisfied, that is for $m \in \La$ we indeed have $\tilde{\G}_2(m,\hat{x}) = \tilde{\G}_1(\hat{x},m) = 0$. Thus we drop the subscripts and identify $\tilde{\G} \equiv \tilde{\G}_1$. In order to conclude the statement of Theorem \ref{thm::G}, it remains to show that $\tilde{\G}(m,s) = \tilde{\G}(s,m)$. In turns out, however, that this cannot be guaranteed without making a suitable adjustment that correctly takes into account the definition of $\Hcc$. The following weaker preliminary result is first obtained.
\begin{lemma}\label{G-side-lemma2}
For any $l,s \in \La$ and $\lambda \in \Rc(l)$, $\tau \in \Rc(s)$, the unique solution $\tilde{\G}$ from Lemma \ref{G-side-lemma} satisfies
\[
D_{1\lambda}D_{2\tau}\tilde{\G}(l,s) = D_{2\lambda}D_{1\tau}\tilde{\G}(s,l).
\]
\end{lemma}
\begin{proof}
Noting the first equation in \eqref{G_min_eq}, we can write
\[
D_{1\lambda}D_{2\tau}\tilde{\G}(l,s) = \sum_{m\in\La}\Dm{}\Ds{\tau}\tilde{\G}(m,s)\cdot\Dm{}\Ds{\lambda}(\hat{\G} + \tilde{\G})(m,l)
\]
and since Lemma \ref{D1D2hatG} ensures that $\left(\Ds{\lambda}\hat{\G}(\cdot,l) - \Ds{\lambda}\hat{\G}(\hat{x},l)\right) \in \Hcc$, we can split this infinite sum and write
\[
D_{1\lambda}D_{2\tau}\tilde{\G}(l,s) = A + B,
\]
where
\[
A := \sum_{m\in\La}\Dm{}\Ds{\tau}\tilde{\G}(m,s)\cdot\Dm{}\Ds{\lambda}\tilde{\G}(m,l)\;\text{ and }\; B := \sum_{m\in\La}\Dm{}\Ds{\tau}\tilde{\G}(m,s)\cdot\Dm{}\Ds{\lambda}\hat{\G}(m,l).
\]
{\cb Since $\delta\tilde{\E}_1$ is linear, one starts by observing that the first equation of \eqref{G_min_eq} implies, for all $v \in \Hcc$, 
\begin{equation}\label{gmineqd2}
0 = \<{\delta\tilde{\E}_1(\tilde{\G}(\cdot,s+\tau))}{v} -  \<{\delta\tilde{\E}_1(\tilde{\G}(\cdot,s))}{v} = \<{\delta\tilde{\E}_1(\Ds{\tau}\tilde{\G}(\cdot,s))}{v}.
\end{equation}
Treating $\left(\Ds{\lambda}\hat{\G}(\cdot,l) - \Ds{\lambda}\hat{\G}(\hat{x},l)\right) \in \Hcc$ as a test function in \eqref{gmineqd2}, we subtract this expression from $B$ to conclude that }
\[
B = \sum_{m\in\La}-\Dm{}\Ds{\tau}\hat{\G}(m,s)\cdot\Dm{}\Ds{\lambda}\hat{\G}(m,l) + \Dm{\tau}\Ds{\lambda}\hat{\G}(s,l).
\]
The same analysis can be employed to further conclude that
\[
D_{2\lambda}D_{1\tau}\tilde{\G}(s,l) = A + C,
\]
where $A$ as above and 
\[
C: = \sum_{m\in\La}\Dm{}\Ds{\lambda}\tilde{\G}(m,l)\cdot\Dm{}\Ds{\tau}\hat{\G}(m,s) = \sum_{m\in\La}-\Dm{}\Ds{\lambda}\hat{\G}(m,l)\cdot\Dm{}\Ds{\tau}\hat{\G}(m,s) + \Dm{\lambda}\Ds{\tau}\hat{\G}(l,s),
\]
where the final passage follows from applying the first equality in \eqref{G_min_eq}. Finally, noting that the variable symmetry of $\hat{\G}$ stated in \eqref{hatGvarsym} implies that
\[
\Dm{\tau}\Ds{\lambda}\hat{\G}(s,l) = \Dm{\lambda}\Ds{\tau}\hat{\G}(l,s),
\]
we can conclude that $B \equiv C$, thus establishing the result. 
\end{proof}
We are now in a position to prove the main result of this section. 
\begin{proof}[Proof of Theorem \ref{thm::G}: existence of a Green's function]
Starting with the equality established in Lemma \ref{G-side-lemma2}, we can apply the indefinite sum operator (discrete anologue of indefinite integration, cf. \cite{jordan}) in the second variable to conclude that  
\[
D_{1\lambda}\tilde{\G}(m,s) = D_{2\lambda}\tilde{\G}(s,m) + f_{\lambda}(m),
\]
for some lattice function $f_{\lambda}\,\colon\,\La \to \R$. Similarly, applying indefinite sum operator in the first variable implies that
\begin{equation}\label{G-sym-rel1a}
\tilde{\G}(m,s) = \tilde{\G}(s,m) + F(m) + K_1(s),
\end{equation}
where $D_{\tau}F(m) = f_{\lambda}(m)$ and $K_1$ a is a lattice function to be determined and originating from indefinite summation. We can repeat the procedure in the reverse order to obtain
\[
D_{2\tau}\tilde{\G}(m,s) = D_{2\tau}\tilde{\G}(s,m) + k_{\tau}(s).
\]
Taking $\tau = \lambda$ and exchanging $m$ and $s$ we obtain that for any lattice direction $\lambda$ we have $k_{\lambda}(m) = - f_{\lambda}(m)$. Indefinitely summing one more time results in  
\begin{equation}\label{G-sym-rel1b}
\tilde{\G}(m,s) = \tilde{\G}(s,m) -F(s) + K_2(m).
\end{equation}
Comparing \eqref{G-sym-rel1a} and \eqref{G-sym-rel1b} we conclude that $K_1(s) = -F(s)$ and $K_2(m) = F(m)$ and thus 
\begin{equation}\label{G-sym-rel1}
\tilde{\G}(m,s) = \tilde{\G}(s,m) + F(m) - F(s),
\end{equation}
for some $F$ that arises from the restriction in the definition of the energy space $\Hcc$. By adding and subtracting the same constant we can in fact also write that
\[
\tilde{\G}(m,s) + F_2(s) = \tilde{\G}(s,m) + F_2(m),
\]
where $F_2(m) := F(m) - F(\hat{x})$, which conveniently implies that $F_2(\hat{x}) = 0$. We now let $m = \hat{x}$ and realise that
\[
F_2(s) = \tilde{\G}(s,\hat{x}).
\]
Thus the actual relation is given by
\begin{equation}\label{gbargbarbar}
\tilde{\G}(m,s) + \tilde{\G}(s,\hat{x}) = \tilde{\G}(s,m) + \tilde{\G}(m,\hat{x})
\end{equation}
and we can conclude the proof by stating that the atomistic correction $\bar{\G}$ we sought is given by $\bar{\G}(m,s)= \tilde{\G}(m,s) + \tilde{\G}(s,\hat{x})$, as it clearly satisfies both equations in \eqref{G_min_eq} and in addition $\bar{\G}(m,s) = \bar{\G}(s,m)$.
\end{proof}
\subsubsection{Proof of Theorem \ref{thm::G}: Green's function decay estimate}\label{P-G2}
The decay of $|\Dm{}\Ds{}\hat{G}|$ is explicitly calculated in Lemma \ref{D1D2hatG}, thus we turn our attention to the decay of the corrector $\bar{\G}$. The general approach we employ is to get insight in the decay behaviour of $\bar{\G}$ in different regions on $\La$. For a fixed $s \in \La$ with $|s|$ large enough, we carve the lattice into three regions:
\[
\Omega_1(s) := B_{\sfrac{|s|}{2}}(0)\cap \La,\quad \mathcal{A}(s) := \left(B_{\sfrac{3|s|}{2}}(0)\setminus B_{\sfrac{|s|}{2}}(0)\right) \cap \La,\quad\Omega_2(s) := \left(\R^2 \setminus B_{\sfrac{3|s|}{2}}(0)\right)\cap \La.
\]
\begin{figure}[!htbp]
\centering
\includegraphics[width=0.65\linewidth]{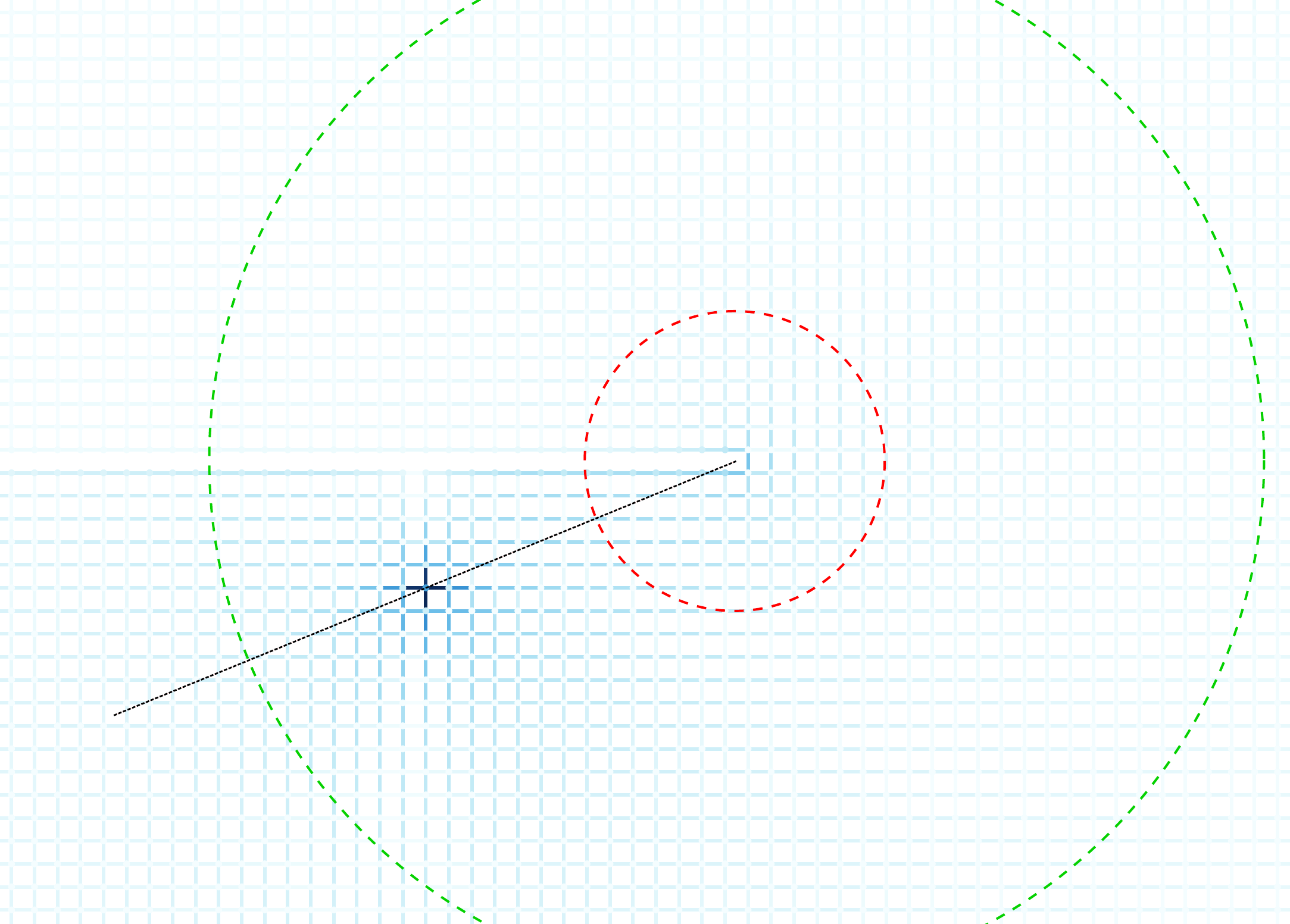}
\caption{The lattice with point-source $s$ depicted in dark blue and $\Omega_1(s)$ being the inner ball with red boundary, $\Omega_2(s)$ the outer region with green boundary and $\mathcal{A}(s)$ the annulus in-between.} 
\label{fig:new}
\end{figure}

In the following we will extensively use the fact that locally the defective lattice does not differ from a homogeneous lattice and thus the result from the spatially homogeneous setup apply, as long as we introduce suitable cut-offs. The general idea behind the cut-off function $\eta\,\colon\,\R^2 \to \R$ to be used throughout is as follows. We define it as $\eta(x) := \hat{\eta}(\sfrac{|x-\hat{x}|}{R})$, where $\hat{\eta}\,\colon\,\R \to \R$ is such that $\hat{\eta}(x) =  1$ for $x\in[0,c_1]$, $\eta(x) = 0$ for $x > c_2$, and smooth and decreasing inbetween. As a result $D\eta$ will only be non-zero on an annulus that scales like $R$. It is also clear, by Taylor expansion, that $|D^j \eta(x)| \lesssim R^{-j}$. The radius $R$, the lattice point $\hat{x}$, and constants $c_1 < c_2$ will be chosen as needed.

Finally, we also recall the existence and the decay of the homogeneous lattice Green's function $\G^{\rm hom}$ corresponding to the homogenous hessian operator $\tilde{H}$:
\[
\tilde{H}u(m) := \divo \tilde{D}u(m),	
\]
where $\tilde{D}u(m) := \left(D_{\rho}u(m)\right)_{\rho \in\Rc}$, i.e. we always use full stencils.  It is proven in \cite{EOS2016} in a much more general setup that there exists $\G^{\rm hom}\,\colon \, \La \to \R$ such that
\[
\tilde{H}\G^{\rm hom}(m-l) = \delta(m,l)\quad\forall m,l\in\La
\]
and
\begin{equation}\label{Ghomdecay}
|D^j\G^{\rm hom}(m-l)| \lesssim (1+|m-l|^{j})^{-1}.	
\end{equation}
With these tools in hand we can gain preliminary insight into the decay behaviour of $\bar{\G}$, however the appearance of the cut-off function restricts us to a suboptimal result. We proceed in steps, starting with the following. 
\begin{lemma}\label{D1-estt}
If $l \in \La\setminus\Omega_1(s)$ and  $\tau \in \Rc(l)$, then 
\[
|\Dm{\tau}\bar{\G}(l,s)| \lesssim (1+|\omega_l||\omega^-_{ls}|)^{-1}.
\]
\end{lemma}
\begin{proof}
Due to the spatial restriction on $l$, we can always choose $\hat{x} = l$, $R = |\omega(l)||\omega(l)-\omega(s)|$ with $c_1$ and $c_2$ such that the support of the cut-off function $\eta$ does not reach the origin, e.g.  $c_1 = \frac{1}{12}$, $c_2 = \frac{1}{6}$. This is true because $|\omega(l)| = |l|^{1/2}$ and trivially $|\omega(l)-\omega(s)| \leq |\omega(l)| + |\omega(s)| \leq (1+\sqrt{2})|\omega(l)|$.

{\cb We distinguish two cases and deal with them separately. The distinction is motivated by the fact that the first case concerns the region away from the crack and translates almost verbatim to vectorial models on an arbitrary Bravais lattice with a finite-range interatomic potential. On the other hand, in Case 2 we handle the near-crack region with an argument that heavily exploits the setting of a scalar anti-plane model posed on a square lattice under a nearest-neighbour pair-potential.}  \\

\paragraph{Case 1: $\supp \eta \cap \Gamma = \emptyset$.}\mbox{}\\
With the support of cut-off function not crossing the crack surface, we can directly write
\begin{align}
\Dm{\tau}\bar{\G}(l,s) &= \Dm{\tau}[\bar{\G}(l,s)\eta(l)] = \sum_{m\in\La} \tilde{H}D_{\tau}\G^{\rm hom}(m-l)]\bar{\G}(m,s)\eta(m) \nonumber \\ \label{D1-simple}&= \sum_{m\in\La}\sum_{\rho\in\Rc}D_{\rho}D_{\tau}\G^{\rm hom}(m-l)\Dm{\rho}[\bar{\G}(m,s)\eta(m)],
\end{align}
where the first equality is due to the fact that near $l$ the cut-off is just 1, the second follows from the definition of the homogeneous lattice Green's function and the fact that with the cut-off in place we effectively sum over a finite region, where the there is no disparity between $H$ and $\tilde{H}$. The last equality is just summation by parts.

In order to use the equation that $\bar{\G}$ satisfies, we need to push the cut-off onto the other term by exploiting the discrete product rule. It leads to
\[
D_{1,\tau}\bar{\G}(l,s) = S_1 + S_2,	
\]
where the first term is in the form allowing us to exploit the equation, namely
\[
S_1 = \sum_{m\in\La}\sum_{\rho\in\Rc}D_{\rho}[D_{\tau}\G^{\rm hom}(m-l)\eta(m)]\Dm{\rho}\bar{\G}(m,s)	
\]
and the second term makes sure that the right-hand side  is consistent with \eqref{D1-simple}, that is 
\begin{align*}
S_2 = \sum_{m\in\La}\sum_{\rho\in\Rc} D_{\rho}\eta(m)&\Bigg[A_{\rho}D_{\tau}\G^{\rm hom}(m-l)\Dm{\rho}\bar{\G}(m,s)\\ &+ D_{\rho}D_{\tau}\G^{\rm hom}(m-l)A_{1\rho}\bar{\G}(m,s)\Bigg].
\end{align*}
Here $A_{\rho}f(m) := \frac{1}{2}\left(f(m+\rho)+ f(m)\right)$ {\cb and $A_{1\rho}f(m,s):= \frac{1}{2}\left(f(m+\rho,s) + f(m,s)\right)$.}

We deal with both terms separately. For $S_1$ we realise that 
\begin{equation}\label{vgheta}
v(m) := D_{\tau}\G^{\rm hom}(m-l)\eta(m)
\end{equation}
is in fact compactly-supported, so is an admissible test function in the energy space $\Hcc$. In particular, it satisfies $|D_{\rho}v(m)| \lesssim |m-l|^{-2} = |\omega_{ml}^-|^{-2}|\omega_{ml}^+|^{-2}$ (relation established in \eqref{csqrt-id}) and for $m\in B_{c_2R}(l)$ (which is equal to $\supp \eta$),  we trivially have that $|\omega_{ml}^-| \leq |\omega_{ml}^+|$ and $|\omega_m| \leq |\omega_{ml}^+|$, thus, we can say that either $|m-l|^{-2} \lesssim |\omega_{ml}^-|^{-4}$ or $|m-l|^{-2} \lesssim |\omega_{ml}^-|^{-2}|\omega_m|^{-2}$. Exploiting the fact that $\bar{\G}$ satisfies \eqref{G_min_eq} we conclude that 
\[
|S_1| = \left|-\<{\delta \tilde{E}_1(0)}{v}\right| \lesssim \sum_{i=1}^4 I_i(v),
\]
where $I_1,\dots,I_4$ as in \eqref{I12} and \eqref{I34}.

We look at each term separately and begin by noting that
\[
I_1(v) \lesssim \sum_{m \in B_{c_2R}(l)}(1+|\omega_m|^{5}|\omega^-_{ms}|)^{-1}(1+|\omega^-_{ml}|^{2}|\omega_m|^{2})^{-1} =: \sum_{m \in B_{c_2R}(l)}f_1(m) 
\]
and observe that away from the potential sharp spikes at $m=l$ and $m=s$ we can bound this series by the corresponding integral, that is we can say
\[
I_1(v) \lesssim f_1(l) + f_1(s)\mathbbm{1}_{B_{c_2R}(l)}(s) + \int_{D_{R}(l)}f_1(x)dx,
\]
where $D_{R}(l):=B_{c_2R}(l)\setminus (B_1(l)\cup B_1(s))$. The indicator function $\mathbbm{1}$ covers cases when $s\not\in B_{c_2R}(l)$. Clearly  
\[
f_1(l) = (1+|\omega_l|^{5}|\omega_{ls}^-|)^{-1},
\]
whereas $f_1(s)\mathbbm{1}_{B_{c_2R}(l)}(s) \neq 0$ only if $s \in B_{c_2R}(l)$, but then $|s| \sim |l|$, which implies 
\begin{equation}\label{s-sim-l}
f_1(s)\mathbbm{1}_{B_{c_2R}(l)}(s) \lesssim (1+|\omega_l|^2|\omega_{ls}^-|^2)^{-1}.
\end{equation}
For the integral term we introduce a change of variables $\xi = \omega(m)$, and set $\gamma := \omega_s$, $\zeta:= \omega_l$, leading to $dm = |\xi|^2d\xi$. As a result, we have
\[
\int_{D_{R}(l)}f_1(x)dx = \int_{\omega(D_{R}(l))}\frac{|\xi|^{2}}{(1+|\xi|^{5}|\xi -\gamma|)(1+|\xi-\zeta|^{4})} d\xi =: \int_{\omega(D_{R}(l))}\tilde{f}_1(\xi)d\xi.
\]
Carving the region of integration into 
\begin{equation}\label{regions1}
\Omega_{\gamma} := B_{\frac{|\gamma-\zeta|}{2}}(\gamma) \cap \omega(D_{R}(l)),\quad \Omega_{\zeta} := B_{\frac{|\gamma-\zeta|}{2}}(\zeta)\cap \omega(D_{R}(l))\quad\text{and}\quad \Omega' := \omega(D_{R}(l)) \setminus (\Omega_{\gamma} \cup \Omega_{\zeta})
\end{equation}
and noting that depending on where $l$ and $s$ are, some of them could be empty, we can estimate the integral as follows. First we notice the following spatial relations
\begin{equation}\label{space-rel1}
\xi \in \omega(D_{R}(l)) \implies |\xi| \sim |\zeta|,\quad \xi \in \Omega_{\gamma} \implies |\xi-\zeta| \gtrsim |\zeta-\gamma|\,\text{ and } \xi \in \Omega_{\zeta} \implies |\xi-\gamma| \gtrsim |\zeta-\gamma|.
\end{equation}
Thus{\cb
\begin{align*}
\int_{\Omega_{\gamma}} \tilde{f}_1 d\xi &\lesssim \frac{|\zeta|^2}{1+|\zeta-\gamma|^4} \int_{\Omega_{\gamma}}\frac{1}{1 + |\zeta|^5|\xi-\gamma|}d\xi \lesssim \frac{|\zeta|^2}{1+|\zeta-\gamma|^4}\int_{{\cb \frac{1}{|\zeta|}}}^{\frac{|\zeta-\gamma|}{2}} \frac{r}{1+|\zeta|^5 r}dr \\
&\lesssim (1+|\zeta|^3|\zeta-\gamma|^3)^{-1},
\end{align*}
where the lower limit of integration in the third integral follows from the fact that we exclude a unit ball around $|l|$ in the original domain. The same reasoning was used to estimate \eqref{fomegazeta1}.}
Likewise,
\begin{equation}\label{int-est1}
\int_{\Omega_{\zeta}} \tilde{f}_1 d\xi \lesssim \frac{|\zeta|^2}{1+|\zeta|^5|\zeta-\gamma|} \int_{\Omega_{\zeta}}\frac{1}{1 +|\zeta-\gamma|^4}d\xi \lesssim (1+|\zeta|^3|\zeta-\gamma|)^{-1}
\end{equation}
and
\[
\int_{\Omega'} \tilde{f}_1 d\xi \lesssim \frac{1}{(1+|\zeta|^5|\zeta-\gamma|)(1+|\zeta-\gamma|^4)} \int_{\Omega'}|\xi|^2d\xi \lesssim (1+|\zeta|^3|\zeta-\gamma|^{3})^{-1},
\]
where the final passage relies on the fact that we can map back to $B_{c_2R}(l)$ and have a volume term that scales like $R^2 = |\zeta|^2|\zeta-\gamma|^2$. 

For $I_2(v)$ {\cb(recalling its definition from \eqref{I12})}, similarly,
\[
I_2(v) \lesssim \sum_{m \in B_{c_2R}(l)}(1+|\omega_m|^{3}|\omega^-_{ms}|)^{-3}(1+|\omega^-_{ml}|^{2}|\omega_m|^{2})^{-1} =: \sum_{m \in B_{c_2R}(l)}f_2(m)
\]
and hence
\[
I_2(v) \lesssim f_2(l) + f_2(s)\mathbbm{1}_{B_{c_2R}(l)}(s) + \int_{D_{R}(l)}f_2(x)dx.
\]
We observe that, due to the same reasoning as in \eqref{s-sim-l}, we have 
\begin{equation}\label{int-est2}
f_2(l) = (1+|\omega_l|^{3}|\omega_{ls}^-|^3)^{-1}\quad\text{and}\quad f_2(s)\mathbbm{1}_{B_{c_2R}(l)}(s) \lesssim (1+|\omega_l|^2|\omega_{ls}^-|^2)^{-1}.
\end{equation}
Furthermore, 
\[
\int_{D_{R}(l)}f_2(x)dx, = \int_{\omega(D_{R}(l))}\frac{|\xi|^{2}}{(1+|\xi|^{3}|\xi -\gamma|^{3})(1+|\xi|^2|\xi-\zeta|^{2})} d\xi =: \int_{D_R(l))}\tilde{f}_2(\xi)d\xi,
\]
with estimates, again arising from the spatial relations established in \eqref{space-rel1},
\begin{align}
\int_{\Omega_{\gamma}} \tilde{f}_2 d\xi &\lesssim \frac{|\zeta|^2}{1+|\zeta|^2|\zeta-\gamma|^2} \int_{\Omega_{\gamma}}\frac{1}{1 + |\zeta|^3|\xi-\gamma|^3}d\xi \nonumber\\  &\lesssim \frac{|\zeta|^2}{1+|\zeta|^{2}|\zeta-\gamma|^2}\int_{{\cb \frac{1}{|\gamma|}}}^{\frac{|\zeta-\gamma|}{2}} \frac{r}{1+|\zeta|^3 r^3}dr \lesssim (1+|\zeta|^{\cb 2}|\zeta-\gamma|^2)^{-1},\label{int-est3}
\end{align}
\[
\int_{\Omega_{\zeta}} \tilde{f}_2 d\xi \lesssim \frac{|\zeta|^2}{1+|\zeta|^3|\zeta-\gamma|^3} \int_{\Omega_{\zeta}}\frac{1}{1 +|\zeta|^{2}|\xi-\zeta|^2}d\xi \lesssim (1+|\zeta|^3|\zeta-\gamma|^3)^{-1}\log|\zeta-\gamma|
\]
and
\[
\int_{\Omega'} \tilde{f}_2 d\xi \lesssim \frac{1}{(1+|\zeta|^3|\zeta-\gamma|^3)(1+|\zeta|^2|\zeta-\gamma|^2)} \int_{\Omega'}|\xi|^2d\xi \lesssim (1+|\zeta|^3|\zeta-\gamma|^{3})^{-1},
\]
Finally, since for now we assume that $\supp \eta \cap \Gamma = \emptyset$, we trivially have that $I_3(v) = I_4(v) = 0$. It can be thus concluded that $S_1 \lesssim (1+|\zeta|^2|\zeta-\gamma|)^{-1} = (1+|\omega_l|^2|\omega_{ls}^-|)^{-1}$, with the exponents taken from combining {\cb\eqref{int-est1}, \eqref{int-est2} \& \eqref{int-est3}}. 

For $S_2$ we realise that $D_{\rho}\eta(m)$ is only non-zero for $m \in \mathcal{A}_{R} := B_{c_2 R}(l) \setminus B_{c_1 R}(l)$, which corresponds to a volume term that scales like $R^2 = |\omega_l|^{2}|\omega^-_{ls}|^{2}$. It also in particular implies that $|m-l|$ and $|\omega_l||\omega^-_{ls}|$ are comparable. We can thus use Cauchy-Schwarz inequality and the decay of each term to conclude that
\begin{align*}
|S_2| &\lesssim \left(|\omega_l|^{(-4+2)} |\omega^-_{ls}|^{(-4+2)} \right)^{1/2}\|D\bar{\G}(\cdot,s)\|_{\ell^2} + \left(|\omega_l|^{(-6+2)}|\omega^-_{ls}|^{(-6+2)}\right)^{1/2}\|A\bar{\G}(\cdot,s)\|_{\ell^2(\mathcal{A}_R)}\\
&\lesssim (1+|\omega_l||\omega^-_{ls}|)^{-1},
\end{align*}
where the last inequality is due to $\|A\bar{\G}(\cdot,s)\|_{\ell^2(\mathcal{A}_R)} \lesssim R\|D\bar{\G}(\cdot,s)\|_{\ell^2}$, a result that immediately follows from \cite[Lemma 7.1]{EOS2016pp}.\\

\paragraph{Case 2: $\supp \eta \cap \Gamma \neq \emptyset$.}\mbox{}\\
To cover this more problematic case we resort to a technical trick at present only seems to be applicable to a square lattice with NN interactions. We begin by constructing a discrete equivalent of a Riemann surface corresponding to the complex square root map, namely we define
\begin{equation}\label{manifoldM}
\mathcal{M}:= \Z^2 \times \{-1, 1 \},
\end{equation}
that is we look at two copies of the square lattice and so $k \in \mathcal{M}$ is such that $k = (k_l,k_b)$, where $k_l$ corresponds to a lattice site and $k_b$ determines whether we are on the positive branch or the negative branch (as with the complex square root mapping). For $\uc \,\colon\,\Mc\to \R$ and a lattice direction $\rho \in \Rc$, we also define the notion of a finite difference $\Dc_{\rho}$ and of a swapping finite difference $\Dc^s_{\rho}$ as
\[
\Dc_{\rho}\uc(k) := \uc(k_l + \rho,k_b) - \uc(k_l,k_b)\quad\text{and}\quad \Dc^s_{\rho}\uc(k) := \uc(k_l + \rho,-k_b) - \uc(k_l,k_b).
\]
Since $k_b \in \{-1,1\}$, we note that in the latter case we simply jump from one branch to another. The corresponding manifold discrete gradient operator as $\Dc\uc(k) \in \R^{\Rc}$ can then be defined as
\begin{equation}\label{dgrad-manifold}
\left(\Dc\uc(k)\right)_{\rho} = \begin{cases}
\Dc_{\rho}\uc(k)\quad &\text{if }\rho \in \Rc(k_l),\\
\Dc^s_{\rho}\uc(k)\quad&\text{if }\rho \not\in\Rc(k_l).
\end{cases}
\end{equation}
Comparing this with the definition of the discrete gradient in \eqref{dgrad}, we observe that the they only differ at lattice points on $\Gamma_+ \cup \Gamma_-$. This underlines the reasoning behind the construction - we take two copies of the lattice and glue them together at the cut, thus ensuring that in fact we always work with full stencils. Consequently, we can again locally use the homogeneous lattice Green's function $\G^{\rm hom}$, as long as we avoid the origin of $\Mc$. 

We further define the manifold equivalent of \eqref{Hcc} as
\begin{equation}\label{Hcc-manifold}
\Hcc_{\Mc} := \left\{\uc\,\colon\,\Mc \to \R \;|\; \Dc\uc \in \ell^2\;\text{and}\; \uc(\hat{x}, \pm 1) = 0 \right\}.
\end{equation}

Likewise, we can extend the notion of the predictor $\hat{\G}$ defined in \eqref{Gpred} to the manifold setup by defining $\hat{\G}_{\Mc}\,\colon\,\Mc\times\Z^2 \to \R$ as 
\[
\hat{\G}_{\Mc}(k,s) := \begin{cases}
\hat{\G}(k_l,s)\quad &\text{if }k_b = 1,\\
\hat{\G}((k_{l_1},-k_{l_2}),s)\quad &\text{if }k_b = -1,
\end{cases}
\]
that is, for the negative branch, we reflect the original predictor along $x$-axis. Note that the manifold finite difference operators are always applied with respect to the first variable. Finally, we can also consider a manifold equivalent of the energy-difference $\tilde{\E}_1$ defined in \eqref{energyG1}, which we define as
\begin{align}\label{energyG1-manifold}
\tilde{\E}_{\Mc}(\G_{\Mc})  = \sum_{k \in \Mc}\Bigg[&\frac{1}{2}\Big(\sum_{\rho \in \Rc}(\Dc\hat{\G}_{\Mc}(k,s))_{\rho} + (\Dc\G_{\Mc}(k,s))_{\rho})^{2} - (\Dc\hat{\G}_{\Mc}(k,s))_{\rho}^{2}\Big)\nonumber\\ 
&- \left(\delta(k,(s,1))+\delta(k,((s_1,-s_2),-1))\right)\left(\hat{\G}_{\Mc}(m,s) +\G_{\Mc}(m,s)\right)\Bigg].
\end{align}
{\cb It follows immediately from Proposition \ref{EG-well-def} that $\tilde{\E}_{\Mc}$ is well-defined over $\Hcc_{\Mc}$ and smooth.} Thus we can again look at the problem of finding a stationary point $\bar{\G}_{\Mc}$ which satisfies
\begin{equation}\label{G_min_eq-manifold}
\<{\delta \tilde{E}_{\Mc}(\bar{\G}_{\Mc})}{\uc} = 0\quad\forall \uc \in \Hcc_{\Mc} ,
\end{equation}
where
\begin{align*}
\<{\delta \tilde{E}_{\Mc}(\bar{\G}_{\Mc})}{\uc} = \sum_{k \in \Mc} \Bigg[\sum_{\rho \in \Rc} &\left(\Dc_{\rho}\hat{\G}_{\Mc}(k,s) + \Dc_{\rho}\bar{\G}_{\Mc}(k,s)\right)\Dc_{\rho}\uc(m)\\
&-\left(\delta(k,(s,1)) + \delta(k,((s_1,-s_2),-1))\right)\uc(m) \Bigg].
\end{align*}
Crucially, the way we define $\hat{\G}_{\Mc}$ implies that the contribution from the new bonds across $\Gamma$ is null, as e.g. for $l \in \Gamma_-$ we have $l+e_2 = (l_1,-l_2)$ and thus $\Dc^s_{e_2}\hat{\G}_{\Mc}((l,1),s) = 0$. This in turn tells us the solution to \eqref{G_min_eq-manifold} is given by
\[
\bar{\G}_{\Mc}(k,s) := \begin{cases}
\bar{\G}(k_l,s)\quad &\text{if }k_b = 1,\\
\bar{\G}((k_{l_1},-k_{l_2}),s)\quad &\text{if }k_b = -1,
\end{cases}
\] 
Thus to obtain the decay estimate for $|\Dm{\tau}\bar{\G}(l,s)|$, we proceed as follows. Without loss of generality we can assume that $l_2 < 0$ and accordingly define a reflected version of $\G = \hat{\G} + \bar{\G}$ as $\G_{\rm ref}\,\colon\,\La \to \R$ with 
\[
\G_{\rm ref}(m,s) := \begin{cases} \G(m,s)\quad &\text{if }m_2 < 0,\\
\G((m_1,-m_2,s)\quad &\text{if }m_2 > 0.
\end{cases}
\]
Hence, we can write
\begin{align*}
\Dm{\tau}\bar{\G}(l,s) &= \Dm{\tau}[\bar{\G}_{\rm ref}(l,s)\eta(l)] = \sum_{m\in\La} \tilde{H}D_{\tau}\G^{\rm hom}(m-l)]\bar{\G}_{\rm ref}(m,s)\eta(m) \\ 
&= \sum_{m\in\La}\sum_{\rho\in\Rc}D_{\rho}D_{\tau}\G^{\rm hom}(m-l)\Dm{\rho}[\bar{\G}_{\rm ref}(m,s)\eta(m)] = S_1 + S_2,
\end{align*}
with
\[
S_1 = \sum_{m\in\La}\sum_{\rho\in\Rc}D_{\rho}[D_{\tau}\G^{\rm hom}(m-l)\eta(m)]\Dm{\rho}\bar{\G}_{\rm ref}(m,s),
\]
and
\begin{align*}
S_2 = \sum_{m\in\La}\sum_{\rho\in\Rc} D_{\rho}\eta(m)&\Bigg[A_{\rho}D_{\tau}\G^{\rm hom}(m-l)\Dm{\rho}\bar{\G}_{\rm ref}(m,s)\\ &+ D_{\rho}D_{\tau}\G^{\rm hom}(m-l)A_{1\rho}\bar{\G}_{\rm ref}(m,s)\Bigg].
\end{align*}

Noting that the nullity of bonds across the $x$-axis of $\bar{\G}_{\rm ref}$ due to reflection ensures that  $\|D_1\bar{\G}_{\rm ref}(\cdot,s)\|_{\ell^2} < \infty$, the argument for $S_2$ is unaffected, thus we can immediately conclude that $S_2 \lesssim (1+|\omega_l||\omega_{ls}^-|)^{-1}$. For $S_1$ we recall the definition of $v$ in \eqref{vgheta} and define its manifold equivalent $v_{\Mc}\,\colon\,\Mc \to \R$ by
\[
v_{\Mc}(k) := \begin{cases} v(k_l)\quad&\text{if } (k_{l_2} < 0\ \land k_b = 1) \lor (k_{l_2} > 0 \land k_b = -1),\\
0\quad&\text{otherwise}.
\end{cases}
\]
As a result we have
\[
S_1 = \sum_{k \in \Mc}\Dc\bar{\G}_{\Mc}(k,s)\cdot\Dc v_{\Mc}(k)
\]
and thus we can exploit \eqref{G_min_eq-manifold} to conclude that
\[
S_1 = \sum_{k \in \Mc} -\Dc\hat{\G}_{\Mc}(k,s)\cdot \Dc v_{\Mc}(k) + v_{\Mc}((s,1)) + v_{\Mc}((s,-1)).
\]
We can now introduce
\[
\hat{\G}_+(m,s) := \begin{cases} \hat{\G}_{\Mc}((m,1),s)\;&\text{if } m_2 < 0, \\ 
0\;&\text{if } m_2 > 0,\end{cases}\quad \hat{\G}_-(m,s) := \begin{cases} \hat{\G}_{\Mc}((m,-1),s)\;&\text{if } m_2 > 0, \\ 
0\;&\text{if } m_2 < 0,\end{cases}
\]
\begin{equation}\label{vplus}
v_+(m) := \begin{cases} v(m)\;&\text{if }m_2 < 0, \\ 0\;&\text{if }m_2 > 0, \end{cases}\quad v_-(m) := \begin{cases} v(m)\;&\text{if }m_2 > 0, \\ 0\;&\text{if }m_2 < 0. \end{cases}
\end{equation}
and are able to conclude that in fact 
\begin{align*}
S_1 &= \left(\sum_{m \in \La}-D_1\hat{\G}_+(m,s)\cdot Dv_+(m) + v_+(s)\right) + \left(\sum_{m \in \La}-D_1\hat{\G}_-(m,s)\cdot Dv_-(m) + v_-((s_1,-s_2))\right)\\
&=: S_+ + S_-,
\end{align*}
i.e. we look at the positive and negative branch separately. The results hence follow from the fact that due to reflection we always have
\[
|D_1\hat{\G}_-(m,s)\cdot Dv_-(m)| \leq |D_1\hat{\G}_+((m_1,-m_2),s) \cdot Dv_+((m_1,-m_2))|,
\]
and consequently any estimate that applies to $|S_+|$ equally applies to $|S_-|$. Furthermore, $|S_+|$ can be estimated as in Case 1, except now $I_3(v_+),I_4(v_+) \neq 0$, but we can estimate them as follows.

With $\Gamma_l:= B_{c_2R}(l) \cap \Gamma$, we first note that
\[
I_3(v_+) \lesssim \sum_{m\in\Gamma_l}(1+|\omega_m|^3|\omega_{ms}^-|)^{-1}(1+|\omega_{ml}^-|^2|\omega_{ml}^+|^2)^{-1} =: \sum_{m\in\Gamma_l} f_3(m)
\]
and again argue that
\[
I_3(v_+) \lesssim f_3(l)\mathbbm{1}_{\Gamma_l}(l) + f_3(s)\mathbbm{1}_{\Gamma_l}(s) + \int_{\tilde{\Gamma}_{l}}f_3(x)dx,
\]
where $\tilde{\Gamma}_{l}:=\Gamma_l \setminus (B_1(l)\cup B_1(s))$. The indicator function $\mathbbm{1}$ is there to cover cases when $l,s\not\in \Gamma_l$. It is clear that 
\[
f_3(l) = (1+|\omega_l|^{3}|\omega_{ls}^-|)^{-1}\quad\text{and}\quad f_3(s)\mathbbm{1}_{\Gamma_l}(s) \lesssim (1+|\omega_l|^2|\omega_{ls}^-|^2)^{-1},
\]
where in particular for the second inequality we argue as in \eqref{s-sim-l}. 
For the integral term we again introduce a change of variables $\xi = \omega(m)$ and due to $\omega(\Gamma_l)$ being a one-dimensional line, we can conclude that $dm \lesssim |\xi|d\xi$. As a result, we have
\[
\int_{\tilde{\Gamma}_{l}}f_3(x)dx =  \int_{\omega(\tilde{\Gamma}_{l})}\frac{|\xi|}{(1+|\xi|^{3}|\xi-\gamma|)(1+|\xi-\zeta|^2||\xi+\zeta|^2)}d\xi =: \int_{\omega(\tilde{\Gamma}_{l})}\tilde{f}_3(\xi)d\xi.
\]
Mimicking the approach for $I_1(v)$ and $I_2(v)$, we carve the region of integration into 
\begin{equation}\label{regions2}
\Gamma_{\gamma} := B_{\frac{|\gamma-\zeta|}{2}}(\gamma) \cap \omega(\tilde{\Gamma}_l),\quad \Gamma_{\zeta} := B_{\frac{|\gamma-\zeta|}{2}}(\zeta)\cap \omega(\tilde{\Gamma}_l)\quad\text{and}\quad \Gamma_l' := \omega(\tilde{\Gamma}_l) \setminus (\Gamma_{\gamma} \cup \Gamma_{\zeta})
\end{equation}
and observe that spatial relations in \eqref{space-rel1} remain valid. Hence, 
\[
\int_{\Gamma_{\gamma}}\tilde{f}_3d\xi \lesssim \frac{|\zeta|}{1+|\zeta|^2|\zeta-\gamma|^2}\int_{\Gamma_{\gamma}}\frac{1}{1+|\zeta|^3|\xi-\gamma|}d\xi \lesssim (1+|\zeta|^4|\zeta-\gamma|^2)^{-1}(\log|\zeta| + \log|\zeta-\gamma|),
\]
\[
\int_{\Gamma_{\zeta}}\tilde{f}_3d\xi \lesssim \frac{|\zeta|}{1+|\zeta|^3|\zeta-\gamma|}\int_{\Gamma_{\zeta}}\frac{1}{1+|\zeta|^2|\xi-\zeta|^2}d\xi \lesssim (1+|\zeta|^4|\zeta-\gamma|)^{-1},
\]
\[
\int_{\Gamma_{l}'}\tilde{f}_3d\xi \lesssim \frac{1}{(1+|\zeta|^2|\zeta-\gamma|^2)(1+|\zeta|^3|\zeta-\gamma|)}\int_{\Gamma_{\gamma}}|\xi|d\xi \lesssim (1+|\zeta|^4|\zeta-\gamma|^2)^{-1},
\]
thus allowing us to conclude that
\[
I_3(v_+) \lesssim (1+|\zeta|^2|\zeta-\gamma|)^{-1} = (1+|\omega_l|^2|\omega_{ls}^-|)^{-1}.
\]
The exact same argument can also be employed to establish that 
\[
I_4(v_+) \lesssim (1+|\zeta|^2|\zeta-\gamma|)^{-1} = (1+|\omega_l|^2|\omega_{ls}^-|)^{-1},
\]
which implies $|S_1| \lesssim (1+|\omega_l|^2|\omega_{ls}^-|)^{-1}$. This concludes the proof. 
\end{proof}

Lemma \ref{D1-estt} sets the scene for the rest of the proof. In particular we exploit it to establish the first result for the mixed derivative of $\bar{\G}$.
\begin{lemma}\label{D1D2annulus}
If $l \in \mathcal{A}(s)$, $\tau \in \Rc(l)$, and $\lambda \in \Rc(s)$,  then 
\[
|D_{1,\tau}D_{2,\lambda}\bar{\G}(l,s)| \lesssim (1+|\omega_l||\omega_s||\omega^-_{ls}|^{2})^{-1}.
\]
\end{lemma}
\begin{proof}
Using the same cut-off function $\eta$ as in Lemma \ref{D1-estt}, we again distinguish two cases. \\

\paragraph{Case 1: $\supp \eta \cap \Gamma = \emptyset$.}\mbox{}\\
We write
\[
\Dm{\tau}\Ds{\lambda}\bar{\G}(l,s) = S_1 + S_2,
\]
where this time
\[
S_1 = \sum_{m\in\La}\sum_{\rho\in\Rc}D_{\rho}[D_{\tau}\G^{\rm hom}(m-l)\eta(m)]\Dm{\rho}\Ds{\lambda}\bar{\G}(m,s)	
\]
and 
\begin{align*}
S_2 = \sum_{m\in\La}\sum_{\rho\in\Rc} D_{\rho}\eta(m)&\Bigg[A_{\rho}D_{\tau}\G^{\rm hom}(m-l)\Dm{\rho}\Ds{\lambda}\bar{\G}(m,s)\\ &+ D_{\rho}D_{\tau}\G^{\rm hom}(m-l)A_{1,\rho}\Ds{\lambda}\bar{\G}(m,s)\Bigg].
\end{align*}
The $S_1$ part can be treated similarly to before, with the key difference being that we have an extra $s$-derivative on terms corresponding to the predictor. We thus let $v(m):= D_{\tau}\G^{\rm hom}(m-l)\eta(m)$ and  estimate
\[
|S_1| \lesssim \sum_{i=1}^4 J_i(v),
\]
where $(J_i)$ are defined as similarly $(I_i)$ in \eqref{I12}-\eqref{I34}, but with an additional derivative with respect to $s$, namely
\begin{equation}\label{J12}
J_1(v) = \sum_{b(m,\rho)\not\subset\Gamma}h^{(a)}_{3}(m)|D_{\rho}v(m)|, \quad J_2(v) = \sum_{b(m,\rho)\not\subset\Gamma}h^{(b)}_{3}(m)|D_{\rho}v(m)| 
\end{equation}
and
\begin{equation}\label{J34}
J_3(v) = \sum_{b(m,\rho)\subset\Gamma}h^{(a)}_{2}(m)|D_{\rho}v(m)|,\quad J_4(v) = \sum_{b(m,\rho)\subset\Gamma}h^{(b)}_{2}(m)|D_{\rho}v(m)|,
\end{equation}
with $h^{(a)}_{\alpha}$ and $h^{(b)}_{\alpha}$ defined in \eqref{hahb}.

Throughout we apply the same procedure as in the proof of Lemma \ref{D1-estt}, thus we omit some repetitions. {\cb We begin by recalling the decay estimate of the homogeneous Green's function $\G^{\rm hom}$ given in \eqref{Ghomdecay} and the subsequent discussion following the definition in \eqref{vgheta}, which together imply that}
\[
J_1(v) \lesssim \sum_{m \in B_{c_2R}(l)}(1+|\omega_m|^{5}|\omega_s||\omega^-_{ms}|^2)^{-1}(1+|\omega^-_{ml}|^{2}|\omega_m|^{2})^{-1} =: \sum_{m \in B_{c_2R}(l)}g_1(m) 
\]
which then leads to 
\[
J_1(v) \lesssim g_1(l) + g_1(s)\mathbbm{1}_{B_{c_2R}(l)}(s) + \int_{D_{R}(l)}g_1(x)dx.
\]
It is further true that
\[
g_1(l) = (1+|\omega_l|^{5}|\omega_s||\omega_{ls}^-|^2)^{-1}\quad\text{and}\quad g_1(s)\mathbbm{1}_{B_{c_2R}(l)}(s) \lesssim (1+|\omega_l||\omega_s||\omega_{ls}^-|^2)^{-1}.
\]
We then consider 
\[
\int_{D_{R}(l)}g_1(x)dx = \int_{\omega(D_{R}(l))}\frac{|\xi|^{2}}{(1+|\xi|^{5}|\gamma||\xi -\gamma|^2)(1+|\xi-\zeta|^{4})} d\xi =: \int_{\omega(D_{R}(l))}\tilde{g}_1(\xi)d\xi
\]
and recall the regions of integration from \eqref{regions1}. Following the same logic as in the proof of Lemma \ref{D1-estt}, we can thus conclude that
\begin{align*}
\int_{\Omega_{\gamma}} \tilde{g}_1 d\xi &\lesssim \frac{|\zeta|^2}{1+|\zeta-\gamma|^4} \int_{\Omega_{\gamma}}\frac{1}{1 + |\zeta|^5|\gamma||\xi-\gamma|^2}d\xi \lesssim \frac{|\zeta|^2}{1+|\zeta-\gamma|^4}\int_{{\cb \frac{1}{|\gamma|}}}^{\frac{|\zeta-\gamma|}{2}} \frac{r}{1+|\zeta|^5|\gamma| r^2}dr \\
&\lesssim (1+|\zeta|^3|\gamma||\zeta-\gamma|^4)^{-1}\log|\zeta-\gamma|,
\end{align*}
\[
\int_{\Omega_{\zeta}} \tilde{g}_1 d\xi \lesssim \frac{|\zeta|^2}{1+|\zeta|^5|\gamma||\zeta-\gamma|^2} \int_{\Omega_{\zeta}}\frac{1}{1 +|\zeta-\gamma|^4}d\xi \lesssim (1+|\zeta|^3|\gamma||\zeta-\gamma|^2)^{-1}
\]
and
\[
\int_{\Omega'} \hat{g}_1 d\xi \lesssim \frac{1}{(1+|\zeta|^5|\gamma||\zeta-\gamma|^2)(1+|\zeta-\gamma|^4)} \int_{\Omega'}|\xi|^2d\xi \lesssim (1+|\zeta|^3|\gamma||\zeta-\gamma|^{4})^{-1}.
\]
For $J_2(v)$, similarly,
\[
J_2(v) \lesssim \sum_{m \in B_{c_2R}(l)}{\cb(1+|\omega_m|^{3}|\omega_s||\omega^-_{ms}|^4)^{-1}}(1+|\omega^-_{ml}|^{2}|\omega_m|^{2})^{-1} =: \sum_{m \in B_{c_2R}(l)}g_2(m)
\]
and thus
\[
J_2(v) \lesssim g_2(l) + g_2(s)\mathbbm{1}_{B_{c_2R}(l)}(s) + \int_{D_{R}(l)}g_2(x)dx.
\]
We further note that  
\[
g_2(l) = (1+|\omega_l|^{3}|\omega_s||\omega_{ls}^-|^4)^{-1}\quad\text{and}\quad g_2(s)\mathbbm{1}_{B_{c_2R}(l)}(s) \lesssim (1+|\omega_l||\omega_s||\omega_{ls}^-|^2)^{-1}
\]
and
\[
\int_{D_{R}(l)}g_2(x)dx, = \int_{\omega(D_{R}(l))}\frac{|\xi|^{2}}{(1+|\xi|^{3}|\gamma||\xi -\gamma|^{4})(1+|\xi|^2|\xi-\zeta|^{2})} d\xi =: \int_{D_R(l))}\tilde{g}_2(\xi)d\xi,
\]
with estimates 
\begin{align*}
\int_{\Omega_{\gamma}} \tilde{g}_2 d\xi &\lesssim \frac{|\zeta|^2}{1+|\zeta|^2|\zeta-\gamma|^2} \int_{\Omega_{\gamma}}\frac{1}{1 + |\zeta|^3|\gamma||\xi-\gamma|^4}d\xi\\  &\lesssim \frac{|\zeta|^2}{1+|\zeta|^{2}|\zeta-\gamma|^2}\int_{{\cb \frac{1}{|\gamma|}}}^{\frac{|\zeta-\gamma|}{2}} \frac{r}{1+|\zeta|^3|\gamma| r^4}dr \lesssim (1+|\zeta|^3|\gamma||\zeta-\gamma|^2)^{-1},
\end{align*}
\[
\int_{\Omega_{\zeta}} \tilde{g}_2 d\xi \lesssim \frac{|\zeta|^2}{1+|\zeta|^3|\gamma||\zeta-\gamma|^4} \int_{\Omega_{\zeta}}\frac{1}{1 +|\zeta|^{2}|\xi-\zeta|^2}d\xi \lesssim (1+|\zeta|^3|\gamma||\zeta-\gamma|^4)^{-1}\log|\xi-\gamma|
\]
and
\[
\int_{\Omega'} \tilde{g}_2 d\xi \lesssim \frac{1}{(1+|\zeta|^3|\gamma||\zeta-\gamma|^4)(1+|\zeta|^2|\zeta-\gamma|^2)} \int_{\Omega'}|\xi|^2d\xi \lesssim (1+|\zeta|^3|\gamma||\zeta-\gamma|^{4})^{-1}.
\]
This establishes that
\[
|S_1| \lesssim (1+|\zeta||\gamma||\zeta-\gamma|^2)^{-1} = (1+|\omega_l||\omega_s||\omega_{ls}^-|^2)^{-1}.
\]
For $S_2$, we note that due to variable symmetry we have $\Ds{\lambda}\bar{\G}(m,s) = \Dm{\lambda}\bar{\G}(s,m)$ and since $l \in \mathcal{A}$, then $m \in B_{c_2R}(l)$ is such that
\[
|m| \leq |l| + c_2R \leq |l|(1+c_2(1+\sqrt{2}))\leq \frac{3}{2}(1+c_2(1+\sqrt{2}))|s| \implies |s| \geq \frac{2}{3(1+c_2(1+\sqrt{2}))}|m|.
\]
As a result, with $c_2 = \frac{1}{6}$ we have that $s \in \La\setminus\Omega_1(m)$ and the result of Lemma \ref{D1-estt} applies, thus $\Ds{\lambda}\bar{\G}(m,s) \lesssim |\omega_s|^{-1}|\omega_{ms}|^{-1}$.  We can exploit this fact by summing the first term by parts and hence consider
\begin{align*}
S_2 = S_{2a} + S_{2b} := &\sum_{m\in\La}-\divo\left( D\eta(m)\odot AD_{\tau}\G^{\rm hom}(m-l)\right)\Ds{\lambda}\bar{\G}(m,s)\\
+&\sum_{m\in\La}\sum_{\rho\in\Rc}D_{\rho}\eta(m)D_{\rho}D_{\tau}\G^{\rm hom}(m-l)A_{1\rho}\Ds{\lambda}\bar{\G}(m,s),
\end{align*}
where $D\eta(m)\odot AD_{\tau}\G^{\rm hom}(m-l) = \left(D_{\rho}\eta(m)\odot A_{1\rho}D_{\tau}\G^{\rm hom}(m-l)\right)_{\rho\in\Rc}$.
 
For $S_{2b}$ we note that since $|D\eta(m)| \lesssim R^{-1} = |\omega_l|^{-1}|\omega_{ls}^-|^{-1}$, we can estimate
\[
|S_{2b}| \lesssim (1+|\omega_l||\omega^-_{ls}|)^{-1}\sum_{m\in\mathcal{A}_l}(1+|m-l|^{2})^{-1}(1+|\omega_s||\omega^-_{ms}|)^{-1}, 
\]
where again $\mathcal{A}_l = B_{c_2R}(l)\setminus B_{c_1R}(l)$. With the substitution $\xi = \omega(m)$ and the identity in \eqref{csqrt-id}, we thus obtain
\[
|S_{2b}| \lesssim (1+|\zeta||\zeta-\gamma|)^{-1}\int_{\omega(\mathcal{A}_l)}\frac{|\xi|^{2}}{(1+|\gamma||\xi-\gamma|)(1+|\xi-\zeta|^2|\xi+\zeta|^2)}d\xi =: \int_{\omega(\mathcal{A}_l)}{\cb \hat{f}_1(\xi)}d\xi
\]
We carve $\omega(\mathcal{A}_l)$ into $U_{\gamma} :=B_{|\zeta-\gamma|/2}(\gamma)\cap \omega(\mathcal{A}_l)$ and $\omega(\mathcal{A}_l)\setminus U_{\gamma}$, noting that in some cases $U_{\gamma}$ can be empty, but it does not affect the argument. We first note that
\[
\int_{U_{\gamma}}\hat{f}_1d\xi \lesssim \frac{|\zeta|^2}{(1+|\zeta|^3|\zeta-\gamma|^3)}\int_{U_{\gamma}}\frac{1}{1+|\gamma||\xi-\gamma|}d\xi \lesssim (1+|\zeta||\gamma||\zeta-\gamma|^2)^{-1}.
\]
On the other hand, if $\xi \in \mathcal{A}_l\setminus U_{\gamma}$, then  $|\xi-\gamma|\gtrsim |\zeta-\gamma|$. Furthermore, \eqref{csqrt-id} together with how $\mathcal{A}_l$ is defined implies that 
\begin{equation}\label{annulus-comprbl}
|\xi-\zeta||\xi+\zeta| \sim R = |\omega_l|\omega^-_{ls}| = |\zeta||\zeta-\gamma|.
\end{equation}
Hence
\[
\int_{\omega(\mathcal{A}_l)\setminus U_{\gamma}}\hat{f}_1d\xi \lesssim \frac{1}{(1+|\zeta|^3|\gamma||\zeta-\gamma|^4)}\int_{\omega(\mathcal{A}_l)\setminus U_{\gamma}} |\xi|^2 \lesssim (1+|\zeta||\gamma||\zeta-\gamma|^2)^{-1}.
\] 
As a result
\[
|S_{2b}| \lesssim (1+|\zeta||\gamma||\zeta-\gamma|^2)^{-1}.
\]
For $S_{2a}$, when we apply the discrete divergence operator, we use the product rule and obtain two sub-terms
\[
S_{2a}^{(i)} := \sum_{m \in \La}\sum_{\rho\in\Rc} D_{\rho}\eta(m) \left({\cb A_{\rho}D_{\tau}\G^{\rm hom}(m-\rho - l) - A_{\rho}D_{\tau}\G^{\rm hom}(m-l)}\right)\Ds{\lambda}\bar{\G}(m,s) 
\]
and
\[
S_{2a}^{(ii)} := \sum_{m \in \La}\sum_{\rho\in\Rc} \left(D_{\rho}\eta(m-\rho) - D_{\rho}\eta(m)\right){\cb A_{\rho}D_{\tau}\G^{\rm hom}(m-\rho - l)}\Ds{\lambda}\bar{\G}(m,s).
\]
In the first one the additional derivative goes onto $D_{\tau}\G^{\rm hom}(m-l)$ and thus this can be estimated in the same way as $S_{2b}$. For the other sub-term we have the additional derivative on the cut-off function, which leads us to exploit $|D^{2}\eta(m)| \lesssim R^{-2} \lesssim (1+|\omega_l|^2|\omega^-_{ls}|^{2})^{-1}$. Hence
\[
|S_{2a}^{(ii)}| \lesssim (1+|\omega_l|^{2}|\omega^-_{ls}|^{2})^{-1} \sum_{m\in\mathcal{A}_l}(1+|m-l|)^{-1}(1+|\omega_s||\omega_{ms}|)^{-1}.
\]
Similarly to how we argued for $S_{2b}$, we write
\[
|S_{2a}^{(ii)}| \lesssim (1+|\zeta|^2|\zeta-\gamma|^2)^{-1}\int_{\omega(\mathcal{A}_l)}\frac{|\xi|^{2}}{(1+|\gamma||\xi-\gamma|)(1+|\xi-\zeta||\xi+\zeta|)}d\xi =: \int_{\omega(\mathcal{A}_l)}\hat{f}_2(\xi,\zeta,\gamma)d\xi
\]
Looking at sets $U_{\gamma}$ and $\omega(\mathcal{A}_l)\setminus U_{\gamma}$ separately again, we get that 
\[
\int_{U_{\gamma}}\hat{f}_2d\xi \lesssim \frac{|\zeta|^2}{(1+|\zeta|^3|\zeta-\gamma|^3)}\int_{U_{\gamma}}\frac{1}{1+|\gamma||\xi-\gamma|}d\xi \lesssim (1+|\zeta||\gamma||\zeta-\gamma|^2)^{-1},
\]
whereas, again exploiting \eqref{annulus-comprbl}, we have 
\[
\int_{\omega(\mathcal{A}_l)\setminus U_{\gamma}}\hat{f}_2d\xi \lesssim \frac{1}{(1+|\zeta|^3|\gamma||\zeta-\gamma|^4)}\int_{\omega(\mathcal{A}_l)\setminus U_{\gamma}} |\xi|^2 \lesssim (1+|\zeta||\gamma||\zeta-\gamma|^2)^{-1}.
\] 
Hence
\[
|S_{2a}^{(ii)}| \lesssim (1+|\zeta||\gamma||\zeta-\gamma|^2)^{-1},
\]
which concludes the result.\\

\paragraph{Case 2: $\supp \eta \cap \Gamma \neq \emptyset$.}\mbox{}\\
Looking at the corresponding proof in Lemma \ref{D1-estt}, we notice that the result will follow from the same manifold $\Mc$ construction {\cb introduced in \eqref{manifoldM}}, as long as we correctly estimate $J_3(v_+)$ and $J_4(v_+)${\cb , defined in \eqref{J34}, with a formula for $v_+$ given in \eqref{vplus}. We proceed as follows.}

We first note that 
\[
J_3(v_+) \lesssim \sum_{m\in\Gamma_l}(1+|\omega_m|^3|\omega_s||\omega_{ms}^-|^2)^{-1}(1+|\omega_{ml}^-|^2|\omega_{ml}^+|^2)^{-1} =: \sum_{m\in\Gamma_l} g_3(m)
\]
and thus
\[
J_3(v_+) \lesssim g_3(l)\mathbbm{1}_{\Gamma_l}(l) + g_3(s)\mathbbm{1}_{\Gamma_l}(s) + \int_{\tilde{\Gamma}_{l}}g_3(x)dx,
\]
where
\[
g_3(l) = (1+|\omega_l|^{3}|\omega_s|\omega_{ls}^-|^2)^{-1}\quad\text{and}\quad g_3(s)\mathbbm{1}_{\Gamma_l}(s) = (1+|\omega_s|^2|\omega_{ls}^-|^2)^{-1} \lesssim (1+|\omega_l||\omega_s||\omega_{ls}^-|^2)^{-1}.
\]
For the integral term we argue that
\[
\int_{\tilde{\Gamma}_{l}}g_3(x)dx =  \int_{\omega(\tilde{\Gamma}_{l})}\frac{|\xi|}{(1+|\xi|^{3}|\gamma||\xi-\gamma|^2)(1+|\xi-\zeta|^2||\xi+\zeta|^2)}d\xi =: \int_{\omega(\tilde{\Gamma}_{l})}\tilde{g}_3(\xi)d\xi.
\]
As before we now look at regions defined in \eqref{regions2} and observe that 
\[
\int_{\Gamma_{\gamma}}\tilde{g}_3d\xi \lesssim \frac{|\zeta|}{1+|\zeta|^2|\zeta-\gamma|^2}\int_{\Gamma_{\gamma}}\frac{1}{1+|\zeta|^3|\gamma||\xi-\gamma|^2}d\xi \lesssim (1+|\zeta|^4|\gamma||\zeta-\gamma|^2)^{-1},
\]
\[
\int_{\Gamma_{\zeta}}\tilde{g}_3d\xi \lesssim \frac{|\zeta|}{1+|\zeta|^3|\gamma||\zeta-\gamma|^2}\int_{\Gamma_{\zeta}}\frac{1}{1+|\zeta|^2|\xi-\zeta|^2}d\xi \lesssim (1+|\zeta|^4|\gamma||\zeta-\gamma|^2)^{-1}
\]
\[
\int_{\Gamma_{l}'}\tilde{g}_3d\xi \lesssim \frac{1}{(1+|\zeta|^2|\zeta-\gamma|^2)(1+|\zeta|^3|\gamma||\zeta-\gamma|^2)}\int_{\Gamma_{\gamma}}|\xi|d\xi \lesssim (1+|\zeta|^4|\gamma||\zeta-\gamma|^3)^{-1},
\]
thus allowing us to conclude that
\[
J_3(v_+) \lesssim (1+|\zeta||\omega|\zeta-\gamma|^2)^{-1} = (1+|\omega_l||\omega_s||\omega_{ls}^-|^2)^{-1}.
\]
Finally, a corresponding argument can be employed to establish that
\[
J_4(v_+) \lesssim (1+|\zeta||\gamma||\zeta-\gamma|^2)^{-1} = (1+|\omega_l||\omega_s||\omega_{ls}^-|^2)^{-1},
\]
which implies $|S_1| \lesssim (1+|\omega_l||\omega_s||\omega_{ls}^-|^2)^{-1}$ and concludes the proof. 
\end{proof}

The procedure described in Lemma \ref{D1D2annulus} cannot be employed if we are too close or too far away from origin relative to $s$ a new approach is needed. It turns out that for $l \in \Omega_1(s) \cup \Omega_2(s)$ one can obtain a preliminary result in the form of norm estimates.
\begin{lemma}\label{lem-norm-est1}
For any $s$ with $|s|$ large enough and $\tau \in \Rc(s)$, the function  $\bar{g}(m,s) := D_{2,\tau}\bar{\G}(m,s)$ satisfies
\[
\|D_1\bar{g}(\cdot,s)\|_{\ell^2(\Omega_1(s))} \lesssim |\omega_s|^{-2}\quad\text{and}\quad \|D_1\bar{g}(\cdot,s)\|_{\ell^2(\Omega_2(s))} \lesssim |\omega_s|^{-2}.	
\]
\end{lemma}
\begin{proof}
We begin by noting that the equation that $\bar{g}$ satisfies is given by
\[
H\bar{g}(m,s) = - H\hat{g}(m,s)\quad\text{ for } m \in \Omega_1(s) \cup \Omega_2(s)
\]
where $\hat{g}(m,s) := D_{2,\tau}\hat{\G}(m,s)$. This point-wise equation is obtained from \eqref{G_min_eq} after testing with $v(l) = \delta(l,m)$ (the Kronecker delta) and noting that inside $\Omega_i(s)$ we are away from $s$.

We multiply both sides by $\bar{g}(\cdot,s)\eta^2_1$ or $\bar{g}(\cdot,s)\eta^2_2$ and sum over $m$. The cut-off function $\eta_i$ is defined to be identically 1 inside $\Omega_i(s)$ and to go smoothly and monotonically to zero over an annulus of radius $c_3|s|$ where the choice of $c_3$ ensures that $\dist(\supp \eta_i,s) \sim |s|$. A particular choice of $\eta_1$ and $\eta_2$ that works is as follows: $\eta_1(m)=1$ for $m \in B_{\sfrac{5|s|}{8}}(0)$ and $\eta_1(m) = 0$ for $m \in \La\setminus B_{\sfrac{6|s|}{8}}(0)$. Similarly, $\eta_2(m) = 1$ for $m \in \La\setminus B_{\sfrac{11|s|}{8}}(0)$ and $\eta_2(m) = 0$ for $m \in B_{\sfrac{10|s|}{8}}(0)$. Thus $c_3 = \frac{1}{8}$ and the aforementioned annuli are still at least $\frac{|s|}{4}$ away from $s$.

Note that as a result we have $|D^j\eta_i| \lesssim |s|^{-j} = |\omega_s|^{-2j}$ and also that it is non-zero only in the region where the result from Lemma \ref{D1D2annulus} can be applied. It can be essentially thought of as imposing a boundary condition on a discrete variant of the Poisson equation on $\Omega_i(s)$. Following summation by parts on the left-hand side we get 
\begin{equation}\label{adcompl1}
\sum_{m\in\La} D_1\bar{g}(m,s)\cdot D_1[\bar{g}(m,s)\eta_i^2(m)]	= -\sum_{m\in\La}\left(H\hat{g}(m,s)\right)\bar{g}(m,s)\eta^2_i(m).
\end{equation}
{\cb Furthermore, we can rewrite the left-hand side of \eqref{adcompl1} due to the following discrete product rule identity satisfied by any $f,g\, : \,\La \to \R$,
\[
D_{\rho}[fg](m) = f(m+\rho)D_{\rho}g(m) + f(m)D_{\rho}g(m) = g(m+\rho)D_{\rho}f(m) + g(m)D_{\rho}f(m),
\]
which implies that (suppressing the $s$-dependence of $\bar{g}$ and $i$-dependence of $\eta$ in notation for brevity)
\begin{align*}
D_{\rho}\bar{g}(m)\cdot D_{\rho}[(\bar{g}\eta)\eta](m) = & D_{\rho}\bar{g}(m)\cdot\left[\bar{g}(m+\rho)\eta(m+\rho)D_{\rho}\eta(m) + \eta(m)D_{\rho}[\bar{g}\eta](m)\right]\\
=& D_{\rho}g(m)D_{\rho}\eta(m)\bar{g}(m+\rho)\eta(m+\rho)\\
&+ D_{\rho}\bar{g}(m)^2\eta(m)\eta(m+\rho)
+D_{\rho}\bar{g}(m)\eta(m)\bar{g}(m)D_{\rho}\eta(m).
\end{align*}
A similar calculation reveals that
\[
D[\bar{g}\eta](m)\cdot D[\bar{g}\eta](m) = D_{\rho}\bar{g}(m)\cdot D_{\rho}[(\bar{g}\eta)\eta](m) + D_{\rho}\eta(m)^2\bar{g}(m+\rho)\bar{g}(m),
\] 
which allows us to rewrite \eqref{adcompl1} as }
\begin{equation}\label{norm-est-eq}
\|D_1[\bar{g}(\cdot,s)\eta_i]\|_{\ell^2}^2 = -\sum_{m\in\La}\left(H\hat{g}(m,s)\right)\bar{g}(m,s)\eta^2_i(m) + \sum_{m\in\La}\sum_{\rho\in\Rc}\left(D_{\rho}\eta_i(m)\right)^2\bar{g}(m,s)\bar{g}(m+\rho,s)
\end{equation}
and subsequently we hope to estimate the right-hand side, in particular noting that $|\bar{g}(m+\rho,s)| \sim |\bar{g}(m,s)|$.  

We first deal with terms that are not on the crack surface. For $m\not\in\Gamma$ we know that $H$ coincides with the homogeneous Hessian operator ($\tilde{H}$), {\cb thus a Taylor expansion of
\[
H\hat{g}(m,s) = \sum_{\rho \in \Rc} \Dm{\rho}\hat{g}(m-\rho) - \Dm{\rho}\hat{g}(m)
\]
up to fourth order, as calculated explicitly in the proof of  \cite[Theorem 5.6.]{2017-bcscrew}, }together with the fact that $\hat{\G}(,\cdot,s)$ solves the Laplace equation away from $s$ we can conclude that 
\[
|H\hat{g}(m,s)| \lesssim  \|\nabla_m^4\nabla_s \hat{G}(\cdot,s)\|_{L^{\infty}(B_{1/2}(m))}
\]
and hence, bearing in mind \eqref{hahb},
\[
|H\hat{g}(m,s)| \lesssim h^{(a)}_4(m,s) +h^{(b)}_4(m,s).
\]
If on the other hand $m \in \Gamma$, then {\cb a careful examination of terms in $H\hat{g}(m)$ reveals that 
\[
H\hat{g}(m,s) = 2\sum_{\rho \in \Rc(m)} \Dm{\rho}\hat{g}(m,s),
\]
which, following a Taylor expansion, implies
\begin{align}
|H\hat{g}(m,s)| \lesssim &\left|\sum_{\rho \in \Rc(m)}  \left(2\nabla_m \bar{g}(m,s)\cdot\rho + \nabla^2_m\hat{g}(m,s)[\rho,\rho] + \frac{1}{3}\nabla_m^3\bar{g}(m,s)[\rho,\rho,\rho]\right)\right|\nonumber\\
&+ \|\nabla_m^4\nabla_s \hat{G}(\cdot,s)\|_{L^{\infty}(B_{1/2}(m))}\nonumber\\
&\lesssim \|\nabla_m \bar{g}(m,s) \cdot e_2\|_{L^{\infty}(B_{1/2}(m))} + \|\nabla_m^4\nabla_s \hat{G}(\cdot,s)\|_{L^{\infty}(B_{1/2}(m))},\label{hghatgamma}
\end{align}
where the final line follows from the fact that 
\[
m \in \Gamma_{\pm} \implies \sum_{\rho \in \Rc(m)} 2\nabla_m \bar{g}(m,s)\cdot\rho = \pm 2\nabla_m\bar{g}(m,s)\cdot e_2,
\]
\begin{align*}
m \in \Gamma_{\pm} \implies \sum_{\rho \in \Rc(m)} \nabla^2_m \bar{g}(m,s)[\rho,\rho] &= 2\Delta_m\bar{g}(m,s) - \nabla^2_m\bar{g}(m,s)[\pm e_2,\pm e_2]\\
&=- \nabla^2_m\bar{g}(m,s)[\pm e_2,\pm e_2],
\end{align*}
since $\bar{g}(m,s)$ solve the Laplace equation away from $s$. Finally,  
\[
m \in \Gamma_{\pm} \implies \sum_{\rho \in \Rc(m)} \nabla^3_m \bar{g}(m,s)[\rho,\rho,\rho] = \pm\frac{1}{3} \nabla^3_m\bar{g}(m,s)[\pm e_2,\pm e_2,\pm e_2].
\]
The second term in \ref{hghatgamma} can be estimated as in \eqref{hahb}.  For the first term, using the boundary condition in \eqref{Gpred-eq1}, we can Taylor-expand this around $m_0 \in \Gamma$ vertically aligned with $m$, allowing us to gain one extra derivative. As a result
\[
\|\nabla_m \bar{g}(m,s) \cdot e_2\|_{L^{\infty}(B_{1/2}(m))} \lesssim h^{(a)}_2(m) +h^{(b)}_2(m).
\]}

For $i=1$ we note that in the first term on the right-hand side of \eqref{norm-est-eq} we only sum over $m\in B_{\frac{3|s|}{4}}(0)$, thus Lemma \ref{D1-estt} ensures that $\bar{g}(m,s) \lesssim |\omega_s|^{-1}|\omega^-_{ms}|^{-1}$. Furthermore, recalling \eqref{csqrt-id}, we can rewrite $|\omega^-_{ms}| = |m-s||\omega_{ms}^+|^{-1}$ and exploit the fact in the region of interest $|m-s| \gtrsim |s| = |\omega_s|^{2}$ and $|\omega_{ms}^+| \leq |\omega_m|+|\omega_s|\lesssim |\omega_s|$. Finally, we note that due to placing the defect core at the origin, we always have $m \in \La \implies |m| > 1/\sqrt{2}$. Thus we can estimate
\begin{align}\label{norm-est1}
\sum_{m\in B_{\frac{3|s|}{4}}(0)}h^{(a)}_4(m,s)|\bar{g}(m,s)| &\lesssim \sum_{m\in B_{\frac{3|s|}{4}}(0)} (1+|\omega_m|^7|\omega_s||\omega_{ms}^-|^2)^{-1}(1+|\omega_{ms}^{-}|\omega_s|)^{-1}\\
&\lesssim |\omega_s|^{-5}\sum_{m\in\La}|\omega_m|^{-7}\lesssim |\omega_s|^{-5},\nonumber
\end{align}
\begin{align}\label{norm-est2}
\sum_{m\in B_{\frac{3|s|}{4}}(0)}h^{(b)}_4(m,s)|\bar{g}(m,s)| &\lesssim \sum_{m\in B_{\frac{3|s|}{4}}(0)} (1+|\omega_m|^4|\omega_s||\omega_{ms}^-|^5)^{-1}(1+|\omega_{ms}^{-}||\omega_s|)^{-1}\\
&\lesssim |\omega_s|^{-8}\sum_{{\cb m\in B_{\frac{3|s|}{4}}(0)}}|\omega_m|^{-4} \lesssim |\omega_s|^{-8}\log|\omega_s|,\nonumber
\end{align}
\begin{align}\label{norm-est-bt1}
\sum_{m\in B_{\frac{3|s|}{4}}(0)\cap\Gamma}h^{(a)}_2(m,s)|\bar{g}(m,s)| &\lesssim \sum_{m\in B_{\frac{3|s|}{4}}(0)\cap\Gamma} (1+|\omega_m|^3|\omega_s||\omega_{ms}^-|^2)^{-1}(1+|\omega_{ms}^{-}||\omega_s|)^{-1}\\
&\lesssim |\omega_s|^{-5}\sum_{{\cb m\in B_{\frac{3|s|}{4}}(0)\cap\Gamma}}|\omega_m|^{-3}\lesssim |\omega_s|^{-5},\nonumber
\end{align}
\begin{align}\label{norm-est-bt2}
\sum_{m\in B_{\frac{3|s|}{4}}(0)\cap\Gamma}h^{(b)}_2(m,s)|\bar{g}(m,s)| &\lesssim\sum_{m\in B_{\frac{3|s|}{4}}(0)\cap\Gamma} (1+|\omega_m|^2|\omega_s||\omega_{ms}^-|^3)^{-1}(1+|\omega_{ms}^{-}||\omega_s|)^{-1}\\
&\lesssim|\omega_s|^{-6}\sum_{{\cb m\in B_{\frac{3|s|}{4}}(0)\cap\Gamma}}|\omega_m|^{-2}\lesssim |\omega_s|^{-6}\log|\omega_s|.\nonumber
\end{align}
Similarly, for the boundary term we note that if $i=1$ then we only sum over $m \in B_{\frac{3|s|}{4}}(0)\setminus B_{\frac{5|s|}{8}}(0)$ and we can estimate
\begin{equation}\label{norm-est3}
\sum_{m\in\La}\sum_{\rho\in\Rc(m)}\left(D_{\rho}\eta_i(m)\bar{g}(m,s)\right)^2 \lesssim |\omega_s|^{-8}\sum_m 1 \lesssim |\omega_s|^{-4}.
\end{equation}

For $i=2$ the estimate of boundary term in \eqref{norm-est3} still holds with the only difference being that we now sum over $m \in B_{\frac{11|s|}{8}}(0)\setminus B_{\frac{5|s|}{4}}(0)$. On the other hand, as now we sum over $m \in \La\setminus B_{\frac{5|s|}{4}}(0)$, the result of Lemma \ref{D1-estt} does not apply to  $\bar{g}(m,s)$. However, we can reproduce the proof of Lemma \ref{D1-estt} using a cut-off function with a radius  $R \sim |s|$ and thus obtain
\[
|\bar{g}(m,s)| = |D_{1,\tau}\bar{\G}(s,m)| \lesssim (1+|\omega_s|^{2}+|\omega_s||\omega^-_{ms}|)^{-1} \lesssim (1+|\omega_s|^{2})^{-1},
\]
since $|\omega_{ms}^-| = |m-s||\omega_{ms}^+| \gtrsim |m||(|\omega_m|+|\omega_s|)^{-1} \gtrsim |\omega_m| \gtrsim |\omega_s|$, as in the region of interest $|m|\,$ ($=|\omega_m|^{2}$) and  $|m-s|$ are comparable and $|m-s| \gtrsim |s|$. Thus we can estimate 
\begin{align}\label{norm-est4}
\sum_{m\in \La\setminus B_{\frac{5|s|}{4}}(0)}\left(h^{(a)}_4(m,s\right)|\bar{g}(m,s)| &\lesssim \sum_{m\in \La\setminus B_{\frac{5|s|}{4}}(0)}(1+|\omega_m|^{7}|\omega_s|^{1}|\omega^-_{ms}|^{2})^{-1}(1+|\omega_s|^{2})^{-1}\\ \nonumber
& \lesssim |\omega_s|^{-5}\sum_{m\in \La\setminus B_{\frac{5|s|}{4}}(0)}|\omega_m|^{-7} \lesssim |\omega_s|^{-5}\int_{\frac{5|s|}{4}}^{\infty}r^{-7/2}rdr \lesssim |\omega_s|^{-8}.
\end{align}
Analogues calculations result in 
\begin{equation}\label{norm-est5}
\sum_{m\in \La\setminus B_{\frac{5|s|}{4}}(0)}\left(h^{(b)}_4(m,s\right)|\bar{g}(m,s)| \lesssim |\omega_s|^{-8}.
\end{equation}
and
\begin{equation}\label{norm-est6}
\sum_{m\in \La\setminus B_{\frac{5|s|}{4}}(0)\cap\Gamma}\left(h^{(a)}_2(m,s\right)|\bar{g}(m,s)| \lesssim |\omega_s|^{-6}, \quad \sum_{m\in \La\setminus B_{\frac{5|s|}{4}}(0)\cap\Gamma}\left(h^{(b)}_2(m,s\right)|\bar{g}(m,s)| \lesssim |\omega_s|^{-6}.
\end{equation}
Since the particular choice of the cut-off functions implies that
\[
\Omega_i(s) \subset \{m \in \La\;\colon\; \eta_i(m) = 1\},
\]
we have thus established that
\[
\|D\bar{g}(\cdot,s)\|_{\ell^2(\Omega_i(s))} \leq \|D[\bar{g}(\cdot,s)\eta_i]\|_{\ell^2} \lesssim |\omega_s|^{-2}.
\]
\end{proof}
This concludes the preliminary suboptimal estimates of $\bar{\G}$. Together with Lemma \ref{D1D2hatG}, they make $\G = \hat{\G} + \bar{\G}$ a partially-functioning technical tool for estimating the decay of discrete functions defined on $\Hcc$ in a crack  geometry. In particular, we can use it to improve the decay estimates of $\bar{\G}$ to get better norm estimates over $\Omega_i$.

By looking at the estimates \eqref{norm-est1}-\eqref{norm-est6}, it is evident that we can improve these sub-optimal norm estimates as long as we are able to get a better rate in  \eqref{norm-est3}, with the summation over an annulus near the boundary of $\Omega_1(s)$, namely for $m \in B_{\frac{3|s|}{4}}(0)\setminus B_{\frac{5|s|}{8}}(0)$ and likewise over an annulus near the boundary of $\Omega_2(s)$, namely for $m \in B_{\frac{11|s|}{8}}(0)\setminus B_{\frac{5|s|}{4}}(0)$. This is possible if, instead of using a cut-off function and the homogeneous lattice Green's function $\G^{\rm hom}$, we employ $\G$. 
\begin{lemma}\label{lem-norm-est2}
Let $\tilde{s} \in \La$ be such that $|\tilde{s}|$ is large enough. If $m \in \La$ is such that $m \in B_{\frac{3|\tilde{s}|}{4}}(0)\setminus B_{\frac{5|\tilde{s}|}{8}}(0)$ or $m \in B_{\frac{11|\tilde{s}|}{8}}(0)\setminus B_{\frac{5|\tilde{s}|}{4}}(0)$, then 
\[
|\bar{g}(m,\tilde{s})| \lesssim |\omega(\tilde{s})|^{-3},
\]
where $\bar{g}(m,\tilde{s}) = \Ds{\tau}\bar{\G}(m,\tilde{s})$.
\end{lemma}
\begin{proof}
Since $\bar{g}(m,\tilde{s}) = \Dm{\tau}\bar{\G}(\tilde{s},m)$, we change the notation to keep it in line with previous proofs by letting $l = \tilde{s}$ and $m = s$. Consequently we will in fact estimate $\Dm{\tau}\bar{\G}(l,s)$ for $l \in B_{\frac{8|s|}{5}}(0)\setminus B_{\frac{4|s|}{3}}(0)$ and $l \in B_{\frac{4|s|}{5}}(0)\setminus B_{\frac{8|s|}{11}}(0)$ with the hope that we can conclude that $\Dm{\tau}\bar{\G}(l,s)\lesssim |\omega_l|^{-3}$.  With $\G$ and its suboptimal decay established, we can write
\[
\Dm{\tau}\bar{\G}(l,s) = \sum_{m\in\La}\left( H\Ds{\tau} \G(m,l)\right)\bar{\G}(m,s) = \sum_{m\in\La}\Dm{}\bar{\G}(m,s)\cdot\Dm{}g(m,l),
\]
where  $g = \hat{g} + \bar{g}$ defined in Lemma \ref{lem-norm-est1}. Noting that both $\hat{g},\bar{g} \in \Hcc$, we exploit the fact that $\bar{\G}$ satisfies \eqref{G_min_eq} and similarly to to the strategy employ in the proof of Lemma \ref{D1-estt}, we conclude that
\begin{equation}\label{D1-better1}
|\Dm{\tau}\bar{\G}(l,s)| \lesssim \sum_{i=1}^4 I_i(\hat{g}) + I_i(\bar{g}), 
\end{equation}
where the terms are as in \eqref{I12}-\eqref{I34}. Recalling that Lemma \ref{D1D2hatG}  establishes that $|\Dm{\rho}\hat{g}(m,l)| = |\Dm{\rho}\Ds{\tau}\hat{\G}(m,l)| \lesssim (1+|\omega(m)||\omega(l)||\omega_{ml}|^{2})^{-1}$ and noting the fact that in the region of interest $|l-s|$, $|l|$, and $|s|$ are all comparable, we can directly estimate four summands corresponding to this term arguing as in the proof of Lemma \ref{D1-estt}. We begin by writing 
\[
I_1(\hat{g}) \lesssim \sum_{m\in\La}(1+|\omega_m|^{5}|\omega^-_{ms}|)^{-1}(1+|\omega_m||\omega_l||\omega^-_{ml}|^{2})^{-1} =: \sum_{m\in\La} h_1(m). 
\]
We first observe that 
\[
\sum_{m\in\La} h_1(m) \lesssim h_1(l) + h_1(s) + \int_{D}\tilde{h}_1({\cb \xi})d\xi,\quad\text{where }\tilde{h}_1(\xi):= \frac{|\xi|^{2}}{(1+|\xi|^{5}|\xi-\gamma|)(1+|\xi||\zeta||\xi-\zeta|^{2})}
\]
and
\[
D:= \R^2_+\setminus (B_{\epsilon_*}(0)\cup B_1(\zeta) \cup B_1(\gamma)),
\]
with $\epsilon_* = \sqrt{1/\sqrt{2}}$. It is clear to see that 
\[
h_1(l) = (1+|\omega_l|^5|\omega_{ls}^-|)^{-1} \lesssim |\omega_l|^{-6},\quad h_1(s) = (1+|\omega_s||\omega_l||\omega_{ls}^-|^2)^{-1} \lesssim |\omega_l|^{-4}
\]
and away from the spikes we would like to estimate the integral term separately for regions close to the origin, $\gamma$ and $\zeta$ separately. To this end, we define radii 
\begin{align}\label{radii1}
R_0 &:= \min \{|\gamma|,|\zeta|\}, R_{\gamma} := \min\{|\gamma|,|\gamma-\zeta|\}, R_{\zeta}:=\min\{|\zeta|,|\zeta-\gamma|\},\\
R_1 &:= \max\{|\zeta|,|\gamma|\} + \frac{\max\{R_{\zeta}, R_{\gamma}\}}{2}\nonumber
\end{align}
and look at
\begin{align}\label{regions3}
\Omega_0 &:= (D \cap B_{\frac{R_0}{2}}(0)),\quad \Omega_{\gamma} := (D\cap B_{\frac{R_{\gamma}}{2}}(\gamma)),\quad \Omega_{\zeta} := (D \cap B_{\frac{R_{\zeta}}{2}}(\zeta)),\\ 
\Omega_1 &:= (D_+ \cap B_{R_1}(0)) \setminus (\Omega_0 \cup\Omega_{\gamma}\cup\Omega_{\zeta}),\quad \Omega' := (D\setminus B_{R_1}(0)). \nonumber
\end{align}
Exploiting the spatial properties of each of these sets and that $|\gamma|, |\zeta|$ and $|\gamma-\zeta|$ are all comparable, we can conclude that
\[
\int_{\Omega_0} \tilde{h}_1 d\xi \lesssim \int_{\Omega_0}\frac{|\xi|^2}{1 +|\xi|^6|\gamma|^4}d\xi \lesssim |\gamma|^{-4}\int_{\epsilon_*}^{\frac{R_0}{2}} \frac{1}{r^3}dr \lesssim |\zeta|^{-4},
\]
\[
\int_{\Omega_{\gamma}} \tilde{h}_1 d\xi \lesssim |\gamma|^{-2}\int_{\Omega_{\gamma}}\frac{1}{1 +|\gamma|^5|\xi-\gamma|}d\xi \lesssim |\gamma|^{-2}\int_{1}^{\frac{R_{\gamma}}{2}} \frac{r}{1+|\gamma|^5r}dr \lesssim |\gamma|^{-7}R_{\gamma} \lesssim |\zeta|^{-6},
\]
\[
\int_{\Omega_{\zeta}} \tilde{h}_1 d\xi \lesssim |\gamma|^{-4}\int_{\Omega_{\zeta}}\frac{1}{1 +|\gamma|^2|\xi-\zeta|^2}d\xi \lesssim |\gamma|^{-4}\int_{1}^{\frac{R_{\gamma}}{2}} \frac{r}{1+|\gamma|^2r^2}dr \lesssim |\gamma|^{-6}\log(R_{\zeta}) \lesssim |\zeta|^{-6}\log|\zeta|,
\]
\[
\int_{\Omega_1} \tilde{h}_1 d\xi \lesssim |\gamma|^{-4}\int_{\Omega_1}|\xi|^{-4} d\xi \lesssim |\gamma|^{-4}\int_{\frac{R_0}{2}}^{R_1} \frac{1}{r^3}dr \lesssim |\zeta|^{-6},
\]
\[
\int_{\Omega'} \tilde{h}_1 d\xi \lesssim |\gamma|^{-4}\int_{\Omega'}|\xi|^{-4} d\xi \lesssim |\gamma|^{-4}\int^{\infty}_{R_1} \frac{1}{r^3}dr \lesssim |\zeta|^{-6}.
\]
Likewise for the second term we begin by saying
\[
I_2(\hat{g}) \lesssim \sum_{m\in\La}(1+|\omega(m)|^{3}|\omega^-_{ms}|^3)^{-1}(1+|\omega_m||\omega(l)||\omega^-_{ml}|^{2})^{-1} =: \sum_{m\in\La} h_2(m)
\]
and 
\[
\sum_{m\in\La}h_2(m) \lesssim h_2(l) + h_2(s) + \int_D\tilde{h}_2d\xi,\quad\text{where }\tilde{h}_2(\xi) := \frac{|\xi|^{2}}{(1+|\xi|^{3}|\xi-\gamma|^3)(1+|\xi||\zeta||\xi-\zeta|^{2})}.
\]
It is clear that
\[
h_2(l) = (1+|\omega_l|^3|\omega_{ls}^-|^3)^{-1}\lesssim |\omega_l|^{-6},\quad h_2(s) = (1+|\omega_s||\omega_l||\omega_{ls}^-|^2)^{-1} \lesssim |\omega_l|^{-4}
\]
and for the integral term we look at regions defined in \eqref{regions3} again to obtain
\[
\int_{\Omega_0} \tilde{h}_2 d\xi \lesssim \int_{\Omega_0}\frac{|\xi|^2}{1 +|\xi|^4|\gamma|^6}d\xi \lesssim |\gamma|^{-6}\int_{\epsilon_*}^{\frac{R_0}{2}} \frac{1}{r}dr \lesssim |\zeta|^{-6}\log|\zeta|,
\]
\[
\int_{\Omega_{\gamma}} \tilde{h}_2 d\xi \lesssim |\gamma|^{-2}\int_{\Omega_{\gamma}}\frac{1}{1 +|\gamma|^3|\xi-\gamma|^3}d\xi \lesssim |\gamma|^{-2}\int_{1}^{\frac{R_{\gamma}}{2}} \frac{r}{1+|\gamma|^3r^3}dr \lesssim |\zeta|^{-5},
\]
\[
\int_{\Omega_{\zeta}} \tilde{h}_2 d\xi \lesssim |\gamma|^{-4}\int_{\Omega_{\gamma}}\frac{1}{1 +|\gamma|^2|\xi-\zeta|^2}d\xi \lesssim |\gamma|^{-4}\int_{1}^{\frac{R_{\gamma}}{2}} \frac{r}{1+|\gamma|r^2}dr \lesssim |\gamma|^{-6}\log(R_{\zeta}) \lesssim |\zeta|^{-6}\log|\zeta|,
\]
\[
\int_{\Omega_1} \tilde{h}_2 d\xi \lesssim |\gamma|^{-6}\int_{\Omega_1}|\xi|^{-2} d\xi \lesssim |\gamma|^{-6}\int_{\frac{R_0}{2}}^{R_1} \frac{1}{r}dr \lesssim |\zeta|^{-6}\log|\zeta|,
\]
\[
\int_{\Omega'} \tilde{h}_2 d\xi \lesssim |\gamma|^{-1}\int_{\Omega'}|\xi|^{-7} d\xi \lesssim |\gamma|^{-1}\int^{\infty}_{R_1} \frac{1}{r^6}dr \lesssim |\zeta|^{-6}.
\]
The same strategy applies to the boundary terms, as we can write
\[
I_3(\hat{g}) \lesssim \sum_{m\in\Gamma}(1+|\omega(m)|^{3}|\omega_{ms}|^1)^{-1}(1+|\omega_m||\omega(l)||\omega_{ml}|^{2})^{-1} =: \sum_{m\in\Gamma} h_3(m),
\]
and 
\[
\sum_{m\in\La}h_3(m) \lesssim h_3(l)\mathbbm{1}_{\Gamma}(l) + h_3(s)\mathbbm{1}_{\Gamma}(s) + \int_D\tilde{h}_3d\xi,\quad\text{where }\tilde{h}_3(\xi) := \frac{|\xi|}{(1+|\xi|^{3}|\xi-\gamma|)(1+|\xi||\zeta||\xi-\zeta|^{2})}.
\]
It is clear that
\[
h_3(l) = (1+|\omega_l|^3|\omega_{ls}^-|)^{-1}\lesssim |\omega_l|^{-4},\quad h_3(s) = (1+|\omega_s||\omega_l||\omega_{ls}^-|^2)^{-1} \lesssim |\omega_l|^{-4}
\]
and also
\[
\int_{\Omega_0\cap\omega(\Gamma)} \tilde{h}_3 d\xi \lesssim \int_{\Omega_0\cap\omega(\Gamma)}\frac{|\xi|}{1 +|\xi|^4|\gamma|^4}d\xi \lesssim |\gamma|^{-4}\int_{\epsilon_*}^{\frac{R_0}{2}} \frac{1}{r^3}dr \lesssim |\zeta|^{-4},
\]
\[
\int_{\Omega_{\gamma}\cap\omega(\Gamma)} \tilde{h}_3 d\xi \lesssim |\gamma|^{-3}\int_{\Omega_{\gamma}\cap\omega(\Gamma)}\frac{1}{1 +|\gamma|^3|\xi-\gamma|}d\xi \lesssim |\gamma|^{-3}\int_{1}^{\frac{R_0}{2}} \frac{1}{|\gamma|^{3}r}dr \lesssim |\zeta|^{-6},
\]
\[
\int_{\Omega_{\zeta}\cap\omega(\Gamma)} \tilde{h}_3 d\xi \lesssim |\gamma|^{-3}\int_{\Omega_{\zeta}\cap\omega(\Gamma)}\frac{1}{1 +|\gamma|^2|\xi-\zeta|^2}d\xi \lesssim |\gamma|^{-3}\int_{1}^{\frac{R_0}{2}} \frac{1}{|\gamma|^{2}r^2}dr \lesssim |\zeta|^{-5},
\]
\[
\int_{\Omega_1\cap\omega(\Gamma)} \tilde{h}_3 d\xi \lesssim |\gamma|^{-4}\int_{\Omega_1\cap\omega(\Gamma)}|\xi|^{-3} d\xi \lesssim |\gamma|^{-4}\int_{\frac{R_0}{2}}^{R_1} \frac{1}{r^3}dr \lesssim |\zeta|^{-6},
\]
\[
\int_{\Omega'\cap\omega(\Gamma)} \tilde{h}_3 d\xi \lesssim |\gamma|^{-4}\int_{\Omega'\cap\omega(\Gamma)}|\xi|^{-3} d\xi \lesssim |\gamma|^{-4}\int^{\infty}_{R_1} \frac{1}{r^3}dr \lesssim |\zeta|^{-6}.
\]
Finally, 
\[
I_4(\hat{g}) \lesssim \sum_{m\in\Gamma}(1+|\omega(m)|^{2}|\omega_{ms}|^2)^{-2}(1+|\omega_m||\omega(l)||\omega_{ml}|^{2})^{-1} =: \sum_{m\in\Gamma} h_4(m),
\]
and 
\[
\sum_{m\in\La}h_4(m) \lesssim h_4(l)\mathbbm{1}_{\Gamma}(l) + h_4(s)\mathbbm{1}_{\Gamma}(s) + \int_D\tilde{h}_4d\xi,\quad\text{where }\tilde{h}_4(\xi) := \frac{|\xi|}{(1+|\xi|^{2}|\xi-\gamma|^2)(1+|\xi||\zeta||\xi-\zeta|^{2})}.
\]
It is clear that
\[
h_4(l) = (1+|\omega_l|^2|\omega_{ls}^-|^2)^{-1}\lesssim |\omega_l|^{-4},\quad h_4(s) = (1+|\omega_s||\omega_l||\omega_{ls}^-|^2)^{-1} \lesssim |\omega_l|^{-4}
\]
and also
\[
\int_{\Omega_0\cap\omega(\Gamma)} \tilde{h}_4 d\xi \lesssim \int_{\Omega_0\cap\omega(\Gamma)}\frac{|\xi|}{1 +|\xi|^3|\gamma|^5}d\xi \lesssim |\gamma|^{-5}\int_{\epsilon_*}^{\frac{R_0}{2}} \frac{1}{r^2}dr \lesssim |\zeta|^{-5},
\]
\[
\int_{\Omega_{\gamma}\cap\omega(\Gamma)} \tilde{h}_4 d\xi \lesssim |\gamma|^{-3}\int_{\Omega_{\gamma}\cap\omega(\Gamma)}\frac{1}{1 +|\gamma|^2|\xi-\gamma|^2}d\xi \lesssim |\gamma|^{-3}\int_{1}^{\frac{R_0}{2}} \frac{1}{|\gamma|^{2}r^2}dr \lesssim |\zeta|^{-5},
\]
\[
\int_{\Omega_{\zeta}\cap\omega(\Gamma)} \tilde{h}_4 d\xi \lesssim |\gamma|^{-3}\int_{\Omega_{\zeta}\cap\omega(\Gamma)}\frac{1}{1 +|\gamma|^2|\xi-\zeta|^2}d\xi \lesssim |\gamma|^{-3}\int_{1}^{\frac{R_0}{2}} \frac{1}{|\gamma|^{2}r^2}dr \lesssim |\zeta|^{-5},
\]
\[
\int_{\Omega_1\cap\omega(\Gamma)} \tilde{h}_4 d\xi \lesssim |\gamma|^{-5}\int_{\Omega_1\cap\omega(\Gamma)}|\xi|^{-2} d\xi \lesssim |\gamma|^{-5}\int_{\frac{R_0}{2}}^{R_1} \frac{1}{r}dr \lesssim |\zeta|^{-6},
\]
\[
\int_{\Omega'\cap\omega(\Gamma)} \tilde{h}_4 d\xi \lesssim |\gamma|^{-5}\int_{\Omega'\cap\omega(\Gamma)}|\xi|^{-2} d\xi \lesssim |\gamma|^{-5}\int^{\infty}_{R_1} \frac{1}{r}dr \lesssim |\zeta|^{-6}.
\]
Since in each estimate we get at least $|\zeta|^{-4} = |\omega_l|^{-4}$, we can conclude that
\[
\sum_{i=1}^4 I_i(\hat{g}) \lesssim |\omega_l|^{-4}.
\]
For the other four terms on the right-hand side of \eqref{D1-better1}, in light of Lemma \ref{lem-norm-est1}, we look separately at the summation over $\Omega_1(l)$, $\Omega_2(l)$ and $\mathcal{A}(l)$. The first two we investigate in detail, but for the sum over $\mathcal{A}(l)$, we simply note that in there we have a point-wise estimate $\Dm{\rho}\bar{g}(m,l) \lesssim (1+|\omega(m)||\omega(l)||\omega^-_{ml}|^{2})^{-1}$, as established in Lemma \ref{D1D2annulus}, so the above estimates translate verbatim. Due to the spatial restriction on $l$ relative to $s$, we have that $m \in \Omega_i(l) \implies |m-s| \gtrsim |s|$, which also implies that $|\omega_{ms}^-| = |m-s||\omega_{ms}^+| \gtrsim |\omega_s|^2(|\omega_m|+|\omega_s|)^{-1}$. As a result we can conclude that
\[
I_1(\bar{g}) \lesssim \sum_{i=1}^2\left(\sum_{m \in \Omega_i(l)}|\omega_m|^{-10}|\omega^-_{ms}|^{-2}\right)^{1/2}\|D\bar{g}\|_{\ell^2(\Omega_i(l))} + |\omega_l|^{-4} \lesssim |\omega_l|^{-3}.
\]
An analogous argument for the remaining terms reveals that 
\[
I_2(\bar{g}) \lesssim \sum_{i=1}^2\left(\sum_{m \in \Omega_i(l)}|\omega_m|^{-6}|\omega^-_{ms}|^{-6}\right)^{1/2}\|D\bar{g}\|_{\ell^2(\Omega_i(l))} + |\omega_l|^{-4} \lesssim |\omega_l|^{-4},
\]
\[
I_3(\bar{g}) \lesssim \sum_{i=1}^2\left(\sum_{m \in \Omega_i(l)}|\omega_m|^{-6}|\omega^-_{ms}|^{-2}\right)^{1/2}\|D\bar{g}\|_{\ell^2(\Omega_i(l))} + |\omega(l)|^{-4} \lesssim |\omega_l|^{-3},
\]
\[
I_4(\bar{g}) \lesssim \sum_{i=1}^2\left(\sum_{m \in \Omega_i(l)}|\omega_m|^{-4}|\omega^-_{ms}|^{-4}\right)^{1/2}\|D\bar{g}\|_{\ell^2(\Omega_i(l))} + |\omega(l)|^{-4} \lesssim |\omega_l|^{-4}.
\]
We have thus estimated each summand in \eqref{D1-better1} and this concludes the proof.
\end{proof}
As a result, we can improve the norm estimates in Lemma \ref{lem-norm-est1} slightly. 
\begin{lemma}\label{lem-norm-est3}
For any $s$ with $|s|$ large enough and $\tau \in \Rc(s)$, the function  $\bar{g}(m,s) := D_{2,\tau}\bar{\G}(m,s)$ satisfies
\[
\|D\bar{g}(\cdot,s)\|_{\ell^2(\Omega_1)} \lesssim |\omega(s)|^{-5/2}\quad\text{and}\quad \|D\bar{g}(\cdot,s)\|_{\ell^2(\Omega_2)} \lesssim |\omega(s)|^{-3}.	
\]
\end{lemma}
\begin{proof}
With Lemma \ref{lem-norm-est2} in hand, the estimate \eqref{norm-est3} now becomes
\begin{equation}\label{norm-est5}
\sum_{m\in\La}\sum_{\rho\in\Rc(m)}\left(D_{\rho}\eta_i(m)\bar{g}(m,s)\right)^2 \lesssim |\omega(s)|^{-6},
\end{equation}
which is enough to conclude the result, as now the terms with the lowest rate of decay are given by \eqref{norm-est1} and \eqref{norm-est-bt1}, but these only apply to $\Omega_1(s)$. 
\end{proof}
To proceed further we improve upon the estimates in \eqref{norm-est1} and \eqref{norm-est-bt1}.
\begin{lemma}\label{lem-norm-est4}
Let $s \in \La$ be such that $|s|$ is large enough. If $m \in \La$ is such that $m \in B_{\frac{3|s|}{4}}(0)$ then 
\[
|\bar{g}(m,s)|\lesssim |\omega(s)|^{-5/2}.
\]
\end{lemma}
\begin{proof}
Proceeding as in Lemma \ref{lem-norm-est2}, we will estimate $\Dm{\tau}\bar{\G}(l,s)$ for $l \in \La\setminus B_{\frac{4|s|}{3}}(0)$, which can be achieved by arguing that 
\begin{equation}\label{D1-better2}
|\Dm{\tau}\bar{\G}(l,s)| \lesssim \sum_{i=1}^4 I_i(\hat{g}) + I_i(\bar{g}), 
\end{equation}
where the terms are as in \eqref{I12}-\eqref{I34}. The terms corresponding to $\hat{g}$ can be estimated as in Lemma \ref{lem-norm-est2}, with the only difference being that we longer have $|\zeta| \sim |\gamma|$, but now $|\zeta| \gtrsim |\gamma|$. We still have that $|\zeta| \sim |\zeta-\gamma|$. We can thus readily conclude that 
\[
 \sum_{i=1}^4 I_i(\hat{g}) \lesssim |\zeta|^{-3} = |\omega_l|^{-3}.
\]
For the other four terms we can still write, e.g.
\[
I_1(\bar{g}) \lesssim \sum_{i=1}^2\left(\sum_{m \in \Omega_i(l)}(1+|\omega_m|^{-5}|\omega^-_{ms}|^{-1})^{-2}\right)^{1/2}\|D\bar{g}\|_{\ell^2(\Omega_i(l))} + |\omega_l|^{-4},
\]
but since as $|l|$ grows larger, we eventually have $s \in \Omega_1(l)$, it implies that the summation over $\Omega_1(l)$ is $\mathcal{O}(1)$, thus we can only rely on the result of Lemma \ref{lem-norm-est3}, which tells us that $\|D\bar{g}\|_{\ell^2(\Omega_1(l))} \lesssim |\omega_l|^{-5/2}$ and so it easy to see that we can only conclude that 
\[
 \sum_{i=1}^4 I_i(\bar{g}) \lesssim |\omega_l|^{-5/2}.
\]
\end{proof}
\begin{lemma}\label{lem-norm-est5}
For any $s$ with $|s|$ large enough and $\tau \in \Rc(s)$, and any $\delta > 0$, the function  $\bar{g}(m,s) := D_{2,\tau}\bar{\G}(m,s)$ satisfies
\[
\|D\bar{g}(\cdot,s)\|_{\ell^2(\Omega_1)} \lesssim |\omega(s)|^{-3+\delta}.	
\]
\end{lemma}
\begin{proof}
The result of Lemma \ref{lem-norm-est4} in particular implies that the terms \eqref{norm-est1} and \eqref{norm-est-bt1} can now be estimated, respectively, by 
\begin{align}\label{norm-est1-new}
\sum_{m\in B_{\frac{3|s|}{4}}(0)}h^{(a)}_4(m,s)|\bar{g}(m,s)| &\lesssim \sum_{m\in B_{\frac{3|s|}{4}}(0)} (1+|\omega_m|^7|\omega_s||\omega_{ms}^-|^2)^{-1}(1+|\omega_s|^{5/2})^{-1}\\
&\lesssim |\omega_s|^{-11/2}\sum_m|\omega_m|^{-7}\lesssim |\omega_s|^{-11/2},\nonumber
\end{align}
\begin{align}\label{norm-est-bt1-new}
\sum_{m\in B_{\frac{3|s|}{4}}(0)\cap\Gamma}h^{(a)}_2(m,s)|\bar{g}(m,s)| &\lesssim \sum_{m\in B_{\frac{3|s|}{4}}(0)\cap\Gamma} (1+|\omega_m|^3|\omega_s||\omega_{ms}^-|^2)^{-1}(1+|\omega_s|)^{-5/2}\\
&\lesssim |\omega_s|^{-11/2}\sum_{{\cb m\in B_{\frac{3|s|}{4}}(0)\cap\Gamma}}|\omega_m|^{-3}\lesssim |\omega_s|^{-11/2},\nonumber
\end{align}
which in particular, repeating the argument in Lemma \ref{lem-norm-est3}, implies that 
\[
\|D\bar{g}(\cdot,s)\|_{\ell^2(\Omega_1)} \lesssim |\omega(s)|^{-11/4}.	
\]
We can thus redo the argument in Lemma \ref{lem-norm-est4} to conclude that for $|m|\leq \frac{3|s|}{4}$ we have $|\bar{g}(m,s)| \lesssim |\omega_s|^{-11/4}$, which in turn implies that
\[
\|D\bar{g}(\cdot,s)\|_{\ell^2(\Omega_1)} \lesssim |\omega(s)|^{-23/4}.	
\]
It is hence apparent that we can repeat this process ad infinitum with the result after $k$ iterations given by
\[
\|D\bar{g}(\cdot,s)\|_{\ell^2(\Omega_1)} \lesssim|\omega_s|^{-d(k)},
\]
where
\[
d(k) = 3\left(\sum_{i=1}^k \frac{1}{2^i} + \frac{b}{2^{k}}\right),
\]
where $b = \frac{2}{3} < 1$, which ensures that $d(k) < 3\;\forall k \in \N$, but with $\lim_{k\to\infty}d(k) = 3$, thus establishing the premise of the lemma.
\end{proof}
\begin{lemma}\label{D1D2outside2}
If $l \in \Omega_1(s)\cup\Omega_2(s)$, $\tau \in \Rc(l)$, and $\lambda \in \Rc(s)$,  then for any $\delta >0$
\[
|D_{1,\tau}D_{2,\lambda}\bar{\G}(l,s)| \lesssim (1+|\omega_l||\omega_s||\omega^-_{ls}|^{2-\delta})^{-1}.
\]
\end{lemma}
\begin{proof}
Consider $s,l \in \La$ such that $|s|$ is large enough and $|l| \leq \frac{|s|}{3}$ and $|l| \geq 4$. We create a cut-off function $\eta$ that scales with $|l|$, namely we say $\eta \equiv 1$ in $B_{\frac{|l|}{4}}(l)$ and $\eta \equiv 0$ outside $B_{\frac{|l|}{2}}(l)$ and smooth and decreasing in-between. Mimicking the approach in Lemma \ref{D1D2annulus}, we conclude that
\[
\Dm{\tau}\Ds{\lambda}\bar{\G}(l,s)  = S_1 + S_2
\]
where
\[
|S_1| \lesssim \sum_{i=1}^{4}J_i(v) \lesssim (1+|\omega_l||\omega_s||\omega^-_{ls}|^{2})^{-1},
\]
due to a verbatim repetition of the argument in Lemma \ref{D1D2annulus}. Likewise, we can immediately conclude that
\[
|S_2| \lesssim |\omega_l|^{-2}\|D_1\bar{g}(\cdot,s)\|_{\ell^2(\mathcal{A}_l)},
\]
where $\bar{g}(m,s) := \Ds{\tau}\bar{\G}(m,s)$ and $\mathcal{A}_l := B_{\frac{|l|}{2}}(l)\setminus B_{\frac{|l|}{4}}(l)$. Crucially, we observe that in the region under consideration we have $B_{\frac{|l|}{2}}(l) \subset \Omega_1(s)$, thus allowing us to employ Lemma \ref{lem-norm-est5} to conclude that
\[
|S_2|\lesssim (1+|\omega_l|^2|\omega_s|^{3-\delta})^{-1} \lesssim (1+|\omega_l|^2|\omega_s||\omega_{ls}^-|^{2-\delta})^{-1},
\]
where the final passage follows from the fact that we have $|\omega_s|\gtrsim |\omega_{ls}^-|$ in the region of interest. 

We further note that this partial result together with the norm estimate over $\Omega_1(s)$ in Lemma \ref{lem-norm-est5} implies that if $|l| \leq 4$, then
\[
|D_{1,\tau}D_{2,\lambda}\bar{\G}(l,s)| \lesssim |\omega_s|^{-3+\delta} \lesssim |\omega_s|^{-1}|\omega_{ls}^-|^{-2+\delta},
\]
which is precisely the result we want for $l \approx 0$. Finally, for $l$ such that $\frac{|s|}{3}\leq |l| \leq \frac{|s|}{2}$, we simply note that the choice of regions $\Omega_1(s), \mathcal{A}(s)$, and $\Omega_2(s)$ at the start of Section \ref{P-G2} was arbitrary in the sense that we can always choose different constants with the $|s|$ scaling and none of the arguments are  affected except for having to readjust the constants for $|s|$ scaling of the cut-off functions used throughout. Thus we have shown that if $l \in \Omega_1(s)$ then 
\[
|\Dm{\tau}\Ds{\lambda}\bar{\G}(l,s)| \lesssim (1+|\omega_l||\omega_s||\omega^-_{ls}|^{2-\delta})^{-1},
\]
which thanks to variable symmetry of $\bar{\G}$ established in the proof of Theorem \ref{thm::G} implies the same for $l \in \Omega_2(s)$. 
\end{proof}
\begin{proof}[Proof Theorem \ref{thm::G}: decay estimate for the mixed derivative of $\G$]
Lemma \ref{D1D2annulus} and Lemma \ref{D1D2outside2} together establish that for all $l,s\in\La$ and $\tau\in\Rc(l)$, $\lambda\in\Rc(s)$, we have for any $\delta >0$
\[
|\Dm{\tau}\Ds{\lambda}\bar{\G}(l,s)| \lesssim  |\omega(l)|^{-1}|\omega(s)|^{-1}|\omega_{ls}|^{-2+\delta}.
\]
Lemma \ref{D1D2hatG} provides the same result for $\hat{\G}$ (including the case $\delta = 0$) and since $\G = \hat{\G} + \bar{\G}$, then in fact   
\[
|\Dm{\tau}\Ds{\lambda}\G(l,s)| \lesssim  |\omega(l)|^{-1}|\omega(s)|^{-1}|\omega_{ls}|^{-2+\delta},
\]
which is what we set out to prove. 
\end{proof}

\bibliographystyle{myalpha} %includes doi
\bibliography{papers}
\end{document}